\documentclass[11pt, psamsfonts, reqno]{amsart}

\usepackage[T1]{fontenc}
\usepackage[latin1]{inputenc}
\usepackage[frenchb, english]{babel}
\usepackage{amsthm}
\usepackage{amscd}
\usepackage{tensor}
\usepackage{enumitem}
\usepackage{amssymb,amsmath,bbm}
\usepackage{yhmath}
\usepackage{tikz-cd}
\usepackage[all]{xy}
\usepackage{mathrsfs}
\usepackage[left=0.0cm,right=0.0cm,top=1.7cm,bottom=1.7cm]{geometry}
\usepackage{hyperref}
\usepackage{cleveref}
\usepackage{xcolor}
\usepackage{relsize}
\usepackage{comment}

\numberwithin{equation}{section}

\theoremstyle{plain}
\newtheorem{theorem}{Theorem}[subsection]

\newtheorem{lemma}[theorem]{Lemma}
\newtheorem{proposition}[theorem]{Proposition}

\newtheorem{corollary}[theorem]{Corollary}

\newtheorem*{theorem*}{Theorem}

\theoremstyle{definition}
\newtheorem{definition}[theorem]{Definition}
\newtheorem{example}[theorem]{Example}

\newtheorem{construction}[theorem]{Construction}

\newtheorem{Convention}[theorem]{Convention}

\theoremstyle{remark}
\newtheorem{remark}[theorem]{Remark}

\usepackage{tikz}
\usetikzlibrary{calc}

\def\Q{{\bb Q}}

\def\Z{{\bb Z}}

\def\N{{\bb N}}

\def\Hom{\mathrm{Hom}}

\def\Sym{\mathrm{Sym}}
\def\Hom{\mathrm{Hom}}
\def\iHom{R\underline{\mathrm{Hom}}}

\def\epsilon{\varepsilon}

\def\et{{\rm \acute{e}t}}
\def\Sp{{\mathbf{Sp}}}
\def\GL{\mathbf{GL}}

\def\Cond{\mathrm{Cond}}
\def\cond{\mathrm{cond}}

\def\iHom{\underline{\mathrm{Hom}}}

\def\sol{\mathsmaller{\square}}

\DeclareMathOperator{\Mod}{Mod}

\DeclareMathOperator{\Spa}{Spa}
\DeclareMathOperator{\Spec}{Spec}

\DeclareMathOperator{\id}{id}

\DeclareMathOperator{\Betti}{Betti}

\DeclareMathOperator{\Map}{Map}

\DeclareMathOperator{\dR}{dR}

\DeclareMathOperator{\AnSpec}{AnSpec}

\DeclareMathOperator{\Corr}{Corr}

\DeclareMathOperator{\HT}{HT}

\DeclareMathOperator{\pr}{pr}

\newcommand{\n}[1]{\mathcal{#1}}
\newcommand{\ob}[1]{\mathrm{#1}}
\newcommand{\cat}[1]{\mathsf{#1}}
\newcommand{\bb}[1]{\mathbb{#1}}

\newcommand{\s}[1]{\mathscr{#1}}

\DeclareMathOperator{\op}{\begin{scriptsize}
op
\end{scriptsize}}

\DeclareMathOperator{\an}{\begin{scriptsize}
an
\end{scriptsize}}

\title{Cartier duality for gerbes of vector bundles}

\author{Juan Esteban Rodr\'iguez Camargo}

\makeatletter
\def\@tocline#1#2#3#4#5#6#7{\relax
  \ifnum #1>\c@tocdepth % then omit
  \else
    \par \addpenalty\@secpenalty\addvspace{#2}%
    \begingroup \hyphenpenalty\@M
    \@ifempty{#4}{%
      \@tempdima\csname r@tocindent\number#1\endcsname\relax
    }{%
      \@tempdima#4\relax
    }%
    \parindent\z@ \leftskip#3\relax \advance\leftskip\@tempdima\relax
    \rightskip\@pnumwidth plus4em \parfillskip-\@pnumwidth
    #5\leavevmode\hskip-\@tempdima
      \ifcase #1
       \or\or \hskip 1em \or \hskip 2em \else \hskip 3em \fi%
      #6\nobreak\relax
    \dotfill\hbox to\@pnumwidth{\@tocpagenum{#7}}\par
    \nobreak
    \endgroup
  \fi}
\makeatother
\usepackage{hyperref}

\begin{document}

\begin{abstract}
We prove a Cartier duality for gerbes of algebraic and analytic vector bundles as an anti-equivalence of Hopf algebras in the category of kernels of analytic stacks. As an application, we prove that the category of solid quasi-coherent sheaves on the  Hodge-Tate stack of a smooth rigid variety over an algebraically closed field $C$ of mixed characteristic $(0,p)$ is equivalent to the category of weight $1$ sheaves on Bhatt-Zhang's Simpson gerbe. 
\end{abstract}

\maketitle
\tableofcontents

\section{Introduction}

\subsection{The motivation behind}

This paper has two main motivations. First, it is more natural to isolate the Cartier duality   of analytic   vector bundles discussed in \cite{camargo2024analytic} in a different document as these results could be of independent interest. Moreover,  from the construction of the Cartier duality in \cite{camargo2024analytic}, it is not clear in which sense Cartier duality is functorial, and whether it  exchanges tensor products with convolutions (as it should be  the case). In this paper we consider these questions more seriously, and provide a conceptual answer  using the category of kernels of a six functor formalism. The first main goal of the paper is to establish a  stacky formulation of the Cartier duality for vector bundles in \Cref{TheoCartierDualityVectorBundles}.

Second, in   joint work in progress with Ansch\"utz, Le Bras and Scholze on the analytic prismatization \cite{AnPrismatization},  we introduce the \textit{analytic Hodge-Tate stack} $X^{\ob{HT}}$ of a smooth rigid space $X$, an object whose category of perfect complexes is naturally equivalent to the category of perfect $\widehat{\s{O}}_X$-modules in the $v$-topology of $X$. Almost in parallel, Bhatt and Zhang introduced the \textit{Simpson gerbe} $\s{S}_X$ \cite{BhattpAdicHodge2026}, that is, a $B\bb{G}_m$-torsor over the (Tate-twisted)  analytic cotangent bundle $T^{*,\an}_X(-1)$ of $X$. It is expected that the Hodge-Tate stack and the Simpson gerbe are Cartier dual to each other in a concrete sense. Indeed, if $X$ is a rigid space over a perfectoid field $K$, the  Hodge-Tate stack $X^{\HT}$ is a $BT_X^{\dagger}(1)$-torsor over $X$ (when considered as analytic stacks), with $T_X^{\dagger}\subset T_X^{\an}$ the overconvergent neighbourhood at $0$, and the Cartier duality of vector bundles produces an equivalence of categories $\ob{D}(BT_X^{\dagger})\cong \ob{D}(T_X^{*,\an})$. Thus, the Cartier duality between the Simpson gerbe and the Hodge-Tate stack ought to be a twisted version of the Cartier duality of vector bundles, producing in particular a natural equivalence 
\[
\ob{D}(X^{\HT})\cong \ob{D}(\s{S}_X)^{\ob{wt}=1}
\]
where the right hand side consists of sheaves of weight $1$ on the Simpson gerbe.  The final result comparing the Hodge-Tate stack and the Simpson gerbe is \Cref{CartierDualityHodgeTateSimpson}.

 It turns out that the Cartier duality  for the Simpson gerbe is not special, and  it holds \textit{universally} for arbitrary gerbes of (analytic or algebraic) vector bundles. The second main goal of this paper is to establish such universal Cartier duality for gerbes  in \Cref{PropDualityForGerbes}. For that, we need to develop some technical tools that allow us to compute and produce Cartier dualities from six functor formalisms. This is the content of   \Cref{s:Preliminaries} which concludes with  (the rather technical but useful) \Cref{CoroCartierDualityAnStkQAff}.

\subsection{Content of the paper}

One of the most elementary instances of Cartier duality  can be stated in terms of  commutative and cocommutative Hopf algebras over a field $K$, or equivalently, in terms of affine commutative group schemes over $K$. Let $\cat{Aff}_K$ be the category of affine schemes over $K$, and let $G$ be an affine commutative group scheme over $\Spec K$ with underlying Hopf algebra $A$. The Cartier dual of $G$ is the functor $\bb{D}(G)\colon \cat{Aff}_K^{\op}\to \cat{Mod}(\Z)$  sending an affine scheme $X$ to the module of group homomorphisms $\Hom_{\Z}(G_X, \bb{G}_{m,K})$. In the case $G$ is a finite flat group scheme, that is, $A$ is a finite dimensional $K$-vector space, the functor $\bb{D}(G)$ is corepresented by the Hopf algebra $A^{\vee}=\Hom_K(A,K)$.  In this way, Cartier duality can be understood  more algebraically as the equivalence of categories 
\begin{equation}\label{eqo0jonalmasifq3eq}
\ob{Hopf}(\ob{Vect}_K^{\ob{fd}})^{\op} \xrightarrow{(-)^{\vee}} \ob{Hopf}(\ob{Vect}_K^{\ob{fd}})
\end{equation}
between (commutative and cocommutative) Hopf algebras on finite dimensional vector spaces obtained by passing to the $K$-linear dual object, and exchanging the multiplication and comultiplication maps.  Both interpretations (the geometric one as the dual $\iHom_{\Z}(-,\bb{G}_m)$ and the algebraic as in \eqref{eqo0jonalmasifq3eq}) are useful in practice and lead to different properties of Cartier duality.  The geometric interpretation implies that  a $\Z$-bilinear pairing $H\times G\to \bb{G}_m$  yields a natural map of groups $H\to \bb{D}(G)$. The algebraic interpretation  implies the anti-involutive property of Cartier duality.  Unfortunately, the formulation of Cartier duality of  \eqref{eqo0jonalmasifq3eq} is rather restrictive.

 Experience has told us that there are many more examples of Cartier duality in many other different contexts. For instance, still in algebraic geometry, one has a naive duality between $\bb{Z}$ and $\bb{G}_m$ in the sense that the group algebra of $\bb{Z}$ over $K$ is precisely the ring of functions of $\s{O}(\bb{G}_m)$. Similarly, if $K$ is of characteristic zero, the dual of the natural comultiplication map $K[T]\to K[X,Y]$ induced by the additive law $T\mapsto X+Y$ is the multiplication map of the power series ring over $K$, yielding a sort of duality between $\bb{G}_{a,K}$ and $\widehat{\bb{G}}_{a,K}$. A way to relate these examples with the Cartier duality  for finite dimensional $K$-algebras, is to work one  categorical level up and consider (derived) categories of quasi-coherent sheaves.  In the previous examples, the Cartier duality is justified by  the classical equivalences 
\begin{equation}\label{eqoojmkwepqnwo13eqwd}
\ob{D}(B\bb{Z})\cong \ob{D}(\bb{G}_m) \mbox{ and } \ob{D}(B\widehat{\bb{G}}_{a,K})\cong \ob{D}(\bb{G}_{a,K})
\end{equation}
that one can prove by hand. With some additional effort one can see that these equivalences exchange tensor products with the convolution products arising from the group action, suggesting that they should promote to an equivalence  of  (commutative and cocommutative) Hopf algebras in a precise sense. It turns out that the categories above can be promoted to honest Hopf algebra objects in the symmetric monoidal category $\cat{Pr}_{\ob{D}(\Z)}$ of $\ob{D}(\Z)$-linear presentable categories, and what is even better, that they are dualizable objects therein. The abstract framework of Cartier duality as  an antiequivalence of (dualizable) Hopf algebra objects in a symmetric monoidal category $\n{S}$ has been carried out by Lurie in \cite{LurieElliptic}.  Lurie  offers not only the algebraic version of \eqref{eqo0jonalmasifq3eq}, but also explains how to think of Cartier duality in more geometric terms by mapping to a suitable variant of $\bb{G}_m$ that is  denoted $\mathbf{GL}_{1,\n{S}}$. This object represents invertible elements in commutative algebras $\n{V}$ in $\n{S}$, that is, maps $x\colon 1_{\n{S}}\to \n{V}$ for which there some $y\colon 1_{\n{S}}\to \n{V}$ whose multiplication $x y\colon 1_{\n{S}}\to \n{V}$ is homotopic to the identity. Thus, if we write $\Spec \n{V}$ for an object in the opposite category $\cat{Aff}_{\ob{D}(\Z)}$ of $\ob{CAlg}(\cat{Pr}_{\ob{D}(\Z)})$, the equivalences \eqref{eqoojmkwepqnwo13eqwd} can be stated as isomorphisms
\[
\iHom( \Spec \ob{D}(B\bb{Z}), \mathbf{GL}_{1,\ob{D}(\Z)}) = \Spec \ob{D}(\bb{G}_m) \mbox{ and }  \iHom( \Spec \ob{D}(B\widehat{\bb{G}}_{a,K}), \mathbf{GL}_{1,\ob{D}(\Z)})  = \Spec \ob{D}(\bb{G}_{a,K})
\]
in presheaves on $\cat{Aff}_{\ob{D}(\Z)}$ valued in connective spectra. %As an observation, the object $\mathbf{GL}_{1,\ob{D}(\Z)}$  represents invertible objects in $\ob{D}(\Z)$-linear symmetric monoidal categories, in particular, there is a natural morphism $\Spec \ob{D}(B\bb{G}_m)\to \mathbf{GL}_{1,\ob{D}(\Z)}$ arising from the universal line bundle of $B\bb{G}_m$. This map is a morphism of connective spectra,  and the Cartier duality 

Consequently, working with linear presentable categories over a fixed  symmetric monoidal stable category $\n{V}$ produces a more general framework for Cartier duality that captures many more examples. However, there are  three main problems of working only in this setup.  First, some important examples  would get excluded, most notably those arising from \'etale sheaves of vector bundles \cite{PMIHES_1987__65__131_0} or Banach-Colmez spaces \cite{zbMATH08075692}.\footnote{We shall not discuss these examples of Cartier duality in this paper, so this is not a serious problem for us.}  Second, it is unclear how to state an honest Cartier duality for more stacky objects, for instance, for  vector bundles on  an arbitrary algebraic or analytic stacks. Third, it is not clear how to obtain the appropriate Hopf algebra structure in the category of modules only from the group structure of, say, a  commutative group scheme $G$. It is even less clear how to describe the Cartier dual of $\ob{D}(G)$ only from the datum of $G$.

 In order to solve the previous problems we work in a more general setup than just linear presentable categories, namely, we work with presentable kernel categories of six functor formalisms. Given a geometric setup $(\n{C},E)$ and $\ob{D}$ a presentable six functor formalism on $(\n{C},E)$ (cf. \cite{HeyerMannSix}), for an object $S\in \n{C}$ one has a \textit{$2$-category of kernels} $\ob{K}_{\ob{D},S}$ whose objects are the objects in $\n{C}^E_{/S}$ consisting on maps $X\to S$ in $E$, and for $Y,X\in \ob{K}_{\ob{D},S}$ the $1$-category of morphisms  given by $\ob{Fun}_S(X,Y) = \ob{D}(Y\times_S X)$. It turns out that $\ob{K}_{\ob{D},S}$  is naturally enriched in presentable categories, passing to enriched presheaves one constructs a \textit{presentable category of kernels} $\cat{Pr}_{\ob{D},S}$ (e.g. as in \cite[Appendix to Lecture V]{SixFunctorsScholze}). This is  a symmetric monoidal $2$-category where one can apply Lurie's theory of Cartier duality. The advantage of working in this framework is that we have at our disposal a symmetric monoidal functor 
 \[
 \ob{Corr}(\n{C}^E_{/S})\to \cat{Pr}_{\ob{D},S}
 \]
from the category of correspondences to the presentable category of kernels. This will allow us to produce Hopf algebras in $\cat{Pr}_{\ob{D},S}$ from commutative group objects in $\n{C}^E_{/S}$. Furthermore, since the objects of $ \ob{Corr}(\n{C}^E_{/S})$ are naturally self dual, and this duality exchanges $*$ and $!$-maps, it is  tautological that Cartier duality is the identity on objects in $\ob{K}_{\ob{D},S}$, and that it will exchange tensor products  (obtained from $*$-pullbacks) with convolution products (obtained from $!$-pushforwards).

\begin{remark}\label{RemarkScholzeStefanich}
A general framework of Cartier duality capturing all existing examples has been missing for years. In the current  work in progress of Peter Scholze and Germ\'an  Stefanich on \textit{Gestalten} (\cite{GestaltenScholze}) they found a clean Cartier duality statement that is an honest anti-equivalence of stable categories. In very heuristic terms, the  key idea is that  a general notion of geometry ought to be captured not just by rings or categories of modules as it occurs in usual algebraic and analytic geometry, but by all the higher categories of modules. In the situation of Cartier duality, this is reflected in the fact that there are instances where working only with algebras does not suffice to obtain a satisfactory duality, and working with categories of modules is necessary to  obtain dualizable Hopf algebras; Scholze and Stefanich's theory takes this idea to the very extreme and shows that to obtain a \textit{perfect Cartier duality theory}  one has to go all the way up. 

The perspective on Cartier duality of this paper is highly inspired from their theory  and it  uses many of their ideas, most notably we use  the  $1$-\'etale topology on linear categories in \Cref{ss:DescentCartierDuality} which is nothing but a $1$-categorical approximation to the natural Grothendieck topology on Gestalten. The main advantage of this point of view is that it allows us to compute Cartier duals by descent and/or devisage, notably reducing some complicated Cartier dualities of stacky objects to much simpler Cartier dualities of  quite concrete objects. 
\end{remark}

  After developing the tools for discussing Cartier duality in six functor formalisms, we focus our attention on  Cartier duality for different incarnations of vector bundles.   For concreteness, let us discuss only the analytic variant over the ring of $p$-adic rational numbers with the induced   solid structure  $\Q_{p,\sol}$. 

\begin{theorem}[\Cref{TheoCartierDualityVectorBundles}]\label{TheoIntroCartierVB}
Let $\mathbf{Vect}^{\an}/\AnSpec \bb{Q}_{p,\sol}$ be the stack of analytic vector bundles over $\AnSpec \bb{Q}_{p,\sol}$, equivalently, $\mathbf{Vect}^{\an}=\bigsqcup_{n\in \N} \GL_{n,\Q_p}^{\an}$ where $\GL_{n,\Q_p}^{\an}$ is the rigid analytic group over $\Q_p$  of invertible $n\times n$-matrices. Let $V^{\an}/\mathbf{Vect}^{\an}$ be the universal vector bundle and let $V^{*,\an}$ be its dual. For an analytic vector bundle $W^{\an}$ we let $W^{\dagger}\subset W^{\an}$ denote its overconvergent neighbourhood at $0$. Consider the exponential pairing 
\[
\exp(YX)\colon V^{\an}\times  V^{*,\dagger}\to \bb{G}_{a}^{\dagger}\to \bb{G}_m
\]
where $\bb{G}_{a}^{\dagger}\subset \bb{G}_{a,\Q_p}^{\an}$ is the overconvergent neighbourhood at $0$, and  $\exp\colon \bb{G}_{a}^{\dagger}\to \bb{G}_m$ is the exponential map. The exponential map gives rise to  equivalences $V^{\an}\cong BV^{*,\dagger}$ and $BV^{\an}\cong V^{*,\dagger}$  in the category of kernels of $\mathbf{Vect}^{\an}$ exchanging the usual tensor product obtained by $*$-functors and the convolution product obtained by $!$-functors. In particular, for any analytic stack $X$ over $\AnSpec \Q_{p,\sol}$ and any analytic vector bundle $W^{\an}/X$  we have an equivalence of categories
\begin{equation}\label{eqoo1jolqmwpqenfq}
\ob{D}(W^{\an})\cong \ob{D}(BW^{*,\dagger}) \mbox{ and } \ob{D}(BW^{\an})\cong \ob{D}(W^{*,\dagger}) 
\end{equation}
exchanging tensor products with convolution products.  The equivalences  \eqref{eqoo1jolqmwpqenfq} arise via the Fourier-Mukai kernel obtained  by the pullback of the universal line bundle of $B\bb{G}_m$ along the maps $W^{\an}\times BW^{*,\dagger}\to B\bb{G}_m$ and $BW^{\an}\times W^{*,\dagger}\to B\bb{G}_m$ induced by taking suspensions of the exponential. 
\end{theorem}

The Cartier duality for vector bundles can  be  extended to gerbes after working stacky enough, see \Cref{PropDualityForGerbes} for the general result we prove in this direction. As an application of the stacky Cartier duality for gerbes we obtain the Cartier duality between the Hodge-Tate stack and the Simpson gerbe (we refer to \Cref{ss:CartierDualityHTSImp} for a brief introduction to these objects):

\begin{theorem}[\Cref{CartierDualityHodgeTateSimpson}]\label{TheoIntroHTSimpson}
Let $X$ be a smooth rigid variety  over an algebraically closed complete non-archimedean extension $C$ of $\Q_p$. Let $X^{\HT}$ be its analytic Hodge-Tate stack and let $\s{S}_X\to T^{*,\an}_X(-1)$  be Bhatt and Zhang's Simpson gerbe (where the $(-1)$ refers to the inverse of the Tate twist). The following holds:
\begin{enumerate}

\item The category of solid quasi-coherent sheaves $\ob{D}(\s{S}_X)$ admits a natural  $\ob{D}(T^{*,\an}_X(-1))$-linear decomposition 
\[
\ob{D}(\s{S}_X)=\prod_{n\in \Z} \ob{D}(\s{S}_X)^{\ob{wt}=n}
\]
as a product of $\ob{D}(T^{*,\an}_X(-1))$-invertible categories. 

\item The decomposition of (1) is multiplicative, that is, for $n,m\in \Z$ there is a natural equivalence 
\[
\ob{D}(\s{S}_X)^{\ob{wt}=n}\otimes_{\ob{D}(T^{*,\an}_X(-1))} \ob{D}(\s{S}_X)^{\ob{wt}=m}=\ob{D}(\s{S}_X)^{\ob{wt}=n+m}
\]
and $\ob{D}(\s{S}_X)^{\ob{wt}=0}= \ob{D}(T^{*,\an}_X(-1))$ is the full subcategory of $\ob{D}(\s{S}_X)$ obtained as the essential image of the pullback along $\s{S}_X\to T^{*,\an}_X(-1)$.

\item There is a natural $\ob{D}(T^{*,\an}_X(-1))$-linear equivalence of categories 
\[
\ob{D}(X^{\HT})\cong \ob{D}(\s{S}_X)^{\ob{wt}=1},
\]
where $\ob{D}(T^{*,\an}_X(-1))$ acts on the left term via Cartier duality $\ob{D}(T^{*,\an}_X(-1))\cong \ob{D}(BT^{\dagger}_X(1))$ and by the $!$-convolution action  arising from  the natural $BT^{\dagger}_X(1)$-torsor structure of $X^{\HT}\to X$.  

\end{enumerate}
\end{theorem}

\subsection{Overview}

This paper is divided in two main sections: a Preliminary \Cref{s:Preliminaries} and an Example \Cref{s:Examples}. 

 In \Cref{ss:PresentableKernels} we introduce the presentable category of kernels. In \Cref{ss:AnStk} we introduce the setup of analytic stacks that we shall use. Next, in \Cref{ss:CartierDualitySix} we recall the abstract set up of Cartier duality of  \cite{LurieElliptic}.   In \Cref{ss:CartierDualitySixFunctors} we specialize  Cartier duality to six functor formalisms and the presentable category of kernels. In particular, we see that  Cartier duality exchanges tensor product with convolution. Finally, in \Cref{ss:DescentCartierDuality} we restrict ourselves to study Cartier duality in presentable linear categories. Motivated from the theory of Gestalten of Scholze and Stefanich,  we introduce the $1$-\'etale topology in presentable commutative algebras over a presentable symmetric monoidal category $\n{V}$ and refine  Lurie's Cartier duality from presheaves to sheaves for that topology. We conclude with  a technical but very useful result that  simplifies some computations in Cartier duality, see \Cref{CoroCartierDualityAnStkQAff}.

We continue with \Cref{ss:CartierTori} where we state  a Cartier duality between tori and finite free $\Z$-modules, see \Cref{TheoCartoerDualityTori}. In \Cref{ss:CartierDualityVectorBundles} we discuss several examples of Cartier dualities for vector bundles. The strategy is always the same; first to prove a basic Cartier duality for the trivial rank $1$-case (\Cref{PropBasicAlgVectCD,PropBasicSolidVectCD,PropBasicDiscVectCD,PropBasicAnVectCD,PropBasicZpVectCD})
 where we heavily use \Cref{CoroCartierDualityAnStkQAff}. Then, using descent techniques of the category of kernels, we easily deduce the general stacky version of Cartier duality as in \Cref{TheoCartierDualityVectorBundles}. In \Cref{ss:CDGerbesVB} we discuss a formal consequence of Cartier duality for tori and vector bundles obtaining Cartier duality for gerbes, see \Cref{PropDualityForGerbes} and \Cref{ExampleCartierDualityGerbes}. Finally, in \Cref{ss:CartierDualityHTSImp} we specialize the Cartier duality for gerbes to the Hodge-Tate stack obtaining the comparison with the Simpson gerbe in \Cref{CartierDualityHodgeTateSimpson}.

\subsection{Conventions}

This paper uses the language of higher category theory and higher algebra as developed in \cite{HigherTopos,HigherAlgebra,LurieSpectralAlg,kerodon}. Our main working tool is the theory of abstract six functor formalisms as in \cite{HeyerMannSix}, see also \cite{SixFunctorsScholze}. We will work with a presentable variant of the category of kernels of a six functor formalism as discussed in \cite[Appendix to lecture V]{SixFunctorsScholze}, for that, we will use the theory of higher presentable categories of \cite{aoki2025higher}. In particular, we let $\cat{Pr}$ be the \textit{big} category of presentable categories with colimit preserving linear functors, and given $\kappa$ a regular cardinal we let $\cat{Pr}^{\kappa}$ the (presentable) category of $\kappa$-presentable categories, that is the (non-full) subcategory $\cat{Pr}^{\kappa}\subset \cat{Pr}$ of $\kappa$-compactly generated presentable categories with maps given by functors of presentable categories that preserve $\kappa$-compact objects. Given a presentable category $\n{C}$, we let $\n{C}_{\kappa}$ denote the full subcategory of $\kappa$-compact objects.

 Occasionally we will need to work with enriched categories, for this we refer to   \cite{GepnerHaugsengEnriched},  \cite{HinichEnriched}, \cite{HinichColimitsEnriched} and \cite{HeineEnriched}. We also refer to \cite[Appendix C]{HeyerMannSix} for a great summary of the theory of enriched categories. In particular, given $\n{V}$ a presentable monoidal category, we let $\cat{Cat}^{\n{V}}$ be the  category of \textit{essentially small} $\n{V}$-enriched categories, if $\n{V}$ is symmetric monoidal then $\cat{Cat}^{\n{V}}$ is naturally symmetric monoidal by \cite{HinichEnriched}.

In some of the applications we will work in the category of analytic stacks of Clausen and Scholze as introduced in \cite{AnalyticStacks}, see also \cite{SolidNotes} and \cite{anschutz2025analytic} for written  references. In particular, we will always work in the light set up of condensed mathematics.

\subsection*{Acknowledgements} I would like to thank Johannes Ansch\"utz,  Arthur-C\'esar Le Bras and Peter Scholze for several discussions in our project in progress on the \textit{analytic prismatization} that lead to the question of how to concretely state the Cartier duality between the Hodge-Tate stack and Bhatt-Zhang's Simpson gerbe.    I also  like to thank Arthur-C\'esar Le Bras for helpful  comments in a draft. I heartily thank   Ko Aoki, Shay Ben-Moshe,  and  Germ\'an Stefanich for many technical discussions regarding category theory and six functor formalisms. Special thanks to Peter Scholze and Germ\'an Stefanich for explaining  their new beautiful theory of Cartier duality in Gestalten, this paper is highly inspired in their theory and actually uses a very rough version of it in \Cref{ss:DescentCartierDuality}.  I thank the  Max Planck Institute  for Mathematics for the excellent working conditions.

\section{Preliminaries}\label{s:Preliminaries}

In this section we recall the construction of the presentable category of kernels of a six functor formalism \cite[Appendix to lecture V]{SixFunctorsScholze}, and the general framework of Cartier duality on symmetric monoidal categories of \cite[Section 3]{LurieElliptic}.  Then, we discuss in some more detail the Cartier duality in the category of kernels of a six functor formalism and analytic stacks.

\subsection{Presentable category of kernels}\label{ss:PresentableKernels}

In this section we introduce the key character that will notably simplify the theory of Cartier duality over stacks, that is, the presentable category of kernels.

 Let $(\n{C},E)$ be a small geometric set up with $\n{C}$ admitting finite limits, let $\n{C}^E$ be the wide non-full subcategory of $\n{C}$ spanned by the arrows in $E$.  We let $\Corr(\n{C},E)$  denote the category of correspondences of $(\n{C},E)$  (cf. \cite[Definition 2.2.3]{HeyerMannSix}), if $E=\ob{all}$ consists of all morphisms we simply write $\Corr(\n{C}):=\Corr(\n{C},\ob{all})$. Let $\ob{D}$ be a presentable  six functor formalism on $(\n{C},E)$  (cf. \cite[Definition 3.1.1]{HeyerMannSix}), by definition  $\ob{D}$ is a lax symmetric monoidal functor $\ob{D}\colon \Corr(\n{C},E)\to \cat{Pr}$. % thanks to \cite[Example 2.2.6.9]{HigherAlgebra} we can see $\ob{D}$ as a commutative algebra object in $\ob{Fun}(\Corr(\n{C},E), \cat{Pr})$ with respect to Day's convolution. 

Since $\n{C}$ is   small, there is some regular cardinal $\kappa$ such that $\ob{D}$ factors through the (non-full) subcategory $\cat{Pr}^{\kappa}$ of $\cat{Pr}$ of $\kappa$-presentable categories. We say that $\ob{D}$ is a  \textit{$\kappa$-presentable six functor formalism} if this holds. From now on we shall fix an uncountable regular cardinal $\kappa$ unless otherwise specified (in practice $\kappa=\aleph_1$ will suffice). 

 We recall the definition of the category of kernels \cite[Definition 4.1.3]{HeyerMannSix}.

\begin{definition}\label{DefKernelCategory}
Let $(\n{C},E)$ be a geometric set up with finite limits and $\ob{D}$ a $\kappa$-presentable six functor formalism on $(\n{C},E)$.
Given $S\in \n{C}$, consider the restriction $\ob{D}_S$ to a six functor formalism on the geometric setup $(\n{C}^E_{/S},\ob{all})$. 
The \textit{category of kernels of $S$}, denoted by $\ob{K}_{\ob{D},S}$, is the $\cat{Pr}^{\kappa}$-enriched category obtained by transfer of enrichement of $\ob{Corr}(\n{C}^{E}_{/S})$ along the lax symmetric monoidal functor $\ob{D}_{S}\colon \ob{Corr}(\n{C}^{E}_{/S})\to \cat{Pr}^{\kappa}$. 
Given two objects $Y,X\in \ob{K}_{\ob{D},S}$, we let $\ob{Fun}_{\ob{D},S}(Y,X)=\ob{D}(X\times_S Y)$ be the $\kappa$-presentable category of morphisms from $Y\to X$. If $\ob{D}$ is clear from the context we also denote $\ob{Fun}_{S}$ instead of $\ob{Fun}_{\ob{D},S}$. 
\end{definition}

\begin{remark}\label{RemarkFunctorsToCorrespondence}
Keep the notation of \Cref{DefKernelCategory}.  By construction, the objects of $\ob{K}_{\ob{D},S}$ are the same as objects of $\n{C}^E_{/S}$. We have a diagram of symmetric monoidal maps  (see the discussion before Remark 4.1.6 of \cite{HeyerMannSix})
\[
\begin{tikzcd}
(\n{C}^E_{/S})^{\op,\sqcup} \ar[rd, "(-)^*"] & & \\ 
&  \Corr(\n{C}^E_{/S}) \ar[r,"\Phi_{\ob{D}}"] & \ob{K}_{\ob{D},S}, \\ 
(\n{C}^{E}_{/S})^{\times} \ar[ru,"(-)_!"'] & &
\end{tikzcd}
\]
 We note that \textit{loc. cit.} only mentions that the two diagonal arrows are symmetric monoidal. To see that $\ob{K}_{\ob{D},S}$ has a natural symmetric monoidal structure and that $\Phi_D$ is symmetric monoidal it is more convenient to use the \textit{presentable category of kernels} (to be defined in \Cref{DefPresentableCategoryKernels}); this is justified in \Cref{RemarkPresentableCategoryKernlesComparison}.

We call the upper and lower diagonal maps the \textit{$*$ and $!$-realization} of $\n{C}^E_{/S}$ respectively. By construction, both the $*$ and $!$-realization are the identity on objects.  Given $f\colon Y\to X$  a map in $\n{C}^E_{/S}$, we let $f^*\colon X\to Y$ and $f_!\colon Y\to X$ be the image along the $*$ and $!$-realization respectively. This name is justified because after taking global sections, i.e. maps from the unit $S\in \ob{K}_{\ob{D},S}$, the map $f^*\colon X\to Y$ gives rise to the $*$-pullback map 
\[
f^*\colon \ob{D}(X)=\ob{Fun}_S(S,X)\to \ob{Fun}_S(S,Y)=\ob{D}(Y).
\]
Similarly, the map $f_!\colon Y\to X$ gives rise the $!$-pushforward $f_!\colon \ob{D}(Y)\to \ob{D}(X)$.   The self duality in the category of correspondences \cite[Corollary 2.4.2]{HeyerMannSix} swaps the two diagonal arrows, hence $f^*\colon X\to Y$ is the dual of $f_!\colon Y\to X$ in $\ob{K}_{\ob{D},S}$. 
\end{remark}

 In order to construct the ($\kappa$-)presentable category of kernels of $\ob{D}$, we recall the following theorem.

\begin{theorem}[{\cite[Theorem 4.2.4]{HeyerMannSix}}]\label{TheoremFunctorialityKernelCategory}
Let $(\n{C},E)$ be an  small geometric set up with finite limits, $\kappa$ an uncountable regular cardinal and $\ob{D}$ a $\kappa$-presentable  six functor formalism on $(\n{C},E)$.   The formation $X\mapsto \ob{K}_{\ob{D},X}$ promotes to a lax symmetric monoidal functor 
\begin{equation}\label{EqPresentableKernelCat}
\ob{K}_{\ob{D},-}\colon \Corr(\n{C},E)\to \cat{Cat}^{\cat{Pr}^{\kappa}}
\end{equation}
where $\cat{Cat}^{\cat{Pr}^{\kappa}}$ is the symmetric monoidal $2$-category of $\cat{Pr}^{\kappa}$-enriched  small categories. 
\end{theorem}

To pass from the category of kernels to its presentable version we need to recall the enriched Yoneda embedding\footnote{We thank Ben-Moshe for the reference to his paper.}, we follow  \cite{ShayNaturalityYoneda}. Let $\n{V}$ be a presentable monoidal category, and let $\cat{Cat}^{\n{V}}$ be the (big)  category of  $\n{V}$-enriched small categories. Let $\cat{Pr}_{\n{V}}=\ob{LMod}_{\n{V}}(\cat{Pr})$ be the category of presentable left $\n{V}$-linear categories and let  $(\cat{Pr}_{\n{V}})^{L}$ be the (non-full) subcategory spanned by left adjoint functors in $\cat{Pr}_{\n{V}}$. 
Given $M\in \cat{Cat}^{\n{V}}$ one constructs a presentable category of $\n{V}$-enriched presheaves $\n{P}^{\n{V}}(M):=\ob{Fun}^{\n{V}^{\ob{rev}}}(M^{\op}, \n{V}^{\ob{rev}})$ together with a $\n{V}$-enriched Yoneda embedding $M\to \n{P}^{\n{V}}(M)$ (see \cite{HinichEnriched}). The formation $M\mapsto \n{P}^{\n{V}}(M)$ is natural in $M$ and extends to a functor 
\[
\n{P}^{\n{V}}\colon \cat{Cat}^{\n{V}}\to \cat{Pr}_{\n{V}}.
\]
By \cite[Theorem 5.20]{ShayNaturalityYoneda} the functor $\n{P}^{\n{V}}$ factors through $(\ob{Pr}_{\n{V}})^L$ and one has an adjunction of categories 
\[
\n{P}^{\n{V}}\colon \cat{Cat}^{\n{V}} \rightleftharpoons (\cat{Pr}_{\n{V}})^L\colon(-)^{\ob{at}} 
\]
where $(-)^{\ob{at}}$ is the functor that sends a presentable $\n{V}$-linear category $\n{M}$ to its full subcategory $\n{M}^{\ob{at}}$ of atomic objects\footnote{An object in  $\n{M}$ is called \textit{atomic} if the induced pullback map  $\n{V}\to \n{M}$ lies in $(\cat{Pr}_{\n{V}})^L$, i.e. it is a left adjoint map in $\n{V}$-linear categories.}. In particular, the category of $\n{V}$-enriched presheaves of $M$ is the free $\n{V}$-linear category  making the objects of $M$ $\n{V}$-atomic. 
If in addition $\n{V}$ is symmetric monoidal, both $\cat{Cat}^{\n{V}}$ and $\cat{Pr}_{\n{V}}$ are symmetric monoidal, and the functor $\n{P}^{\n{V}}\colon \cat{Cat}^{\n{V}}\to \cat{Pr}_{\n{V}}$ is symmetric monoidal \cite[Corollary E]{reutter2025enriched}.

\begin{lemma}\label{LemmaAtomicPrkappa}
Let $\kappa$ be an uncountable regular cardinal and let $\n{V}$ be a $\kappa$-presentable monoidal category. Then the functor $\n{P}^{\n{V}}\colon \cat{Cat}^{\n{V}}\to \cat{Pr}_{\n{V}}$ factors through $\cat{Pr}^{\kappa}_{\n{V}}:=\ob{LMod}_{\n{V}}(\cat{Pr}^{\kappa})$. 
\end{lemma}
\begin{proof}
Let us recall the definition of an enriched category and of the enriched category of sheaves from \cite[Proposition 4.5.3]{HinichEnriched}. 

Let $M$ be a small $\n{V}$-enriched category with anima of objects $X=M^{\simeq}$. By definition   $M$ is an algebra $\s{C}\in \ob{Alg}(\ob{Fun}(X\times X, \n{V}))$ where $\ob{Fun}(X\times X, \n{V})$ is the endomorphism category of the  $\n{V}$-linear category $\ob{Fun}(X, \n{V})\in \cat{Pr}^{\kappa}_{\n{V}}$.  We have that  
\[
\ob{Fun}^{\n{V}}(M,\n{V})=\ob{LMod}_{\s{C}}(\ob{Fun}(X, \n{V}))
\]
seen as a right $\n{V}$-linear module, or equivalently, as an object in $\cat{Pr}_{\n{V}^{\op}}$. The $\n{V}$-enriched presheaf category is by definition 
\[
\n{P}^{\n{V}}(M):=\ob{Fun}^{\n{V}^{\ob{rev}}}(M^{\op},\n{V}^{\ob{rev}})
\]
which is endowed with a right $\n{V}^{\ob{rev}}$-linear structure, or equivalently, with a left $\n{V}$-linear  structure, i.e. an object in $\cat{Pr}_{\n{V}}$. Thus, to prove $\kappa$-presentability of $\n{P}^{\n{V}}(M)$, by passing to the opposite category we can argue  with the functor category seen as a $\n{V}^{\ob{rev}}$-linear module. 

By \cite[Theorem 4.8.4.6]{HigherAlgebra}, we have that 
\begin{equation}\label{eqTensorkappaPresentableV}
\ob{Fun}^{\n{V}}(M,\n{V})= \ob{LMod}_{\s{C}}(\ob{Fun}(X\times X, \n{V}))\otimes_{\ob{Fun}(X\times X, \n{V})}  \ob{Fun}(X, \n{V})
\end{equation}
as right $\n{V}$-module. Hence, to show that $\ob{Fun}^{\n{V}}(M,\n{V})\in \cat{Pr}^{\kappa}_{\n{V}^{\ob{rev}}}$ it suffices to prove the following facts: 

\begin{enumerate}

\item $\ob{Fun}(X\times X, \n{V})$ is a $\kappa$-presentable monoidal category. 

\item $\ob{Fun}(X, \n{V})$ is a $\kappa$-presentable   $(\ob{Fun}(X\times X, \n{V}), \n{V})$-bimodule. 

\end{enumerate}
Indeed, if that is the case then $\ob{LMod}_{\s{C}}(\ob{Fun}(X\times X, \n{V}))$ is a $\kappa$-presentable right $\ob{Fun}(X\times X, \n{V})$-module by \cite[Lemma 2.1]{aoki2025higher} and hence so is the tensor product \eqref{eqTensorkappaPresentableV} as a right $\n{V}$-module.

We now prove the previous two claims. First suppose that $X$ is itself a $\kappa$-small anima. Then the category of $\kappa$-compact objects of $\ob{Fun}(X, \n{V})$ is given by 
\[
\ob{Fun}(X, \n{V})_{\kappa}= \ob{Fun}(X, \n{V}_{\kappa}).
\]
Similarly, the $\kappa$-compact objects of the endormorphism category $\ob{Fun}(X\times X, \n{V})$ is $\ob{Fun}(X\times X, \n{V}_{\kappa})$. Given $F,G\in \ob{Fun}(X\times X, \n{V})$ their composition $F\circ G\colon X\times X\to \n{V}$ is the functor 
\[
F\circ G (y,x)=\varinjlim_{z\in X} F(y,z)\otimes_{\n{V}} G(z,x).
\]
As $X$ is $\kappa$-small, we see that the composition  $\circ$ leaves $\ob{Fun}(X\times X, \n{V}_{\kappa})$ stable, proving that $\ob{Fun}(X\times X, \n{V})$ is a $\kappa$-presentable monoidal category. The same argument also shows that $\ob{Fun}(X,\n{V})$ is a $\kappa$-presentable $(\ob{Fun}(X\times X, \n{V}), \n{V})$-bimodule.

For a general anima $X$, note that by taking left Kan extensions one has that 
\[
\ob{Fun}(X,\n{V})=\varinjlim_{X'\subset X} \ob{Fun}(X',\n{V})
\]
where $X'$ runs over $\kappa$-small anima mapping to $X$. By passing to the endomorphism category, and since colimits are $2$-functors, we get that 
\[
\ob{Fun}(X\times X, \n{V})=\varinjlim_{X'\to X} \ob{Fun}(X'\times X', \n{V})
\]
as monoidal category. From the discussion when $X$ is $\kappa$-small,  the previous shows that $\ob{Fun}(X\times X, \n{V})$ is a $\kappa$-presentable monoidal category, and that $\ob{Fun}(X, \n{V})$ is a $\kappa$-presentable $(\ob{Fun}(X\times X, \n{V}), \n{V})$-bimodule as wanted.

Let $f\colon N\to M$ be a functor of $\n{V}$-enriched  small categories. We want to see that $f_!\colon \n{P}^{\n{V}}(N)\to \n{P}^{\n{V}}(M)$   preserves $\kappa$-compact objects (as $f_!$ is already $\n{V}$-linear, and $N$ and $M$ are $\kappa$-presentable left $\n{V}$-modules,  this automatically implies that the morphism $f$ lands in $\ob{LMod}_{\n{V}}(\cat{Pr}^{\kappa})$). By \cite[Theorem B]{ShayNaturalityYoneda} we know that the map  $f_!$ is an internal left adjoint in $\n{V}$-linear categories and its right adjoint $f^{\circledast}$ also preserves colimits.  This formally implies that $f_!$ preserves $\kappa$-compact objects, proving what we wanted.
\end{proof}

\begin{corollary}\label{CoroPresentableYoneda}
Let $\kappa$ be an uncountable regular cardinal. The enriched presheaf functor $\n{P}^{\ob{Pr}^{\kappa}}\colon \cat{Cat}^{\cat{Pr}^{\kappa}}\to \ob{LMod}_{\cat{Pr}^{\kappa}}(\cat{Pr})$ factors through a functor 
\[
\n{P}^{\cat{Pr}^{\kappa}}\colon \cat{Cat}^{\cat{Pr}^{\kappa}}\to 2\cat{Pr}^{\kappa}
\]
where $2\cat{Pr}^{\kappa}:=\ob{LMod}_{\cat{Pr}^{\kappa}}(\cat{Pr}^{\kappa})$ (see \cite[Remark 2.8]{aoki2025higher}). 
\end{corollary}

After the previous discussion, we can define the $\kappa$-presentable category of kernels. 

\begin{definition}\label{DefPresentableCategoryKernels}
Let $(\n{C},E)$ be a  small geometric set up with finite limits, let $\ob{D}$ be a $\kappa$-presentable  six functor formalism on $(\n{C},E)$ and $S\in \n{C}$. The \textit{presentable category of kernels over $S$} is the $\cat{Pr}^{\kappa}$-enriched presheaf category
\[
\cat{Pr}^{\kappa}_{\ob{D},S}:=\n{P}^{\cat{Pr}^{\kappa}}(\ob{K}_{\ob{D},S}). 
\]
More functorially, we define the lax symmetric monoidal functor $\cat{Pr}_{\ob{D},-}^{\kappa}$ from $\Corr(\n{C},E)$ to $2\cat{Pr}^{\kappa}$ to be the composite  of the kernel category functor and the enriched presheaves functor
\begin{equation}\label{eqPresentableKernels}
\cat{Pr}_{\ob{D},-}^{\kappa}\colon \ob{Corr}(\n{C},E)\xrightarrow{\ob{K}_{\ob{D},-}} \cat{Cat}^{\cat{Pr}^{\kappa}} \xrightarrow{\n{P}^{\cat{Pr}^{\kappa}}} \ob{LMod}_{\cat{Pr^{\kappa}}}(\cat{Pr}^{\kappa})=2\cat{Pr}^{\kappa}.
\end{equation}
Given $f\colon Y\to X$ a map in $\n{C}$ we let $f^*_1\colon \cat{Pr}^{\kappa}_{\ob{D},X}\to \cat{Pr}^{\kappa}_{\ob{D},Y}$ be the corresponding pullback functor. If $f\in E$ we let $f_{1,!}\colon \cat{Pr}^{\kappa}_{\ob{D},Y}\to \cat{Pr}^{\kappa}_{\ob{D},X}$ be the lower $!$-functor\footnote{The lower-index ``$1$'' refers to the categorical level where the $*$-pullback and the $!$-pushforward are defined.}. We refer to \cite[Theorem 4.2.4]{HeyerMannSix} for a more explicit description of such maps when restricted to the category of kernels $\ob{K}_{\ob{D},-}$. 
\end{definition}

\begin{remark}\label{RemarkImprovementStefanich}
The functor \eqref{eqPresentableKernels} has a significant enhancement if $E=\ob{all}$ consists of all $!$-able arrows, this can always be assumed by restricting the six functor formalism from $\n{C}$ to $\n{C}_E$.  Namely, if $\ob{D}\colon \Corr(\n{C})\to \cat{Pr}^{\kappa}$ is a six functor formalism in the full span category of $\n{C}$, Stefanich proved in \cite{StefanichCategorification} that $\ob{D}$ promotes to a lax symmetric monoidal functor of symmetric monoidal $2$-categories 
\[
\cat{Pr}^{\kappa}_{\ob{D},-}\colon 2\Corr(\n{C})\to 2\cat{Pr}^{\kappa}
\]
from the $2$-category of spans of $\n{C}$ (he even proved all the higher categorical analogues of this statement). A special consequence of this is the \textit{ambidexterity} of the six functors at the $2$-categorical level, namely, if $f\colon Y\to X$ is a morphism in $\n{C}$ then the induced functor  $f_!\colon \cat{Pr}^{\kappa}_{\ob{D},Y}\to \cat{Pr}^{\kappa}_{\ob{D},X}$ is naturally both left and right adjoint to the pullback functor $f^*\colon \cat{Pr}^{\kappa}_{\ob{D},X}\to \cat{Pr}^{\kappa}_{\ob{D},Y}$, and this adjunction is  natural and preserved under pullbacks along $X$, see \cite[Lemma 4.2.7]{HeyerMannSix}.
\end{remark}

The presentable category of kernels has a simple description in  case the six functor formalism satisfies the  categorical K\"unneth formula. We borrow some definitions from \cite{kesting2025categorical}. 

\begin{definition}\label{DefinitionKunneth}
Let $(\n{C},E)$ be a geometric set up with finite limits and $\ob{D}$ a $\kappa$-small presentable six functor formalism. Let $f\colon Y\to S$ be a map in $E$, we say that $f$ is \textit{K\"unneth in $\n{C}$} if for all $X\to S$ in $\n{C}$ the natural map 
\[
\ob{D}(Y)\otimes_{\ob{D}(S)} \ob{D}(X)\to \ob{D}(Y\times_S X)
\]
is an equivalence. We say that $\ob{D}$ is \textit{K\"unneth over $S$ in $\n{C}$} if any map $f\colon Y\to S$ in $E$ is K\"unneth in $\n{C}$.  If the category $\n{C}$ is clear from the context, we simply say that $f$ is K\"unneth, or that $\ob{D}$ is K\"unneth over $S$. 
\end{definition}

\begin{proposition}\label{PropKunnethKernel}
Let $(\n{C},E)$ be a small geometric set up and let $\ob{D}$ be a $\kappa$-presentable six functor formalism. Let $S\in \n{C}$ and suppose that all objects in $\n{C}^E_{/S}$ are K\"unneth over $S$ in $\n{C}^E_{/S}$. Then the natural map 
\begin{equation}\label{eqMapPresentableKunneth}
\ob{Fun}_S(S,-)\colon \ob{K}_{\ob{D},S} \to \cat{Pr}^{\kappa}_{\ob{D}(S)}
\end{equation}
is $2$-fully faithful. Moreover, the natural map 
\begin{equation}\label{eqpkampsasdfmpasd}
\cat{Pr}_{\ob{D},S}^{\kappa}\xrightarrow{\sim} \cat{Pr}_{\ob{D}(S)}^{\kappa}
\end{equation}
obtained by the adjunction of \cite[Theorem 6.4.4]{HinichColimitsEnriched}  is an equivalence of $\cat{Pr}^{\kappa}$-linear categories. 
\end{proposition}
\begin{proof}
To see that the map \eqref{eqMapPresentableKunneth} is $2$-fully faithful, it suffices to see that for $X,Y\in \n{C}^E_{/S}$, the natural map 
\[
\ob{Fun}_S(X,Y)\to \ob{Fun}_{\ob{D}(S)}(\ob{D}(X), \ob{D}(Y))
\]
is an equivalence. Since $X$ is naturally self dual in $\ob{K}_{\ob{D},S}$,  it suffices to show that $\ob{Fun}_S(S,-)$ is a symmetric monoidal functor, this is precisely the condition that  $\ob{D}$ is K\"unneth over $S$. 

Next, we show that \eqref{eqpkampsasdfmpasd} is an equivalence of $\cat{Pr}^{\kappa}$-enriched categories. Let $B\ob{D}(S)$ be the $\cat{Pr}^{\kappa}$-enriched category with one object $*$ and endomorphism category $\ob{D}(S)$. We have a fully faithful map $B\ob{D}(S)\to \ob{K}_{\ob{D},S}$, and passing to $\cat{Pr}^{\kappa}$-enriched categories we get a functor 
\begin{equation}\label{eqom2wfpqnwhf3of}
\cat{Pr}_{\ob{D}(S)}^{\kappa}=\n{P}^{\cat{Pr}^{\kappa}}(B\ob{D}(S))\to \cat{Pr}_{\ob{D},S}^{\kappa}
\end{equation}
whose composite with \eqref{eqpkampsasdfmpasd} is the identity. It is left to show that the composite $\cat{Pr}_{\ob{D},S}^{\kappa}\to \cat{Pr}_{\ob{D}(S)}^{\kappa} \to \cat{Pr}_{\ob{D},S}^{\kappa}$ is equivalent to the identity, by passing to  enriched  presheaves it suffices to show that it induces the identity on $\ob{K}_{\ob{D},S}$, but this follows from the fully faithful embedding  $\ob{Fun}_{S}(S,-)\colon \ob{K}_{\ob{D},S}\to \cat{Pr}_{\ob{D}(S)}^{\kappa}$ and the fact that $\cat{Pr}_{\ob{D}(S)}^{\kappa}\to \cat{Pr}_{\ob{D},S}^{\kappa}\to \cat{Pr}_{\ob{D}(S)}^{\kappa}$ is the  identity. 
\end{proof}

Categorical K\"unneth can be checked in a family of generators of the category of $D$-stacks  for the six functor formalism.

\begin{lemma}\label{LemmaKunnethSetGeneratorsSixFunctor}
Let $(\n{C},E)$ be a geometric set up with finite limits, let $\ob{D}$ be a $\kappa$-presentable six functor formalism on $(\n{C},E)$  and $S\in \n{C}$. Let $\cat{Shv}_{\ob{D}}(\n{C})$ denote the category of $\ob{D}$-sheaves on $(\n{C},E	)$ obtained by localizing $!$-equivalences as in \cite[Definition 6.3.16]{SolidNotes}, and let $E'$ be the class of $!$-able arrows constructed from $\ob{D}$ by applying \cite[Theorem 3.4.11]{HeyerMannSix} (cf. \cite[Corollary 6.3.19]{SolidNotes}). Finally, let  $\cat{Shv}_{\ob{D}}(\n{C}_{/S}^E)$ denote the category of $\ob{D}$-sheaves constructed from the restriction of $\ob{D}$ to $\n{C}_{/S}^{E}$.  Suppose $\ob{D}$ is K\"unneth over $S$ in $\n{C}$,   then any $!$-able arrow $Y'\to S$ in $\cat{Shv}_{\ob{D}}(\n{C}_{/S}^E)$ is K\"unneth in $\cat{Shv}_{\ob{D}}(\n{C})$. 
\end{lemma}
\begin{proof}
The hypothesis of the lemma say that, given $Y\to S$ a map in $E$ of objects in $\n{C}$, for any $X\to S$ in $\n{C}$ the natural map 
\[
\ob{D}(Y)\otimes_{\ob{D}(S)}\ob{D}(X)\xrightarrow{\sim} \ob{D}(Y\times_S X)
\]
is an equivalence. Let $Y'\to S$ be a $!$-able object in $\ob{Shv}_{\ob{D}}(\n{C}^E_{/S})$ and let $X'\in \ob{Shv}_{\ob{D}}(\n{C})_{/S}$ be a $\ob{D}$-sheaf over $S$. The conditions on $Y'$ imply that $Y'=\varinjlim_{i} Y_i$ with $Y_i\to S$ objects in $E$ (hence $Y_i\in \n{C}$). Furthermore, $!$-ability of $Y'$  (in particular, the tameness of \cite[Theorem 3.4.11 (iv)]{HeyerMannSix}) yields a natural equivalence of categories 
\begin{equation}\label{eqKunneth1Lemma}
\ob{D}(Y')=\varinjlim_{i} \ob{D}(Y_i), 
\end{equation}
where the transitions maps are given by $!$-pushforward maps. That is, the diagram $\{Y_i\to Y\}_i$ is a diagram of $!$-able arrows in $\ob{Shv}_{\ob{D}}(\n{C}^E_{/S})$  which induces a $!$-equivalence after taking colimits. 

On the other hand, given a  map $X'\to S$ in $\ob{Shv}_{\ob{D}}(\n{C})$  we can write $X'=\varinjlim_{j} X_j$ with $X_j\to S$ an arrow in $\n{C}$. In particular, we have that 
\begin{equation}\label{eqKunneth2Lemma}
\ob{D}(X')=\varprojlim_{j} \ob{D}(X_j)
\end{equation}
where the transition morphisms are now given by $*$-pullback maps. 
We deduce that 
\[
\begin{aligned}
\ob{D}(Y')\otimes_{\ob{D}(S)} \ob{D}(X') & = \varinjlim_{i} \ob{D}(Y_i)\otimes_{\ob{D}(S)} (\varprojlim_{j} \ob{D}(X_j)) \\ 
& = \varinjlim_{i} \varprojlim_{j}( \ob{D}(Y_i) \otimes_{\ob{D}(S)} \ob{D}(X_j) ) \\ 
& = \varinjlim_i \varprojlim_j \ob{D}(Y_i\times_S X_j) \\ 
& =\varinjlim_i  \ob{D}(Y_i\times_S X') \\
& = \ob{D}(Y'\times_S X')
\end{aligned}
\]
where:  in the first equivalence we have used \eqref{eqKunneth1Lemma} and \eqref{eqKunneth2Lemma}.  In the second equivalence we have used that $\ob{D}(Y_i)$ is a naturally self-dual over  $\ob{D}(S)$ as $Y_i\to S$ is $!$-able and satisfies K\"unneth. In the third equivalence we used K\"unneth of $Y_i\to S$ in $\n{C}$. The fourth equivalence follows form the fact that, for each $i$,  $Y_i\times_{S} X'=\varinjlim_{j} Y_i\times_S X_j$ in  $\ob{Shv}_{\ob{D}}(\n{C})$, and hence we have $\ob{D}(Y_i\times_{S} X')=\varprojlim_{j}\ob{D}(Y_i\times_S X_j)$. The fifth equivalence follows form the fact that   the map $\varinjlim_{i} Y_i\times_{S} X'\to Y'\times_{S} X'$ is a $!$-equivalence as these are stable under base change. We deduce the lemma.  
\end{proof}

\begin{Convention}\label{ConventionPresentable}
From now on we shall fix an uncountable cardinal $\kappa$ and assume that all our six functor formalisms are $\kappa$-presentable. Given $\ob{D}$ a six functor formalism in a small geometric set up $(\n{C},E)$ and $S\in \n{C}$, we   write $\cat{Pr}_{\ob{D},S}$ for the $\kappa$-presentable category of kernels over $S$ (making $\kappa$ implicit in the notation).
\end{Convention}

\begin{remark}\label{RemarkPresentableCategoryKernlesComparison}
In \cite[Appendix to Lecture V]{SixFunctorsScholze} it is discussed a different construction of the presentable category of kernels that we recall here. Let $\n{C}$ be a small category with finite limits and $\ob{D}\colon \ob{Corr}(\n{C})\to \cat{Pr}^{\kappa}$ a $\kappa$-presentable six functor formalism, that is, a lax symmetric monoidal functor. Consider the functor category $\cat{Fun}(\ob{Corr}(\n{C}), \cat{Pr}^{\kappa})$ endowed with Day's convolution, then $\ob{D}$ is equivalent to a commutative algebra $D\in \ob{CAlg}(\cat{Fun}(\ob{Corr}(\n{C}), \cat{Pr}^{\kappa}))$. Denote by $S$ the final object of $\n{C}$, then \textit{loc. cit.} defines the $\kappa$-presentable category of kernels as the symmetric monoidal $\cat{Pr}^{\kappa}$-linear category  
\[
\cat{Pr}^{\kappa}_{\ob{D},S}:=\ob{Mod}_D(\cat{Fun}(\ob{Corr}(\n{C}), \cat{Pr}^{\kappa})). 
\]
By construction, the Yoneda embedding together with the natural anti-involution of $\ob{Corr}(\n{C})$  produces a natural symmetric monoidal map $\Psi' \colon \ob{Corr}(\n{C})\to \cat{Pr}^{\kappa}_{\ob{D},S}$. We claim that the $\cat{Pr}^{\kappa}$-enriched essential image of this map is nothing but the category of kernels $\ob{K}_{\ob{D},S}$ of \Cref{DefKernelCategory},  that $\Psi'$ identifies with the map $\Psi_{\ob{D}}$ of \Cref{RemarkFunctorsToCorrespondence}, and that the natural map of $\cat{Pr}^{\kappa}$-enriched categories
\[
\n{P}^{\cat{Pr}^{\kappa}}(\ob{K}_{\ob{D},S})\to \cat{Pr}^{\kappa}_{\ob{D},S}
\]
is an equivalence.  This is precisely the content of \cite[Corollary  2.7]{aoki2026berkovich2motivesnormedring}.  
\end{remark}

\subsection{Analytic stacks}\label{ss:AnStk}

In this paper we will study Cartier duality of certain analytic stacks, this requires to set up the theory.  
 We will work with the category of analytic stacks as in \cite[Section 6.3]{SolidNotes}. We say that an analytic stack $X$ is \textit{affinoid} if it is of the form $X=\ob{AnSpec} A$ for $A$ an analytic ring.\footnote{The choice of the terminology \textit{affinoid} instead of \textit{affine} analytic stack is justified by two main reasons. In one hand, the theory of Berkovich and adic spaces uses the term \textit{affinoid} to refer to the spectrum associated to Banach algebras and Huber pairs respectively. Since analytic geometry generalizes these geometric frameworks, and not just that of algebraic varieties, it is reasonable to keep the standard terminology from classical analytic geometry. 
 
  On the other hand, the theory of Gestalten of Scholze and Stefanich already defines for each $n\in \bb{N}$ the notion of \textit{$n$-affine map}. Using this language, it turns out that morphisms between affine schemes are always $0$-affine (that is, they only depend on the underlying rings),  while maps between affinoid analytic stacks are only $1$-affine in general (that is, they depend on the full category of complete modules and not just in the underlying ring). Moreover, the notions of $0$ and $1$-affine maps will appear later in the discussion of descent of Cartier duality, for this reason it is more conveniable to use a different terminology for analytic stacks.}  Given a condensed ring $R$ we let $R^{\cond}$ be the trivial analytic ring structure on $R$.    We let $\n{C}^0$ be a small full subcategory of affinoid analytic stacks such that $\n{C}^{0,\op}\subset \cat{AnRing}$ contains at least $\Z^{\cond}$, $\Z_{\sol}$, $\Z[T]_{\sol}$, is stable under countable colimits, and is stable under taking countably presented algebras  with induced analytic ring structure. We let $\n{C}\subset \cat{AnStk}$ be the full subcategory generated under countable colimits by $\n{C}^0$. In particular, $\n{C}$ is also stable under finite limits. We let $\ob{D}$ denote the six functor formalism  of quasi-coherent sheaves on $\n{C}$,  and let $E$ be the class of $!$-able arrows as in \cite[Definition 6.3.24]{SolidNotes}. We fix $\kappa$ a regular cardinal  such that $\ob{D}\colon \ob{Corr}(\n{C},E)\to \cat{Pr}$ factors through $\cat{Pr}^{\kappa}$. In the next sections we will omit $\kappa$ from the notation following \Cref{ConventionPresentable}.

By \cite[Proposition 4.1.14]{SolidNotes} and  \Cref{LemmaKunnethSetGeneratorsSixFunctor},  one deduces that  the category $\n{C}^E_{/S}$ of $!$-able analytic stacks over an affinoid base $S$ satisfies the categorical K\"unneth formula over $S$.     In this paper we will need a slightly more general version of this property for other analytic stacks $S$, in particular for suitable classifying stacks. We will require the following definition. 
 
 %By \cite[Corollary 1.5.1]{kesting2025categorical} the category $\n{C}^E_{/S}$ of $!$-able analytic stacks over a base stack $S$ satisfies the categorical K\"unneth formula under the following condition: there is an affinoid analytic stack  $S_0$, and a  $!$-cover $S_0\to S$ in $E$ whose \v{C}ech nerve  consists of affinoid analytic stacks. In this paper we will need a slightly more general version of this statement. 
 
 \begin{definition}\label{DefQuasiAffineAnStk}
A morphism of analytic stacks $f\colon Y\to X$ is an \textit{open immersion} if it is an immersion and suave. We say that $f$ is a \textit{closed immersion} if it is an immersion and prim.

 An analytic stack $Y$ is said \textit{quasi-affinoid} if there is an open immersion $Y\to \overline{Y}$  where $\overline{Y}$ is an affinoid analytic stack.\footnote{For similar reasons justifying the choice of \textit{affinoid analytic stack}, we prefer the terminology \textit{quasi-affinoid} instead of the more standard notion \textit{quasi-affine} appearing in algebraic geometry. Furthermore, it is not true that, for the usual quasi-coherent six functor formalism, a quasi-affine scheme is quasi-affinoid as Zariski open immersions are not suave in general.} 
 \end{definition}
 
 \begin{lemma}\label{LemmaOpenClosed}
 Let $f\colon Y\to X$ be a map of analytic stacks. 
 
 \begin{enumerate}
 
\item Suppose that $f$ is an open immersion. Then $f$ is cohomologically \'etale, K\"unneth and the pullback map 
\[
f^*\colon \ob{D}(X)\to \ob{D}(Y)
\]
is an open localization of symmetric monoidal categories in the sense of \cite[Definition 5.1.3]{SolidNotes}. In particular, there is a unique coidempotent coalgebra $C$ in $\ob{D}(X)$ such that $\ob{D}(Y)=\ob{cLMod}_{C}(\ob{D}(X))$.

\item Suppose that $f$ is a closed immersion. Then $f$ is cohomologically proper, K\"unneth and the pullback map 
\[
f^*\colon \ob{D}(X)\to \ob{D}(Y)
\]
is a closed localization of symmetric monoidal categories in the sense of \cite[Definition 5.1.3]{SolidNotes}. In particular, there is an idempotent algebra $A\in \ob{D}(X)$ such that $\ob{D}(Y)=\ob{LMod}_A(\ob{D}(X))$.
 
 \end{enumerate}
 \end{lemma}
 \begin{proof}
 As $f$ is an immersion, the diagonal is an isomorphism, and to check whether it is cohomologically \'etale or proper it suffices to check whether $f$ is suave or prim respectively (\cite[Definition 6.4.1]{HeyerMannSix}). 
 
The proof that $f$ is K\"unneth and that the pullback is an open immersion of symmetric monoidal categories follows from \cite[Proposition 3.12]{kesting2025categorical} (where the hypothesis that $f$ is suave and an immersion are missing). The same  argument works for $f$ a closed immersion.  
 \end{proof}

The following proposition gives the desired criterion for categorical K\"unneth formula  over the  analytic stacks of interest in this paper.

\begin{proposition}\label{PropositionKunnethQuasiAffine}
Let $S\in \n{C}$ and suppose  that there is a family of  $!$-able maps $f=\bigsqcup_i f_i\colon \bigsqcup_{i\in I} S_{i}\to S$ satisfying universal $!$-descent,  together with  factorizations $S_i\to \overline{S}_i\to S$ for all $i\in I$  satisfying the following properties:
\begin{enumerate}

\item  $S_i=\ob{AnSpec}(A_i)$ is affinoid.

\item The map $S_{i}\to \overline{S}_{i}$ is an open immersion.

\item The base change $S_i\times_{S} \overline{S}_i$ is corepresented by  an analytic ring $B_i$ over $A_i$ with the induced analytic ring structure. 

\end{enumerate}

Then all the objects in $\n{C}^E_{/S}$ are K\"unneth over $S$ in $\n{C}$. 
\end{proposition}
\begin{proof}
We claim that the maps $f_i\colon S_i\to S$ are K\"unneth in $\n{C}$. Suppose that this holds, let $Y\to S$ be a $!$-able map of analytic stacks and let $X\to S$ be in $\n{C}$. We want to show that the natural map 
\begin{equation}\label{eqCategoricalKunnethStacks1}
\ob{D}(Y)\otimes_{\ob{D}(S)} \ob{D}(X)\to \ob{D}(Y\times_S X)
\end{equation}
is an equivalence.  Since the $f_i$ are K\"unneth, we know that the base change of \eqref{eqCategoricalKunnethStacks1} along $\ob{D}(S)\to \ob{D}(S_i)$ is given by 
\[
\ob{D}(Y\times_S S_i)\otimes_{\ob{D}(S_i)} \ob{D}(X\times_{S} S_i)\to \ob{D}((Y\times_S S_i)\times_{S_i}(X\times_{S} S_i))
\]
which is an equivalence as $S_i$ is an affinoid analytic stack and $Y\times_S S_i\to S_i$ is $!$-able. Therefore, to show that \eqref{eqCategoricalKunnethStacks1} is an equivalence, it suffices to prove that base change along the functors $\ob{D}(S)\to \ob{D}(S_i)$ is jointly  conservative  on presentable $\ob{D}(S)$-linear categories. For this, it suffices to show that base change along $\ob{D}(S)\to \ob{D}(S')$ is conservative with $S'=\bigsqcup_i S_i$. Since $f\colon S'\to S$ is a $!$-able map satisfying universal $!$-descent, we can write
\[
\ob{D}(S)=\varinjlim_{\Delta^{\op}} \ob{D}(S^{\prime,\bullet})
\]
where  the transition maps are given by $!$-maps, and $S^{\prime,\bullet}\to S$ is the \v{C}ech nerve of $f$. Then, if $N\to M$ is a morphism of presentable  $\ob{D}(S)$-linear categories that is an isomorphism after base change to $\ob{D}(S')$, we have an equivalence of simplicial categories 
\[
\ob{D}(S^{\prime,\bullet})\otimes_{\ob{D}(S)} N \xrightarrow{\sim}  \ob{D}(S^{\prime,\bullet})\otimes_{\ob{D}(S)} M,
\]
 passing to geometric realizations along $!$-maps one obtains an equivalence $N\xrightarrow{\sim} M$,  proving the desired conservativity.

It is left to show the claim, namely, that $S_i\to S$ is K\"unneth in $\n{C}$. Since this is a question for each individual $i$, we can assume that $I=\{0\}$ and denote $f\colon S_0\to \overline{S}_0\to  S$.  By \Cref{LemmaOpenClosed} (1) we know that the map $S_0\to \overline{S}_0$ is K\"unneth, and since K\"unneth morphisms are stable under composition (\cite[Proposition 1.2 (1)]{kesting2025categorical}), it suffices to show that $g\colon \overline{S}_0\to S$ is K\"unneth. By hypothesis, the map $f\colon S_0\to S$ is an epimorphism of analytic stacks  and the base change $g'\colon S_1:=S_0\times_S \overline{S}_0\to S_0$ is corepresented by a morphism of analytic rings $A\to B$ with the induced analytic ring structure. This implies that the map $g$ is prim and that the functor $g_*\colon \ob{D}(\overline{S}_0)\to \ob{D}(S)$ is conservative. Therefore, the monadicity theorem and the $\ob{D}(S)$-linearity of the monad $g_*g^*$ yield a natural equivalence of categories
\begin{equation}\label{eqPresentationModulesCategoryKunnethCompc}
\ob{D}(\overline{S}_0)\cong \ob{Mod}_{g_* 1} (\ob{D}(S)). 
\end{equation}
Let $X\to S$ be an analytic stack, we want to show that the natural map
\begin{equation}\label{eqKunneth3Lemmasas}
\ob{D}(\overline{S}_0)\otimes_{\ob{D}(S)} \ob{D}(X)\xrightarrow{\sim} \ob{D}(\overline{S}_0\times_S X)
\end{equation}
is an equivalence. Consider the cartesian diagram
\[
\begin{tikzcd}
\overline{S}_0\times_S X \ar[d,"g_X"] \ar[r,"h'"] & \overline{S}_0  \ar[d,"g"]\\ 
X \ar[r,"h"]& S
\end{tikzcd}
\]
 By \eqref{eqPresentationModulesCategoryKunnethCompc}  and \cite[Theorem 4.8.4.6]{HigherAlgebra} the left term in \eqref{eqKunneth3Lemmasas} is equivalent to $\ob{Mod}_{g_{X,*}1} (\ob{D}(X))$, where $g_{X,*}1= h^{*}g_*1$ by proper base change. As $g$ is prim then so is $g_X$, and to show that \eqref{eqKunneth3Lemmasas} is an equivalence it suffices to show that $g_{X,*}\colon \ob{D}(\overline{S}_0\times_S X)\to \ob{D}(X)$ is conservative, namely, this together with primness of $g_X$ will imply that the $\ob{D}(X)$-linear  functor $g_{X,*}$ is monadic  and therefore that 
 \[
 \ob{D}(\overline{S}_0\times_S X)\cong \ob{Mod}_{g_X,*}(\ob{D}(X)). 
 \]
Finally, by proper base change,  the conservativity of $g_{X,*}$ can be proved after taking pullbacks along a cover of $X$. Taking as cover $X'=S_0\times_S X\to X$, we see that $\overline{S}_0\times_S  X' = \AnSpec(B)\times_{\AnSpec(A)} X'$ with projection $g_{X'}\colon \AnSpec(B)\times_{\AnSpec(A)} X'\to X'$. Since $B$ has the induced structure over $A$, one easily computes that 
\[
\ob{D}(\overline{S}_0\times_S  X')= = \ob{D}(\AnSpec(B)\times_{\AnSpec (A)} X')=\ob{Mod}_B(\ob{D}(X'))
\]
yielding the conservativity of $g_{X',*}$ as desired, and finishing the proof of the proposition. 
\end{proof}
%By \cite[Proposition 1.2]{kesting2025categorical}, it suffices to prove the statement when $S$ is itself quasi-affinoid. Let $j\colon S\to \overline{S}$ be an open immersion with $\overline{S}$ an affinoid analytic stack. Since $j$ is  an immersion, we have a fully faithful embedding $\n{C}_{/S}\hookrightarrow \n{C}_{/\overline{S}}$ preserving fiber products. Furthermore, given $Y,X\to S$ two  maps over $S$ with $X\to S$ in $E$, we have an equivalence 
%\[
%X\times_{S} Y=X\times_{\overline{S}} Y.
%\]
%By \cite[Corollary 1.5.1]{kesting2025categorical} the natural map 
%\[
%\ob{D}(X)\otimes_{\ob{D}(\overline{S})} \ob{D}(Y)\to \ob{D}(X\times_{\overline{S}} Y)=\ob{D}(X\times_{S} Y)
%\]
%is an equivalence. Since $j^*\colon \ob{D}(\overline{S})\to\ob{D}(S) $ is an idempotent map of symmetric monoidal categories, being an open map, we have that 
%\[
%\ob{D}(X)\otimes_{\ob{D}(\overline{S})} \ob{D}(Y)=\ob{D}(X)\otimes_{\ob{D}(S)} \ob{D}(Y),
%\]
%proving that $X\to S$ is K\"unneth as wanted. 

Categorical K\"unneth can be proved after composition along open immersions.

\begin{lemma}\label{LemmaCategoricalKunnethAfterOpenIm}
Let $j\colon S\to \overline{S}$ be an open immersion of analytic stacks and let $Y\to S$ be a $!$-able map. Then $Y\to S$ is K\"unneth if and only if $Y\to \overline{S}$ is K\"unneth in $\n{C}$. 
\end{lemma}
\begin{proof}
If $Y\to S$ is K\"unneth then so is $Y\to \overline{S}$ as open immersions are K\"unneth in $\n{C}$ (\Cref{LemmaOpenClosed} (1)), and K\"unneth morphisms are stable under compositions (\cite[Proposition 1.2 (1)]{kesting2025categorical}). Suppose now that $Y\to \overline{S}$ is K\"unneth, we want to show that $Y\to S$ is K\"unneth. For that, let $X\to S $ be a morphism in $\n{C}$, we want to show that the natural map 
\[
\ob{D}(Y)\otimes_{\ob{D}(S)} \ob{D}(X)\to \ob{D}(Y\times_S X)
\]
is an equivalence. As $S\to \overline{S}$ is open, the morphism of symmetric monoidal categories $\ob{D}(S)\to \ob{D}(\overline{S})$ is idempotent in $\cat{Pr}_{\ob{D}(S)}$ which yields  
\[
\ob{D}(Y)\otimes_{\ob{D}(S)} \ob{D}(X)= \ob{D}(Y)\otimes_{\ob{D}(\overline{S})} \ob{D}(X).
\]
Similarly, since $S\subset \overline{S}$ is an immersion, we have that $Y\times_S X= Y\times_{\overline{S}} X$. Hence, since $Y\to \overline{S}$ is K\"unneth, we get that 
\[
\begin{aligned}
\ob{D}(Y)\otimes_{\ob{D}(S)} \ob{D}(X) & = \ob{D}(Y)\otimes_{\ob{D}(\overline{S})} \ob{D}(X) \\ 
& = \ob{D}(Y\times_{\overline{S}} X) \\ 
& =\ob{D}(Y\times_{S} X)
\end{aligned}
\]
proving what we wanted. 
\end{proof}

We end this section with a  technical lemma in representation theory that will be helpful later to study classifying stacks of affinoid groups in analytic stacks.

\begin{lemma}\label{LemmaRepTheoryFlat}
Let $A$ be a static analytic ring, that is, an analytic ring with $A[S]$ static for all $S$ light profinite set. Let $G\to \AnSpec A$ be an analytic group stack  that is corepresented by an $A$-algebra $R$ with induced analytic ring structure.  Suppose that that $R$ is flat with respect to $\otimes_A$, and that the counit map $R\to A$ is of finite tor amplitude. Let $BG=\AnSpec A/G$ be the classifying stack of $G$ over $\AnSpec A$ and let $e\colon \AnSpec A\to BG$ be the quotient map. Then $\ob{D}(BG)$ has a natural $t$-structure making $e^*$ a $t$-exact functor. Furthermore, the following hold: 

\begin{enumerate}

\item $\ob{D}(BG)^{\heartsuit}$ is naturally equivalent to the category of right  $R$-comodules in $\ob{D}(A)^{\heartsuit}$. Similarly, $\ob{D}(BG)$ is equivalent to the category of right $R$-comodules in $\ob{D}(A)$.

\item   The full subcategory $\ob{D}_{-}(BG)\subset \ob{D}(BG)$ of eventually coconnective objects is the left bounded derived category of the heart.

\item $\ob{D}(BG)$ is the left completion of $\ob{D}_{-}(BG)$.

\end{enumerate}
\end{lemma}
\begin{proof}
Let $e\colon \AnSpec A\to BG$, the map $e$ is prim being locally represented by an affinoid analytic stack with the induced analytic ring structure. In particular, both functors $e^*$ and $e_*$ are $\ob{D}(A)$-linear. The functor $e^*$ is conservative as $e$ is an epimorphism of analytic stacks. Hence, by the (co)monadicity theorem, we have that 
\[
\ob{D}(BG)= \ob{cLMod}_{U(A)}(\ob{D}(A))
\]
is the category of left comodules, where $U$ is the comonad $e^*e_*$. We have that $G=(\AnSpec A)\times_{BG} (\AnSpec A)$, so  proper base change identifies $U(A)=R$ as $A$-modules. In particular, the comonad $U(A)$ is flat over $\ob{D}(A)$, which produces a natural $t$-structure on  $\ob{D}(BG)$ such that $e^*$ is a $t$-exact functor.  To compute the  co-algebra $U(A)$, it suffices to determine the abelian category  $\ob{D}(BG)^{\heartsuit}$. This abelian category is given by the category of descent data 
\[
\ob{D}(BG)^{\heartsuit}=\varprojlim_{\Delta_{\leq 2}} \ob{D}(G^{\bullet})^{\heartsuit}
\]
with non-derived $*$-pullbacks as transition maps.  A standard computation yields that $U(A)=R^{\op}$ as coalgebras (the reverse coalgebra of $R$), proving (1). For completeness, let us give the computation. We follow the convention that the classifying stack $BG$ is the colimit of the simplicial diagram of  \cite[Definition B.3.3]{HeyerMannSix}.  The category $\ob{D}(BG)^{\heartsuit}$ is the limit of the diagram 
\[
\begin{tikzcd}
\ob{D}(A)^{\heartsuit} \ar[r, shift left =1 ex]
\ar[r, shift right =1 ex] & \ob{D}(G)^{\heartsuit} \ar[r,"m^*"] \ar[r, "\pi_2^*", shift left =2ex] \ar[r,"\pi_1^*"', shift right =2ex] & \ob{D}(G\times G)^{\heartsuit}. 
\end{tikzcd}
\]
 Thus, an object in the category of descent data is given by a tuple $(M,\alpha)$ where $M\in \ob{D}(A)^{\heartsuit}$ and $\alpha\colon  M\otimes_A R\to M\otimes_A R$ is an  isomorphism of $R$-modules (the orbit map)  satisfying  the following properties:
 \begin{itemize}
 
 \item[(i)] The pullback of $\alpha$ along the counit map $\mu\colon R\to A$ is the identity.
 \item[(ii)]  $\alpha$ satisfies a  cocycle condition, that is,  $\pi_1^* \alpha \circ \pi_2^* \alpha=m^* \alpha$ (see \cite[Lemma A.4.23]{HeyerMannSix}).  
 
\end{itemize} 
  A morphism of descent data $(M,\alpha)\to (N,\beta)$ is a map $M\to N$ of $A$-modules that intertwines the isomorphisms $\alpha$ and $\beta$ respectively. Given $(M,\alpha)$ as above, its associated \textit{right} $R$-comodule is given by the pair $(M,\rho)$ where $\rho\colon M\xrightarrow{\id_M\otimes 1} M\otimes_A R \xrightarrow{\alpha} M\otimes_A R$.    We need to see that $(M,\rho)$ satisfies the following axioms:
  
  \begin{itemize}
  
\item[(a)] The composite of $\rho$ with the counit $\mu$ is the identity of $M$.

\item[(b)]   One has $(\rho\otimes \id_R)\circ \rho= (\id_M\otimes \Delta)\circ \rho$ where $\Delta\colon R\to R\otimes R$ is the comultiplication map. 
  
  \end{itemize}
 
The axiom (a) follows immediately from the property (i). For the axiom (b), by (ii) we have the following commutative diagram 
 \[
 \begin{tikzcd}
M\otimes_A 1\otimes_A 1 \ar[r,"\rho"]  \ar[d] & M\otimes_A 1  \otimes_A R  \ar[d] \ar[rd, "\rho\otimes \id_R"] & \\ 
M\otimes_A R \otimes_A R   \ar[r,"\pi_2^* \alpha"] &   M\otimes_A R\otimes_A R \ar[r,"\pi_1^* \alpha"]& M\otimes_A R\otimes_A R  
 \end{tikzcd}
 \]
On the other hand, we have a commutative diagram 
\[
\begin{tikzcd}
M \otimes_A 1\ar[d] \ar[rd,"\rho"] & \\ 
M\otimes_A R \ar[r,"\alpha"] \ar[d,"\id_M\otimes \Delta"] & M\otimes_A R \ar[d,"\id_M\otimes \Delta"] \\ 
M\otimes_A R\otimes_A R \ar[r,"m^*\alpha"] & M\otimes_A R\otimes_A R
\end{tikzcd}
\]
Combining both we see that (ii) yields the axiom (b). 

Conversely, given a right $R$-comodule $(M,\rho)$, the $R$-linearization of the comodule action $\rho$ yields a map of $R$-modules $\alpha\colon M\otimes_A R\to M\otimes_A R$. It is immediate to see that the $R$-linearization of the axiom (a) yields the axiom (i), and that the $R\otimes_A R$-linearization of (b) yields (ii).  It is left to see that $\alpha$ is an isomorphism. Let $\mathcal{C}\subset \ob{cRMod}_R(\ob{D}(A)^{\heartsuit})$ be the full subcategory of  right $R$-comodules $(M,\rho)$ for which  the linearization $\alpha\colon M\otimes_A R\to M\otimes_A R$ is an isomorphism, we want to see that the inclusion is an equivalence. Note that, since $R$ is flat,  $\mathcal{C}$ is stable under finite limits, colimits, extensions and tensor products by objects in $A$.  By the Bar construction, the right $R$-comodule $M$ is equivalent to the totalization (in the abelian category) of $(M\otimes_A R^{\otimes_A n+1})_{[n]\in \Delta_{\leq 2}}$, where the right comodule structure on $M\otimes_A R^{\otimes_A n+1}$ arises from the  last  copy of $R$ on the right. Therefore, to show that $\mathcal{C}=\ob{cRMod}_R(\ob{D}(A)^{\heartsuit})$ it suffices to show that $R\in \mathcal{C}$,  but this is obvious since the map $\alpha\colon  R\otimes_A R\to R\otimes_A R$ arises from the isomorphism of stacks  $G\times G\xrightarrow{\sim} G\times G$ sending $(h,g)\mapsto (hg, g)$.     The previous discussion produces an equivalence of categories $\ob{D}(BG)^{\heartsuit}\cong \ob{cRMod}_R(\ob{D}(A)^{\heartsuit})=\ob{cLMod}_{R^{\op}}(\ob{D}(A)^{\heartsuit})$ compatible with the forgetful functor towards $A$-modules, proving that $U(A)=R^{\op}$ as wanted.

For (2), since $R$ is $A$-flat and the augmentation map $R\to A$ has finite tor amplitude, the diagram 
\[
\ob{D}(BG)=\ob{Tot}(\ob{D}(G^{\bullet}))
\]
restricts to an equivalence 
\begin{equation}\label{eqoj1okqwndkqwd}
\ob{D}_{-}(BG) = \ob{Tot}(\ob{D}_{-}(G^{\bullet}))
\end{equation}
on the left-bounded full subcategories (i.e. on eventually coconnective objects).  Part (2) then follows from the  argument of  \cite[Proposition A.1.2]{MannSix}.  Part (3) is clear from \eqref{eqoj1okqwndkqwd} after passing to left completions, and since the categories $\ob{D}(G^{\bullet})$ are left complete. 
\end{proof}

\begin{remark}\label{RemarkLeftActionChange}
In \Cref{LemmaRepTheoryFlat} we have identified $\ob{D}(BG)$ with right $R$-comodules on $\ob{D}(A)$. The inverse map produces an isomorphism of groups $G\xrightarrow{\sim} G^{\op}$ between $G$ and its reversed group structure where multiplication is switched. This produces an equivalence of Hopf algebras $\iota\colon R\xrightarrow{\sim} R^{\op}$, which yields an isomorphism between left and right $R$-comodules. 
\end{remark}

\subsection{Cartier duality in symmetric monoidal categories}\label{ss:CartierDualitySix}

In this section we briefly recall the abstract Cartier duality of \cite[Section 3]{LurieElliptic}. Let $\n{C}$ be a small symmetric monoidal category, let $\ob{CAlg}(\n{C})$ be the category of commutative algebra objects in $\n{C}$, and let $\cat{Aff}_{\n{C}}:=\ob{CAlg}(\n{C})^{\op}$ be its opposite category, given $A\in \ob{CAlg}(\n{C})$ we write $\Spec^{\n{C}} A\in \cat{Aff}_{\n{C}}$ for its spectrum.  We work with the category $\ob{CMon}(\ob{PShv}(\cat{Aff}_{\n{C}}))=\ob{PShv}(\cat{Aff}_{\n{C}}, \cat{CMon})$ of presheaves on $\cat{Aff}_{\n{C}}$ valued in commutative monoids in anima. We let  $\ob{PShv}(\cat{Aff}_{\n{C}}, \Sp_{\geq 0}) \subset \ob{PShv}(\cat{Aff}_{\n{C}}, \cat{CMon})$ be the full subcategory of presheaves in connective spectra, equivalently, the full-subcategory of commutative  grouplike objects \cite[Definition 1.3.8]{LurieElliptic}. We see $\ob{PShv}(\cat{Aff}_{\n{C}}, \cat{CMon})$  endowed with the smashing symmetric monoidal tensor product. 

\begin{definition}\label{DefinitionAffineGl}
We let $\bb{A}^1_{\n{C}}$ be the functor in $\ob{PShv}(\cat{Aff}_{\n{C}}, \cat{CMon})$ sending an object $\Spec^{\n{C}} A$ to the monoid $\ob{Map}_{\n{C}}(1, A)$. We let ${\bf{GL}}_{1,\n{C}}\subset \bb{A}^1_{\n{C}}$ be the full sub-presheaf of invertible elements in $\ob{Map}_{\n{C}}(1, A)$. 
\end{definition}

The main result in the abstract  Cartier duality set up of \cite{LurieElliptic} is the following theorem. 

\begin{theorem}\label{ThmCartierDualitySymmetricmonoidal}
Define the Cartier  duality functor $\bb{D}\colon \ob{PShv}(\cat{Aff}_{\n{C}}, \cat{CMon})\to \ob{PShv}(\cat{Aff}_{\n{C}}, \cat{CMon})$ to be 
\[
\bb{D}(H):= \iHom(H, \bb{A}^1_{\n{C}})
\]
where the internal Hom is taking place in $\ob{PShv}(\cat{Aff}_{\n{C}}, \cat{CMon})$.  The following holds: 
\begin{enumerate}

\item Suppose that $H$ is grouplike, then the natural map 
\[
\iHom(H, {\bf{GL}}_{1,\n{C}})\to \iHom(H, \bb{A}^1_{\n{C}})
\]
is an equivalence and $\bb{D}(H)$ is grouplike. 

\item Let $\n{C}^{\ob{dual}}\subset \n{C}$ be the full subcategory of dualizable objects in $\n{C}$. Let $A\in \ob{bCAlg}(\n{C}^{\ob{dual}})$ be a commutative and cocommutative bialgebra in $\n{C}^{\ob{dual}}$ (\cite[Definition 3.3.1]{LurieElliptic}) with spectrum  $H:=\Spec^{\n{C}} A\in \ob{CMon}(\cat{Aff}_{\n{C}})$. Then  the  presheaf $\bb{D}(H)$ is corepresented by the algebra $A^{\vee}$, and the Cartier  duality functor restricts to the natural equivalence of categories
\begin{equation}\label{eqpmapsnfoawsd}
\ob{bCAlg}(\n{C}^{\ob{dual}}) \cong \ob{bCAlg}(\n{C}^{\ob{dual},\op})
\end{equation}
induced by the symmetric monoidal equivalence $\n{C}^{\ob{dual}}\cong \n{C}^{\ob{dual},\op}$ given  by mapping $A\mapsto A^{\vee}$. 

\item We say that an object $A\in \ob{cBAlg}(\n{C})$ is Hopf if $\Spec^{\n{C}} A$ is grouplike. Suppose that the underlying object $A$ is dualizable. Then $A$ is Hopf if and only if $A^{\vee}$ is Hopf. 

\item  Let $A, B\in \ob{bCAlg}(\n{C})$ with $A$ dualizable in $\n{C}$. Let $H:=\Spec^{\n{C}} A$ and $G=\Spec^{\n{C}} B$, and suppose that we are given with a linear pairing $\Psi\colon G\otimes H\to \bb{A}^1_{\n{C}}$. Consider  the natural map $G\to \bb{D}(H)$ induced by  $\Psi$, and let $\Psi^*\colon A^{\vee}\to B$ be its associated map of bialgebras. Let $\theta\colon 1\to A\otimes B$ be the map that corresponds to the composition 
\[
G\times H\to G\otimes H\to \bb{A}^1_{\n{C}}.
\]
Then $\Psi^*$ is the adjoint of $\theta$ under the natural equivalence $A\otimes B=\iHom_{\n{C}}(A^{\vee}, B)$.  In particular, $G\to \bb{D}(H)$ is an equivalence if and only if $\theta$ is the unit of a duality in $\n{C}$. 
\end{enumerate}
\end{theorem}
\begin{proof}
Part (1) is clear since ${\bf{GL}}_{1,\n{C}}$ is the largest submonoid of $\bb{A}^1_{\n{C}}$  which is grouplike. Part (2) is \cite[Propositions 3.8.1 and 3.8.5]{LurieElliptic}. Part (3) is \cite[Proposition 3.9.9]{LurieElliptic}. Finally, part (4) is \cite[Remark 3.8.4 and Proposition 3.8.5 (3)]{LurieElliptic}, but we give the details in the next paragraph.

By enlarging $\n{C}$ via Yoneda and the Day convolution if necessary, we can assume that $\n{C}$ admits countable colimits and that the symmetric monoidal structure commutes with countable colimits in each variable. This assumption on the tensor product is necessary for applying \cite[Proposition 3.8.1]{LurieElliptic}, but does not affect the conclusion for general $\n{C}$ thanks to the Yoneda embedding as both $\Spec^{\n{C}} A$ and $\Spec^{\n{C}} A^{\vee}$ belong to the essential image of $\n{C}$ in the presheaf category.

 Consider the prestack $[H, \bb{A}^1_{\n{C}}]$ on $\cat{Aff}_{\n{C}}$ sending $X$ to $\bb{A}^1_{\n{C}}(X\times H)$. If $X=\Spec^{\n{C}} C$ one has  that 
\[
\bb{A}^1_{\n{C}}(X\times H)= \Map_{\n{C}}(1, C\otimes A)= \Map_{\n{C}}(A^{\vee}, C)=\Map_{\ob{CAlg}(\n{C})}(\Sym_{\n{C}} A, C).
\]
Then, $[H, \bb{A}^1_{\n{C}}]$ is corepresented by the algebra $\Sym_{\n{C}} A^{\vee}$ and the composition  
\[
G\to \bb{D}(H)\to [H,\bb{A}^1_{\n{C}}]
\]
arises from morphism of algebras
\[
\Sym_{\n{C}} A^{\vee}\to A^{\vee}\xrightarrow{\Psi^*} B. 
\]
By definition, the composite $A^{\vee}\to \Sym_{\n{C}} A^{\vee} \to B$ is dual to the map $\theta$. Since the composite $A^{\vee}\to \Sym_{\n{C}} A^{\vee}\to A^{\vee}$ is the identity, we deduce that $\Psi^*$ is adjoint to $\theta$ as wanted. 
\end{proof}

\subsection{Cartier duality in a six functor formalism}
\label{ss:CartierDualitySixFunctors}

We now specialize the  discussion of Cartier duality from \Cref{ss:CartierDualitySix} to a six functor formalism.  We fix a regular cardinal $\kappa$ and assume that the six functor formalisms are $\kappa$-presentable as in \Cref{ss:PresentableKernels}, see \Cref{ConventionPresentable}.

Let $(\n{C},E)$ be a small geometric set up with finite limits and let $\ob{D}$ be a six functor formalism on $(\n{C},E)$.

\begin{construction}\label{ConstructionCartierDualityKernel}
 Let $S\in \n{C}$, recall from \Cref{RemarkFunctorsToCorrespondence} that we have a symmetric monoidal functor 
\[
(-)^*\colon (\n{C}_{/S}^E)^{\sqcup,\op}\to \Corr(\n{C}_{/S})\to \ob{K}_{\ob{D},S}\to \cat{Pr}_{\ob{D},S}.
\]
Composing with the natural map $(\n{C}_{/S}^{E})^{\op}\to \ob{CAlg}((\n{C}^E_{/S})^{\op,\sqcup})$ one gets a coproduct-preserving functor
\begin{equation}\label{eqImageAlgebrasPr}
[-]^*\colon (\n{C}_{/S}^{E})^{\op}\to  \ob{CAlg}(\ob{K}_{\ob{D},S}) \to \ob{CAlg}(\cat{Pr}_{\ob{D},S}).
\end{equation}
In other words, for $X\in \n{C}^E_{/S}$, we  also write $X\in \ob{Pr}_{\ob{D},S}$ when considered only as an object in the presentable category of kernels, and $[X]^*$ when considered with its natural commutative algebra structure arising from the map \eqref{eqImageAlgebrasPr}. More explicitly, let $f\colon X\to S$ be the structural map and $\Delta_X\colon X\to X\times_S X$ be the diagonal.   The unit $S\to [X]^*$ is given by the $*$-pullback  map $f^*\colon S\to X$, and  the multiplication $[X]^*\otimes  [X]^*\to [X]^*$ is given by $*$-pullback $\Delta^*_X\colon X\times_S X\to X$.

Since $\cat{Pr}_{\ob{D},S}$ is symmetric monoidal, the evaluation at the unit 
\[
\ob{Fun}_{S}(S, -)\colon \cat{Pr}_{\ob{D},S}\to \cat{Pr}^{\kappa}
\]
is lax symmetric monoidal. In particular, for $\s{A}\in \ob{CAlg}( \cat{Pr}_{\ob{D},S})$ a commutative algebra, the category $\ob{D}(\s{A}):=\ob{Fun}_{S}(S, \s{A})$ is naturally endowed with a symmetric monoidal structure. In case $\s{A}=[X]^*$ arises from an object $X\in \n{C}^E_{/S}$ via \eqref{eqImageAlgebrasPr}, this is nothing but the natural symmetric monoidal structure of the category $\ob{D}(X)$ arising from the six functor formalism.

Passing to  duals in the kernel category (cf. \Cref{RemarkFunctorsToCorrespondence}) and using its self duality (\cite[Proposition 4.1.4]{HeyerMannSix}), we get a functor 
\begin{equation}\label{eqShierkRealization}
[-]_!\colon  \n{C}_{/S}^E \to \ob{cCAlg}(\ob{K}_{\ob{D},S}). 
\end{equation}
Informally, $[X]_!$ is the cocommutative coalgebra with counit $[X]_!\to S$ given by the $!$-pushforward $f_!\colon X\to S$, and with comultiplication $[X]_!\to [X]_!\otimes [X]_!$ given by $\Delta_!\colon X\to X\times_S X$.

 Passing to cocommutative algebra objects in \eqref{eqImageAlgebrasPr} and \eqref{eqShierkRealization} we obtain  functors 
\[
[-]^*\colon (\ob{CMon}(\n{C}_{/S}^E))^{\op}\to  \ob{bCAlg}(\ob{K}_{\ob{D},S}),
\]
and
\[
[-]_!\colon \ob{CMon}(\n{C}_{/S}^E) \to \ob{bCAlg}(\ob{K}_{\ob{D},S}). 
\]
If $f\colon X\to S$ is a commutative algebra object with multiplication $m\colon X\times_S X\to X$ and unit $e\colon S\to X$, the coalgebra structure of $[X]^*$ arises from the $*$-functors $e^*\colon X\to S$ and $m^*\colon X\to X\times_S X$. Similarly, the algebra structure of $[X]_!$ arises from the $!$-functors $e_!\colon S\to X$ and $m_!\colon X\times_S X\to X$. In particular, the algebra structure of $\ob{D}([X]_!)$ arises from convolution along the $!$-pushforward of $m\colon X\times_S X\to X$. 
\end{construction}

\begin{remark}
The essential image of \eqref{eqImageAlgebrasPr} lands in the full subcategory of commutative algebras in dualizable objects in  $\cat{Pr}_{\ob{D},S}$. As  $\cat{Pr}_{\ob{D},S}$ is $\kappa$-presentable by design, the functor \eqref{eqImageAlgebrasPr} lands in the \textit{small} full subcategory of commutative algebras of $\kappa$-compact objects $\cat{Pr}_{\ob{D},S,\kappa}$.   Therefore, for the sake of applying the Cartier duality formalism of \Cref{ss:CartierDualitySix}, it will suffice to work with  $\cat{Pr}_{\ob{D},S,\kappa}$. 
\end{remark}

%Let $\lambda\geq \kappa$ be a (fixed and large enough) uncountable regular cardinal, and let us denote $\ob{Aff}_{\ob{D},S}^{\lambda}:= \ob{CAlg}(\cat{Pr}_{\ob{D},S,\lambda})^{\op}$.  
%
%\begin{remark}\label{RemarkDualizableObjects}
%By \Cref{LemmaAtomicPrkappa} the symmetric monoidal category $\cat{Pr}_{\ob{D},S}$ is $\kappa$-presentable. In particular, the dualizable objects of $\cat{Pr}_{\ob{D},S}$ are $\kappa$-compact, and therefore contained in $\cat{Pr}_{\ob{D},S,\lambda}$ for $\lambda\geq \kappa$. 
%\end{remark}

We denote $\ob{Aff}_{\ob{D},S}^{\kappa}:= \ob{CAlg}(\cat{Pr}_{\ob{D},S,\kappa})^{\op}$. Given $\s{A}\in \ob{CAlg}(\cat{Pr}_{\ob{D},S,\kappa})$ we write $\Spec^{S}(\s{A})\in \ob{Aff}_{\ob{D},S}^{\kappa}$ for its spectrum. In particular, if $X\in \n{C}^{E}_{/S}$, we write  $\Spec^{S} [X]^{*}\in \ob{Aff}_{\ob{D},S}^{\kappa}$.\footnote{One could just denote this object by $X$, but we prefer to make explicit the distinction in order to keep track in which category the object lands.}   Let $\bb{A}^1_{\ob{D,S}}\in \ob{PShv}(\ob{Aff}^{\kappa}_{\ob{D},S}, \cat{CMon})$ be the presheaf on commutative monoids  sending an object $\Spec^{S} \s{A}$ to the commutative monoid 
\[
\ob{D}(\s{A})^{\simeq}_{\kappa} \subset \ob{D}(\s{A})_{\kappa}
\]
consisting on the largest anima contained in the symmetric monoidal category   $\ob{D}(\s{A})_{\kappa}$ of $\kappa$-compact objects.  We let $\GL_{1,\ob{D},S}\subset \bb{A}^1_{\ob{D,S}}$ be the group-like submonoid representing those elements $\s{L}\in \ob{D}(\s{A})$ which are invertible (note that since the six functor formalism is $\kappa$-presentable, any invertible object in $\ob{D}(\s{A})$ is automatically $\kappa$-compact).

The following lemma describes the Cartier dual of a commutative grouplike object in $\n{C}^{E}_{/S}$.  

\begin{proposition}\label{PropGL1restriction}
Let $\bb{D}(-)\colon\ob{PShv}(\ob{Aff}^{\kappa}_{\ob{D},S}, \cat{CMon}) \to \ob{PShv}(\ob{Aff}^{\kappa}_{\ob{D},S}, \cat{CMon})$ be the Cartier duality functor of \Cref{ThmCartierDualitySymmetricmonoidal}.  Let  $\ob{K}_{\ob{D},S}\subset \cat{Pr}_{\ob{D},S}$ be the category of kernels given by the essential image of the functor $\cat{Corr}(\n{C}^{E}_{/S})\to  \cat{Pr}_{\ob{D},S}$. Let $\cat{Aff}_{\ob{K},S}:=\ob{CAlg}(\ob{K}_{\ob{D},S})^{\op}\subset \ob{Aff}^{\kappa}_{\ob{D},S}$.  Then the Cartier duality  functor $\bb{D}$ preserves the full subcategory  $\ob{CMon}(\cat{Aff}_{\ob{K},S})\subset \ob{PShv}(\ob{Aff}^{\kappa}_{\ob{D},S}, \cat{CMon})$ and its restriction is given by the (opposite of the) natural functor 
\begin{equation}\label{eqCartierDualityKernel}
\ob{bCAlg}(\ob{K}_{\ob{D},S})^{\op} \xrightarrow{(-)^{\vee}} \ob{bCAlg}(\ob{K}_{\ob{D},S}^{\op}) = \ob{bCAlg}(\ob{K}_{\ob{D},S})
\end{equation}
given by passing to the dual in the kernel category (and where in the second equality we use the natural self-duality of \cite[Proposition 4.1.4]{HeyerMannSix}).

In particular, the composite 
\[
  \ob{CMon}(\n{C}^{E}_{/S})^{\op} \xrightarrow{\Spec^{S} ([-]^*)} \ob{CMon}(\cat{Aff}_{\ob{K},S})^{\op} \xrightarrow{\bb{D}} \ob{CMon}(\cat{Aff}_{\ob{K},S})
\]
 (where the left arrow is induced from \eqref{eqImageAlgebrasPr} after taking $\Spec^{S}$) is given by the map 
 \[
\Spec^S ([-]_!) \colon \ob{CMon}(\n{C}^{E}_{/S})^{\op} \to   \ob{CMon}(\cat{Aff}_{\ob{K},S})
 \]
 arising from \eqref{eqShierkRealization} after taking $\Spec^S$.
\end{proposition}
\begin{proof}
By \Cref{ThmCartierDualitySymmetricmonoidal} (2) Cartier duality leaves stable bialgebras in dualizable objects of $\cat{Pr}_{\ob{D},S,\kappa}$. Since $\ob{K}_{\ob{D},S}$ is stable under passing to the dual, then Cartier duality also leaves stable bialgebras in $\ob{K}_{\ob{D},S}$ which yields \eqref{eqCartierDualityKernel}. The second claim that Cartier duality exchanges $[X]^*$ and $[X]_!$ follows from \Cref{ConstructionCartierDualityKernel} and \Cref{ThmCartierDualitySymmetricmonoidal}.  
\end{proof}

Let $X\in \ob{CMon}(\n{C}^E_{/S})$ be a commutative monoid object.  \Cref{PropGL1restriction} tells us that the Cartier dual of $[X]^*$ in the category of kernels is given by $[X]_!$ (in other words, it is the identity in the underlying object but exchanges $*$ and $!$ (co)multiplications). In many examples, at least those discussed in this paper, the Cartier dual of $X$ already exists in $\n{C}^E_{/S}$. Using \Cref{ThmCartierDualitySymmetricmonoidal} (4), the standard way to witness this is by constructing a pairing with values in  $\bb{A}^1_{\ob{D},S}$ (or $\GL_{1,\ob{D},S}$ if $X$ is grouplike). We finish this section by explaining how such pairing can  be constructed from $\ob{PShv}(\n{C}^E_{/S}, \cat{CMon})$, the category of commutative  monoid prestacks on $\n{C}^E_{/S}$.

 The natural functor  $F:=\Spec^S([-]^*)\colon \n{C}^E_{/S}\to \ob{Aff}_{\ob{D},S}^{\kappa}$ arising from \eqref{eqImageAlgebrasPr} preserves finite products (as $[-]^*$ is symmetric monoidal). Passing to left Kan extensions we obtain a colimit preserving and finite product preserving functor of presheaves categories 
\[
F_!\colon \ob{PShv}(\n{C}^{E}_{/S})\to \ob{PShv}(\ob{Aff}_{\ob{D},S}^{\kappa}).
\]
The   right adjoint is the restriction functor $F^*\colon \ob{PShv}(\ob{Aff}_{\ob{D},S}^{\kappa})\to \ob{PShv}(\n{C}^E_{/S})$, which also preserves limits and colimits.   In particular, both $F_!$ and $F^*$ are symmetric monoidal with respect to the cartesian symmetric monoidal structure, and passing to commutative monoid objects (or equivalently, taking $-\otimes_{\cat{Ani}} \cat{CMon}$ in $\cat{Pr}$) we obtain an adjunction 
\[
F_! \colon \ob{PShv}(\n{C}^E_{/S}, \cat{CMon}) \rightleftarrows \ob{PShv}(\ob{Aff}_{\ob{D},S}^{\lambda}, \cat{CMon})\colon F^*
\]
where both functors are symmetric monoidal with respect to the smashing tensor product of $\cat{CMon}$.

 The pullback $F^*\bb{A}^1_{\ob{D},S}$ is the presheaf on $\n{C}^{E}_{/S}$ sending an element $X$ to the anima  $\ob{D}(X)_{\kappa}$ of $\kappa$-compact objects endowed with the monoidal structure given by the tensor product of $\ob{D}(X)$. Similarly, $F^*\GL_{1,\ob{D},S}$ sends $X$ to the commutative monoid of invertible elements in $\ob{D}(X)$ with respect to the natural tensor product. 
 
Suppose that $X,Y\in \ob{CMon}(\n{C}^E_{/S})$ are commutative monoid objects and that we are given with a map of commutative monoids $h\colon X\otimes Y\to F^* \bb{A}^1_{\ob{D},S} $ in $\cat{PShv}(\n{C}^E_{/S}, \cat{CMon})$ (or a pairing $h\colon X\otimes Y\to F^* \GL_{1,\ob{D},S}$ in case $X$ and $Y$ are grouplike). Taking $F_!$, and using its symmetric monoidal structure for the smashing tensor product, we get a map of commutative monoids in $\cat{PShv}(\cat{Aff}^{\kappa}_{\ob{D},S}, \cat{CMon})$
\[
F_!h\colon \Spec^S [X]^* \otimes \Spec^S [Y]^*= F_!(X\otimes Y )\to F_!F^*\bb{A}^1_{\ob{D},S}\to \bb{A}^1_{\ob{D},S}
\] 
 (and a map $\Spec^S [X]^* \otimes \Spec^S [Y]^*\to \GL_{1,
 \ob{D},S}$ in case $X$ and $Y$ are grouplike). Applying \Cref{PropGL1restriction},  the map $F_!h$ induces a natural map of commutative monoids in $\cat{Aff}^{\kappa}_{\ob{D},S}$
 \[
 \Spec([X]^*)\to \Spec([Y]_!),
 \] or equivalently, a natural map of commutative bialgebras in $\ob{K}_{\ob{D},S}$
 \begin{equation}\label{eqFourierMukai}
\n{FM}_h \colon [Y]_!\to  [X]^*. 
 \end{equation}
 that we call the \textit{Fourier-Mukai transform} of $h\colon X\otimes Y\to F^*\bb{A}^1_{\ob{D},S}$. We have the following proposition.

\begin{proposition}\label{PropFourierMukai}
Keep the previous notation, let $X,Y\in \ob{CMon}(\n{C}^{E}_{/S})$ be commutative monoid objects, and  $h\colon X\otimes Y\to F^*\bb{A}^1_{\ob{D},S}$ a morphism of commutative monoids. Let $\n{FM}_h\colon [Y]_!\to [X]^*$ be the Fourier-Mukai transform  \eqref{eqFourierMukai} of commutative bialgebras in $\ob{K}_{\ob{D},S}$. Then the sheaf $\s{F}_{h}\in \ob{D}(Y\times_S X)$ associated to the underlying morphism of $\n{FM}_{h}$ in $\ob{K}_{\ob{D},S}$ is given by the natural sheaf associated to the composite $Y\times X\to Y\otimes X \to F^*\bb{A}^1_{\ob{D},S}$. 
\end{proposition}
\begin{proof}
Let $f\colon Y\to S$ and $g\colon X\to S$ be the structural maps of $Y$ and $X$.  The evaluation map $\ob{ev}_Y\colon Y\times_S Y \to S$ in $\ob{K}_{\ob{D},S}$ is given by the composite $f_! \Delta^*_Y$ associated to the correspondence $Y\times_S Y \xleftarrow{\Delta} Y\xrightarrow{f} S $.

 Let $\s{F}_h'\in \ob{D}(Y\times_S X)$ be the sheaf associated to the map $Y\times X\to Y\otimes X\to  F^*\bb{A}^1_{\ob{D},S}$. By \Cref{ThmCartierDualitySymmetricmonoidal} (4)  the map $\n{FM}_{h}\colon Y\to X$ of \eqref{eqFourierMukai} is given as the composite 
\[
Y \xrightarrow{\ob{id}\otimes \s{F}_h'} Y\times_S Y\times_S X \xrightarrow{\ob{ev}_Y} S\times_S X = X.
\]
Since $\ob{ev}_Y= \Delta_{Y,!} f^*$, a simple computation shows that this composite is nothing but the map $Y\to X$ associated to $\s{F}_h'\in \ob{D}(Y\times_S X)=\ob{D}(X\times_S Y)= \ob{Fun}_S(Y,X)$, proving what we wanted. 
\end{proof}

\begin{remark}\label{RemarkCartierDualityCompatibility}
Let $X,Y\in \ob{CMon}(\n{C}^E_{/S})$ be commutative monoid objects and $h\colon X\otimes Y\to F^*\bb{A}^1_{\ob{D},S}$ a pairing giving rise to a Fourier-Mukai transform 
\begin{equation}\label{eqFourierMukai2}
\n{FM}_{h} \colon [Y]_!\to   [X]^*
\end{equation}
given by the kernel $\s{F}_h\in \ob{D}(X\times_S Y)$. Then the dual of $\n{FM}_h$ gives rise to another morphism of commutative bialgebras
\[
\n{FM}_h^{\vee}\colon [X]_!\to [Y]^*
\]
whose underlying morphism is given by the same kernel $\s{F}_h\in \ob{D}(Y\times_S X)$  but exchanging the $!$ and $*$-maps. In particular, passing to underlying categories, the Fourier-Mukai transform gives rise to a morphism of commutative algebras 
\[
\n{FM}_h\colon \ob{D}([Y]_!)\to \ob{D}([X]^*)
\]
(resp. $\n{FM}_h^{\vee}\colon \ob{D}([X]_!)\to \ob{D}([Y]^*)$), where $\ob{D}([Y]_!)$ is endowed with the convolution product and $\ob{D}([X]^*)$ with the usual tensor product. In other words, the Fourier-Mukai transform sends the convolution tensor product to the usual tensor product. Similarly, if $f\colon X\to S$ is the structural map and $e\colon S\to Y$ is the unit map, then the Fourier Mukai transform preserves units, i.e. it makes the following diagram commute
\[
\begin{tikzcd}
\ob{D}(S) \ar[d,"e_!"'] \ar[rd, "f^*"] & \\ 
 \ob{D}(Y) \ar[r, "\n{FM}_h"']  & \ob{D}(X).
\end{tikzcd}
\] 
\end{remark}

\begin{remark}\label{RemarkDescentBaseChange}
We keep the notation of \Cref{RemarkCartierDualityCompatibility}.  
Let $g\colon S'\to S$ be a map in $\n{C}$, then the pullback map $g^*_1\colon \ob{K}_{\ob{D},S}\to \ob{K}_{\ob{D},S'}$ is symmetric monoidal, and it sends the Fourier-Mukai transform \eqref{eqFourierMukai} to a Fourier-Mukai transform
\[
g^*_1\n{FM}_h\colon [Y\times_S S']_!\to [X\times_S S']^*.
\]
In particular, if $\n{FM}_h$ is an isomorphism then so is $g^*_1 \n{FM}_h$.

Conversely, if $h\colon S'\to S$ satisfies universal $*$-descent, then the functor $g^*_1\colon \ob{K}_{\ob{D},S}\to \ob{K}_{\ob{D},S'}$ is conservative by \cite[Proposition 4.3.1]{HeyerMannSix}. Thus, if $g^*\n{FM}_h$ is an equivalence then so is $\n{FM}_h$. 
\end{remark}

\begin{definition}\label{DefinitionCartierDuals}
Let $X,Y\in \ob{CMon}(\n{C}^E_{/S})$ be commutative monoid objects and $h\colon X\otimes Y\to F^*\bb{A}^1_{\ob{D},S}$. We say that $h$ induces a \textit{$1$-categorical Cartier duality} if the Fourier-Mukai transform of \eqref{eqFourierMukai} is an isomorphism. If that is the case, we say that $Y$ is a \textit{$1$-categorical Cartier dual} of $X$.
\end{definition}

\begin{remark}\label{RemarkCartierDualityGestalten}
In \Cref{DefinitionCartierDuals}, we intentionally do not say that $Y$ is \textit{the} Cartier dual of $X$ as this statement does not make sense in the present discussion of Cartier duality; Cartier duality ought to be (at least) an antiequivalence of dualizable commutative bialgebras, and neither $X$ or $Y$ are dualizable in $\n{C}$.  When passing to the kernel category, the objects $[X]^*$ and $[Y]^*$ acquire the structure of  dualizable commutative  bialgebras, and only there it makes sense to discuss the notion of Cartier duality  as in  \Cref{ThmCartierDualitySymmetricmonoidal}. Of course, unless one has strong Tannaka duality properties, it is too naive to expect that the object $X\in \n{C}^E_{/S}$ can be recovered from its incarnation $[X]^*$ in the category of kernels. 

A far more satisfying and general Cartier duality that solves the aforementioned problems is the one in Gestalten discovered by Scholze and Stefanich where one not only encodes the Cartier duality at a $1$-categorical level as in \Cref{DefinitionCartierDuals}  but at \textit{all} higher categorical levels, we refer to   \cite{GestaltenScholze} and to their future work for this theory. 
\end{remark}

\subsection{Descent and Cartier duality}
\label{ss:DescentCartierDuality}

We finish this section with a discussion of how the passage to sheaves for a suitable analogue of the $!$-topology can be used to simplify some computations in Cartier duality. We will apply this discussion to the presentable category of kernels of a six functor formalism satisfying the  categorical K\"unneth formula, and hence by \Cref{PropKunnethKernel} to a $2$-category of linear categories. The following ideas are entirely motivated by the theory of Gestalten of Scholze and Stefanich, in particular we will borrow some of their terminology.

\subsubsection{Descent in $\cat{Pr}_{\n{V}}$}\label{sss:DescentPrV}

Let $\n{V}$ be a  $\kappa$-presentable symmetric monoidal stable category and let $\cat{Pr}_{\n{V}}:=\ob{Mod}_{\n{V}}(\cat{Pr}^{\kappa})$ be the category of $\kappa$-presentable symmetric monoidal $\n{V}$-linear categories. Consider $\ob{CAlg}(\n{V})$ the category of commutative algebras in $\n{V}$ and let $0\cat{Aff}_{\n{V}}$ be its opposite category, similarly, we let $1\cat{Aff}_{\n{V}}$ be the opposite category of $\ob{CAlg}(\cat{Pr}_{\n{V}})$. One has a fully faithful embedding $\ob{CAlg}(\n{V})\to \ob{CAlg}(\cat{Pr}_{\n{V}})$ given by $A\mapsto \ob{Mod}_A(\n{V})$  \cite[Corollary 4.8.5.21]{HigherAlgebra}. The categories $0\cat{Aff}_{\n{V}}$ and  $1\cat{Aff}_{\n{V}}$ admit small limits, and the inclusion  $0\cat{Aff}_{\n{V}}\subset  1\cat{Aff}_{\n{V}}$ preserves them.

Given $\n{A}\in \cat{CAlg}(\cat{Pr}_{\n{V}})$ we let $\Spec^{\n{V}} \n{A}\in 1\cat{Aff}_{\n{V}}$ be its associated object in the opposite category. Similarly, if $A\in \cat{CAlg}(\n{V})$ we denote $\Spec^{\n{V}} A\in 0\cat{Aff}_{\n{V}}$ for its corresponding object. Notice that we have a natural identification $\Spec^{\n{V}} A=\Spec^{\n{V}} ( \ob{Mod}_A(\n{V}))$ via the fully faithful inclusion $0\cat{Aff}_{\n{V}}\hookrightarrow 1\cat{Aff}_{\n{V}}$. Conversely, given $X\in 1\cat{Aff}_{\n{V}}$ we shall write $\ob{D}(X)$ for its corresponding object in $\ob{CAlg}(\cat{Pr}_{\n{V}})$, and let $\s{O}(X)=\ob{End}^{\n{V}}_{\ob{D}(X)}(1)$ denote the commutative algebra in $\n{V}$ given by endomorphisms of the unit in $\ob{D}(X)$. Finally, given $X\in 1\cat{Aff}_{\n{V}}$ we denote $\cat{Pr}_{X}:= \Mod_{\ob{D}(X)}(\cat{Pr}_{\n{V}})=\Mod_{\ob{D}(X)}(\cat{Pr}^{\kappa}_{\n{V}})$. 

To define the analogue of the $!$-topology in $1\cat{Aff}_{\n{V}}$ one needs the following key definitions (see \cite[Lecture VI]{GestaltenScholze}):

\begin{definition}\label{DefEtaleMaplinear}
Consider a map $f\colon Y\to X$ in $1\cat{Aff}_{\n{V}}$ such that $\ob{D}(X)\to \ob{D}(Y)$ is a $\kappa$-presented morphism.

\begin{enumerate}

\item We say that $f$ is  $1$-suave if the pullback functor 
\[
f^*_{1}\colon \cat{Pr}_{X}\to \cat{Pr}_{Y}
\]
admits a $\cat{Pr}_{X}$-linear left adjoint $f_{1,\sharp}\colon \cat{Pr}_{Y}\to  \cat{Pr}_{X}$. 

\item We say that $f$ is $1$-\'etale if any diagonal $Y\to Y^{S^n/X}$ (with $S^n$ the $n$-th sphere) is $1$-suave.\footnote{By definition, $Y^{S^n/X}$  is the limit $\varprojlim_{S^n} Y$ of the constant $S^n$-diagram given by $Y$ in the slice category of $1\cat{Aff}_{\n{V},/X}$. }

\item We say that $f$ is $0$-suave if it is $1$-\'etale and the natural  map $f_{!}\colon f_{1,\sharp}\ob{D}(Y)\to \ob{D}(X)$  in $\cat{Pr}_{X}$ admits a right adjoint in $\cat{Pr}_X$.  If this holds for all diagonals of $f$, we say that it is $0$-\'etale.

\item We say that $f$ is $0$-prim if the pullback map $f^*\colon \ob{D}(X)\to \ob{D}(Y)$ admits a right adjoint $f_{*}$ in $\cat{Pr}_{X}$.  If this holds  for all diagonals of $f$ we say that it is $0$-proper. 
\end{enumerate}
\end{definition} 

\begin{remark}\label{Remark1Primmaps}
\begin{enumerate}

\item In \Cref{DefEtaleMaplinear} (3), the natural map $f_{!}\colon f_{1,\sharp} \ob{D}(Y)\to \ob{D}(X)$ arises as follows. Since $f$ is $1$-\'etale, the pullback $f_1^*\colon \cat{Pr}_X\to \cat{Pr}_Y$ has a linear left adjoint $f_{1,\natural}$. Thus, $f_{1,\natural}f^*_1$ is a comonad and $f_!$ is the   counit map $f_{1,\natural} \ob{D}(Y)=f_{1,\natural} f^*_1\ob{D}(X)\to \ob{D}(X)$. The $\ob{D}(X)$-dual of $f_!$ is nothing but the pullback map $f^*\colon \ob{D}(X)\to \ob{D}(Y)$. In particular, if $f$ is $0$-suave then $f^*$ also admits a linear left adjoint $f_{\sharp}$ which is nothing but the dual of the right adjoint $f^!$ of $f_!$.

\item The definition of $0$ and $1$-suave map of \Cref{DefEtaleMaplinear} is compatible with those of \cite[Definition 6.16]{GestaltenScholze}.  Indeed, using the language of \textit{loc. cit.},  we are working with $1$-affine Gestalten over $\n{V}$ which by \cite[Proposition 6.10 (i)]{GestaltenScholze}  are $1$-proper over $\n{V}$. By \cite[Proposition 6.21]{GestaltenScholze}  any $\kappa$-presented morphism in $1\ob{Aff}_{\n{V}}$ is $2$-\'etale. In particular, the pullback map $f_2^*\colon 2\cat{Pr}_X\to 2\cat{Pr}_Y$ of presentable $2$-categories satisfies ambidexterity so that  the functor $f_{2,*}\colon  2\cat{Pr}_Y\to 2\cat{Pr}_X$ is both a left and right adjoint of $f_2^*$, and $f_{2,*} 1_{2\cat{Pr}_Y}= \cat{Pr}_{Y}\in 2\cat{Pr}_X$ is dualizable. Thus, thanks to  the discussion after \cite[Definition 6.16]{GestaltenScholze},   the requirement of $f$ being $1$-suave is equivalent to $f_1^*\colon \cat{Pr}_X\to \cat{Pr}_Y$ admitting a linear left adjoint. The condition of being $0$-suave is then identical to that of \textit{loc. cit}.

\item Following  \Cref{DefEtaleMaplinear}, any $\kappa$-presented map $f\colon Y\to X$ in $1\cat{Aff}_{\n{V}}$ is $1$-prim. More precisely, the pullback $f_{1}^*\colon \cat{Pr}_{X}\to \cat{Pr}_{Y}$ has a linear right adjoint  $f_{1,*}$ given by the forgetful functor as $\cat{Pr}_{Y}=\ob{Mod}_{\ob{D}(Y)}(\cat{Pr}_{X})$. Since this holds for all diagonals of $f$, we say that $f$ is $1$-proper.  On the other hand, by \Cref{LemmaProperties1etaleMaps} down below any $\kappa$-presented map $f\colon Y\to X$ in $0\cat{Aff}_{\n{V}}$ is $1$-\'etale, and the pullback $f^*\colon \ob{D}(X)\to \ob{D}(Y)$ has a linear right adjoint $f_*$  as $\ob{D}(Y)=\ob{Mod}_{\s{O}(X)}\ob{D}(X)$ (this right adjoint preserves $\kappa$-compact objects so it is a right adjoint in $\cat{Pr}_X$). Hence, $f$ is $0$-prim, and since this holds after passing to diagonals it is $0$-proper. 

\item For more examples of properties of morphisms of Gestalten we refer to \cite[Lectures 6-9]{GestaltenScholze}. 
\end{enumerate}

\end{remark}

The following lemma gives some formal properties of the maps introduced in \Cref{DefEtaleMaplinear}.

\begin{lemma}\label{LemmaProperties1etaleMaps}
The following hold:

\begin{enumerate}

\item Let $f\colon Y\to X$ be a map in $0\cat{Aff}_{\n{V}}$ such that $\s{O}(X)\to \s{O}(Y)$ is $\kappa$-presented, then $f$ is $1$-\'etale.  Moreover, let $f^*_{1}\colon \cat{Pr}_{X}\to \cat{Pr}_{Y}$ be the pullback map, then there is a natural identification of functors $f_{1,\sharp}=f_{1,*}$ between the left and right adjoint of $f^*_1$. 

\item Let $(\n{C},E)$ be a small geometric set up with finite limits, and let $\ob{D}$ be a $\kappa$-presentable six functor formalism on $(\n{C},E)$ with values in $\cat{Pr}_{\n{V}}$. Let $X\in \n{C}$ and suppose that $\ob{D}_X\colon \n{C}^{E,\op}_{/X}\to \cat{Pr}_{\ob{D}(X)}$ is K\"unneth, i,e, preserves finite coproducts.  Then for all $!$-able map $Y\to X$, the map $\Spec^{\n{V}} \ob{D}(Y)\to \Spec^{\n{V}} \ob{D}(X)$ is $1$-suave in  $\cat{Pr}_{\n{V}}$. If in addition any diagonal of $Y\to X$ is K\"unneth, \[\Spec^{\n{V}} \ob{D}(Y)\to \Spec^{\n{V}} \ob{D}(X)\] is $1$-\'etale.

\item Let $f\colon Y\to X$ be a $1$-suave (resp. $1$-\'etale or   $0$-suave, prim, \'etale and proper maps) in $1\cat{Aff}_{\n{V}}$. Then the same holds for any base change $f'\colon Y'\to X'$ with $g\colon X\to X'$ a map in $1\cat{Aff}_{\n{V}}$.

\item Consider a diagram $Z\xrightarrow{g} Y \xrightarrow{f} X$ in $1\cat{Aff}_{\n{V}}$, if $f$ and $g$ are $1$-suave (resp. $0$-suave or $0$-prim) then so is $f\circ g$.

\item Consider a diagram $Z\xrightarrow{g} Y \xrightarrow{f} X$ in $1\cat{Aff}_{\n{V}}$.  Suppose that $f$ is $1$-\'etale (resp. $0$-proper or $0$-\'etale). Then $g$ is $1$-\'etale (resp. $0$-proper or $0$-\'etale) if and only if $f\circ g$ is $1$-\'etale (resp. $0$-proper or $0$-\'etale).

\end{enumerate}

\end{lemma}
\begin{proof}

\begin{enumerate}

\item Consider the functor $\ob{D}\colon 0\cat{Aff}_{\n{V}}^{\op} \to \cat{CAlg}(\cat{Pr}_{\n{V}})$ mapping $X$ to $\ob{D}(X)$. Note that for any $f\colon Y\to X$ in $0\cat{Aff}_{\n{V}}$ such that $\s{O}(X)\to \s{O}(Y)$ is $\kappa$-compactly presented,  the pullback map $f^*\colon \ob{D}(X)\to \ob{D}(Y)$ has a linear right adjoint $f_*$ in $\cat{Pr}^{\kappa}$, and this satisfies base change.  Letting $E=P$ be $\kappa$-presented maps,  and letting $I$ be equivalences, by \cite[Proposition 3.3.3]{HeyerMannSix} we can promote $\ob{D}$ to a six functor formalism on $0\cat{Aff}_{\n{V}}$ with all $\kappa$-presented maps being $!$-able. It is obvious that $\ob{D}$ is K\"unneth, hence part (1) follows from part (2).  The identification of $f_{1,\sharp}$ and $f_{1,*}$ follows from \Cref{RemarkImprovementStefanich}.

\item For (2), since $\ob{D}_X$  is K\"unneth, the functor 
\[
\cat{Fun}_X(X,-)\colon \ob{K}_{\ob{D},X}\to \cat{Pr}_{\ob{D}(X)}
\]
is symmetric monoidal by \Cref{PropKunnethKernel}. Since all the objects of $\ob{K}_{\ob{D},X}$ are self dual, the same holds for $\ob{D}(Y)\in \cat{Pr}_{\ob{D}(X)}$ and  $Y\to X$ a $!$-able map. Hence, the base change 
\[
f_1^*\colon \cat{Pr}_{\ob{D}(X)}\to \cat{Pr}_{\ob{D}(Y)}
\]
is naturally identified with the functor $\cat{Fun}_{\ob{D}(X)}(\ob{D}(Y),-)$, proving that it admits a linear left adjoint $f_{1,\sharp}$ as wanted. Note that by the ambidexterity of  \Cref{RemarkImprovementStefanich} one can even identify $f_{1,\natural}$ with $f_{1,*}$.   Now, if in addition all diagonal of $Y\to X$ is K\"unneth, then for $f\colon Y\to X$ and any $n\in \bb{N}$, one has  that 
\[
\Spec^{\n{V}} \ob{D}(Y^{S^n/X})= \Spec (\ob{D}(Y))^{S^n/\Spec^{\n{V}} \ob{D}(X)}
\]
where $S^n$ is the $n$-th sphere. By the previous discussion we know that $\Spec^{\n{V}} \ob{D}(Y)\to  \Spec^{\n{V}} \ob{D}(X^{S^n/X})$ is $1$-suave, this proves that $\Spec^{\n{V}} \ob{D}(Y)\to \Spec^{\n{V}} \ob{D}(X)$ is $1$-\'etale as wanted.

\item Suppose that $f\colon Y\to X$ is $1$-suave, and let $g\colon X'\to X$ with base change. Then one has that 
\[
\cat{Pr}_{Y'}=\cat{Pr}_{Y}\otimes_{\cat{Pr}_X} \cat{Pr}_{X'}
\]
in $2\cat{Pr}^{\kappa}$. Thus, if $f_1^*\colon \cat{Pr}_{X}\to \cat{Pr}_{Y}$ has a linear left adjoint then so does its base change to $\cat{Pr}_{X'}$, proving the stability of $1$-suave maps under base change. The same argument holds for $1$-\'etale maps since the base change commutes with diagonals, and similarly for $0$-proper and $0$-prim maps. For $0$-suave maps, the left adjoint $f_{1,\natural}\colon \cat{Pr}_Y\to \cat{Pr}_X$ base changes to the left adjoint $f'_{1,\natural}\colon \cat{Pr}_{Y'}\to \cat{Pr}_{X'}$. In particular, evaluating the counit at $1$, one has that the map $f'_!\colon f_{1,\natural}' \ob{D}(Y')\to \ob{D}(X')$ is the base change along $\ob{D}(X)\to \ob{D}(X')$ of the map $f_!\colon f_{1,\natural}\ob{D}(Y)\to \ob{D}(X)$, in particular $f'_!$ admits a right adjoint if $f_!$ does so.

\item To see that $1$-suave maps are stable under composition, consider the induced  pullback maps 
\begin{equation}\label{eqo9ojnkaosndasdlm}
\cat{Pr}_{X}\xrightarrow{f_1^*} \cat{Pr}_{Y} \xrightarrow{g_1^*} \cat{Pr}_Z.
\end{equation}
By hypothesis both $f_1^*$ and $g_1^*$ admit a linear left adjoint, and then so does the composite. By (5) down below,  we know that $1$-\'etale maps are stable under composition (this only uses the stability of $1$-suave maps under pullbacks and composition which have already been proved). The case of $0$-prim maps is proven in the same way. 

Suppose now that $f$ and $g$ are $1$-suave, then they are $1$-\'etale and then so is its composite. Passing to left adjoints in \eqref{eqo9ojnkaosndasdlm} we find that the map $(f\circ g)_!\colon (f\circ g)_{1,\sharp}\ob{D}(Z)\to \ob{D}(X)$ is the composite 
\[
(f\circ g)_{1,\sharp}\ob{D}(Z)\xrightarrow{f_{1,\sharp} (g_!)} f_{1,\sharp} \ob{D}(Y)\xrightarrow{f_{!}} \ob{D}(X).
\]
Thus, if both functors $g_!$ and $f_!$ admit right adjoints, then so do the functor $f_{1,\sharp}(g_!)$ (as $f_{1,\sharp}$ is a $2$-functor and therefore it preserves adjoints) and $f_! \circ f_{1,\sharp}(g_!) = (f\circ g)_!$.

\item We first see that if $g$ and $f$ are $1$-\'etale, then so is $f\circ g$. For that, let $n\in \bb{N}$, we have a cartesian square in $1\cat{Aff}_{\n{V}}$
\[
\begin{tikzcd}
Z^{S^n/Y} \ar[r] \ar[d] & Z^{S^n/X} \ar[d]\\ 
Y \ar[r] & Y^{S^n/X}.
\end{tikzcd}
\]
Since $f$ is $1$-\'etale, the map $Y\to Y^{S^n/X}$ is $1$-suave and then so is $Z^{S^n/Y}\to Z^{S^n/X}$ by base change and part (3). Since $g$ is $1$-\'etale the map $Z\to Z^{S^n/Y}$ is $1$-suave, proving that $Z\to  Z^{S^n/Y}\to  Z^{S^n/X}$ is $1$-suave by stability of $1$-suave maps under composition (4).  The same argument shows that $0$-proper and $0$-\'etale maps are stable under composition. 

Conversely, suppose that $f\circ g$ is $1$-\'etale, we want to show that $g$ is $1$-\'etale. First note that by definition $1$-\'etale maps are stable under taking diagonals. Note that the map $Y\times_X Z\to Y$ is $1$-\'etale being the base change of $Z\to X$ along $Y\to X$. We also have a cartesian square 
\[
\begin{tikzcd}
Z \ar[r] \ar[d] & Y\times_X Z \ar[d] \\ 
Y \ar[r] & Y\times_X Y
\end{tikzcd}
\]
making the map $Z\xrightarrow{(g,\id_Z)} Y\times_X Z$ a $1$-\'etale map as $Y\to Y\times_X Y$ is $1$-\'etale. It follows that the composite $Z\to Y\times_X Z\to Y$ is $1$-\'etale, proving what we wanted. The same argument holds for $0$-\'etale and proper maps as they are stable under taking diagonals. 

\end{enumerate}

\end{proof}

We can now define the analogue of the $!$-topology for the objects in $1\cat{Aff}_{\n{V}}$. Following Scholze and Stefanich, this is the natural Grothendieck topology of Gestalten \cite{GestaltenScholze} when restricted to $1$-affine Gestalten.  

\begin{definition}\label{DefinitionGestaltTopology}
A $\kappa$-small family of $1$-\'etale maps $\{f_i\colon Y_i\to X\}$ in $1\cat{Aff}_{\n{V}}$ is called a  \textit{$1$-\'etale cover} if  the natural map 
\[
\cat{Pr}_{X}\xrightarrow{\sim} \varprojlim_{([n],\iota_{\bullet})\in \Delta^{\op}_{I}}  \cat{Pr}_{Y_{([n],i_{\bullet})}}
\]
is an equivalence, where $ \Delta_{I}$ is as in \cite[Definition A.4.5]{HeyerMannSix}, and $\{Y_{([n],i_{\bullet})}\}_{\Delta_{I}^{\op}}$ is the \v{C}ech nerve of $\{Y_i\to X\}$. 
\end{definition}

\begin{lemma}\label{LemmaEquivalentConditionsCover}
Let $\{f_i\colon Y_i\to X\}_{i\in I}$ be a $\kappa$-small family of $1$-\'etale maps in $1\cat{Aff}_{\n{V}}$, then the following are equivalent:

\begin{enumerate}

\item  The natural map 
\[
\cat{Pr}_{X}\to   \varprojlim_{([n],\iota_{\bullet})\in \Delta^{\op}_{I}}  \cat{Pr}_{Y_{([n],i_{\bullet})}}
\]
is an equivalence. 

\item The natural map 
\[
\cat{Pr}_{X}\to   \varprojlim_{([n],\iota_{\bullet})\in \Delta^{\op}_{I}}  \cat{Pr}_{Y_{([n],i_{\bullet})}}
\]
if $2$-fully faithful.

\item Let $f_{i,1}^*\colon \cat{Pr}_{X}\to \cat{Pr}_{Y_i}$ be the pullback map and let $f_{i,1,\sharp}\colon \cat{Pr}_{Y_i}\to \cat{Pr}_{X}$ be its linear left adjoint. For $([n],i_{\bullet})\in \Delta_{I}$ let $f_{([n],i_{\bullet})}\colon Y_{([n],i_{\bullet})} \to X$ be the structural map.  Then the natural map 
\[
\varinjlim_{\Delta_I^{\op}} f^{\bullet}_{([n],i_{\bullet}),1,\sharp}\ob{D}(Y_{([n],i_{\bullet})})  \to \ob{D}(X) 
\]
is an equivalence in $\cat{Pr}_{\n{V}}$

\end{enumerate}
\end{lemma}
\begin{proof}
It is clear that (1) implies (2).   For (2) if and only if (3), let $G\colon \cat{Pr}_{X}\to   \varprojlim_{([n],\iota_{\bullet})\in \Delta^{\op}_{I}}  \cat{Pr}_{Y_i}$ be the natural map, since each composite $\cat{Pr}_{X}\to \cat{Pr}_{Y_{([n],i_{\bullet})}}$ has a linear left adjoint $f_{([n],i_{\bullet}),1,\sharp}$, then so does the functor $F$ (by \cite[Lemma D.4.7 (i)]{HeyerMannSix}) and it is  given by 
\[
F=\varinjlim_{\Delta^{\op}_I} f_{([n],i_{\bullet}),1,\sharp} \colon  \varprojlim_{([n],\iota_{\bullet})\in \Delta^{\op}_{I}}  \cat{Pr}_{Y_{([n],i_{\bullet})}}\to \cat{Pr}_X. 
\]
Thus,  $G$ is fully faithful if and only if $FG\to \id_{\cat{Pr}_{X}}$ is an equivalence. Since both $G$ and $F$ are $\cat{Pr}_X$-linear, this holds if and only if it does after evaluating at the unit, i.e. if and only if $FG(\ob{D}(X))\to \ob{D}(X)$ is an equivalence.

For (2) implies (1), let $(M_{([n],i_{\bullet})})_{\Delta_I}$ be a cocartesian section of $\varinjlim_{\Delta_I^{\op}} f^{\bullet}_{([n],i_{\bullet}),1,\sharp}\ob{D}(Y_{([n],i_{\bullet})}) $. We want to see that the natural map $GF((M_{([n],i_{\bullet})})_{\Delta_I})\to (M_{([n],i_{\bullet})})_{\Delta_I}$ is an equivalence. This boils down to proving that for all $([n],i_{\bullet})\in \Delta_{I}$ one has that 
\begin{equation}\label{eqo1mk01majd}
\big(\varinjlim_{\Delta_I^{\op}} f_{([n],i_{\bullet}),1,\sharp} M_{([n],i_{\bullet})} \big) \otimes_{\ob{D}(X)} \ob{D}(Y_{(m,j_{\bullet})}) \to M_{([m],j_{\bullet})}
\end{equation}
is an equivalence. But 
\[
\big(\varinjlim_{\Delta_I^{\op}} f_{([n],i_{\bullet}),1,\sharp} M_{([n],i_{\bullet})} \big) \otimes_{\ob{D}(X)} \ob{D}(Y_{(m,j_{\bullet})})= \varinjlim_{\Delta_I^{\op}}   f_{([n+m+1], i_{\bullet}\star j_{\bullet})} M_{([n+m+1], i_{\bullet}\star j_{\bullet})}
\]
(with $i_{\bullet}\star j_{\bullet}=(i_{0},\ldots, i_n,j_0,\ldots, j_m)$) since $M_{([n],i_{\bullet})}$ is cocartesian. This last diagram is split with colimit $M_{([m],j_{\bullet})}$, proving that \eqref{eqo1mk01majd} is an equivalence as wanted. 
\end{proof}

\begin{remark}\label{RemAnalyticRingsTop}
Let $\cat{AnRing}$ be the category of analytic rings. The functor of quasi-coherent sheaves
\[
\ob{D}\colon \cat{AnRing}\to \cat{Pr}_{\ob{D}(\Z^{\cond})}
\]
preserves finite colimits by \cite[Proposition 4.1.14]{SolidNotes}. Hence, if $f\colon A\to B$ is a $!$-able map of analytic rings, \Cref{LemmaProperties1etaleMaps} (2) implies that $\ob{D}(A)\to \ob{D}(B)$ is a $1$-\'etale map in $\cat{Pr}_{\ob{D}(\Z^{\cond})}$ (up to restricting to $\kappa$-presentable six functors). Furthermore, \Cref{LemmaEquivalentConditionsCover} together with \cite[Lemma 4.2.7]{HeyerMannSix} imply that $f$ is a $!$-cover of analytic rings if and only if $\ob{D}(A)\to \ob{D}(B)$ induces a $1$-\'etale cover in $1\cat{Aff}_{\ob{D}(\Z^{\cond})}$. 
\end{remark}

\begin{definition}\label{DefinitionEtaleTop}
We define the $1$-\'etale topology on $1\cat{Aff}_{\n{V}}$ to be the Grothendieck topology where the covering sieves  of an object $X\in 1\cat{Aff}_{\n{V}}$  are those that contain  a $1$-\'etale cover as in \Cref{DefinitionGestaltTopology}. 
\end{definition}

\begin{remark}\label{RemarkBasicProperties1etTopology}
The $1$-\'etale topology is a well defined Grothendieck topology thanks to the stability of $1$-\'etale maps under composition and base change of \Cref{LemmaProperties1etaleMaps}.
\end{remark}

\begin{lemma}\label{LemmaSubcanonicalDescent}
 Let $\{f_i\colon Y_i\to X\}_{I}$ be a $1$-\'etale cover in $1\cat{Aff}_{\n{V}}$.  Then the following holds:

\begin{enumerate}

\item  The cover $\{f_i\}$ is subcanonical. More precisely, the natural map $\ob{D}(X)\to \varprojlim_{\Delta_{I}} \ob{D}(Y_{([n],i_{\bullet})})$ with transition maps given by $*$-pullbacks is an equivalence of presentable categories. 
 
\item The natural map of $\kappa$-compact objects
\[
\ob{D}(X)_{\kappa}\xrightarrow{\sim }\varprojlim_{\Delta_{I}} \ob{D}(Y_{([n],i_{\bullet})})_{\kappa}
\]
is an equivalence. 

\item Let $\ob{D}(X)^{\ob{dual}}\subset \ob{D}(X)$ be the full subcategory of dualizable objects. Then the natural map 
\[
\ob{D}(X)^{\ob{dual}}\xrightarrow{\sim }\varprojlim_{\Delta_{I}} \ob{D}(Y_{([n],i_{\bullet})})^{\ob{dual}}
\]
is an equivalence. 
 
 \end{enumerate}
\end{lemma}
\begin{proof}
By definition, we have an equivalence of presentable $2$-categories 
\[
G\colon \cat{Pr}_X\xrightarrow{\sim} \varprojlim_{([n],\iota_{\bullet})\in \Delta^{\op}_{I}}  \cat{Pr}_{Y_{([n],i_{\bullet})}}. 
\]
In particular, looking at the right adjoint $H$ of $G$, we have the equivalence $\ob{D}(X)\xrightarrow{\sim} HG (\ob{D}(X))$ which translates precisely to (1). For part (2), since $\cat{Pr}_{\n{V}}=\Mod_{\n{V}}(\cat{Pr}^{\kappa})$, pullback maps preserve $\kappa$-compact objects by definition, and therefore (1) restricts to an equivalence as in (2). Finally, for (3), we clearly have a fully faithful map 
\[
\ob{D}(X)^{\ob{dual}}\hookrightarrow \varprojlim_{\Delta_{I}} \ob{D}(Y_{([n],i_{\bullet})})^{\ob{dual}}.
\]
It is left to see that if $M\in \ob{D}(X)$ is such that its pullback to $\prod_{i} \ob{D}(Y_i)$ is dualizable, then $M$ is so. Let $f_{([n],i_{\bullet})}\colon Y_{([n],i_{\bullet})}\to X$ be the structural map of the \v{C}ech nerve of $\{Y_{i}\to X\}$, and let $M_{([n],i_{\bullet})}=f^*_{([n],i_{\bullet})} M$. Then the objects $M_{([n],i_{\bullet})}$ are dualizable, and their duals $M_{([n],i_{\bullet})}^{\vee}$ form a cocartesian section of $\varprojlim_{\Delta_{I}} \ob{D}(Y_{([n],i_{\bullet})})^{\ob{dual}}$, this cocartesian section descends to an object $N\in \ob{D}(M)$. Furthermore, the unit and counit witnessing the duality of $M_{([n],i_{\bullet})}^{\vee}$ and $M_{([n],i_{\bullet})}$ are cocartesian and descent in a unit and counit map $1\to N\otimes M \to 1$ in $\ob{D}(X)$.  It is clear that these yield the unit and counit of a duality by descent.   
\end{proof}

When restricted to morphisms in $0\cat{Aff}_{\n{V}}$, a $1$-\'etale cover is the same as a descendable cover, provided the unit in $\n{V}$ is compact. 

\begin{lemma}\label{LemmaComparing1etaleDescendable}
Let $\{f_i\colon Y_i\to X\}_{i\in I}$ be a family of maps in $0\cat{Aff}_{\n{V}}$ such that $\s{O}(X)\to \s{O}(Y_i)$ is $\kappa$-presented. Suppose that the unit in $\n{V}$ is compact.  Then $\{f_i\}$ is a $1$-\'etale cover if and only if there is a finite subfamily $Y_{i_1},\ldots, Y_{i_n}\to X$ that forms a descendable cover of $X$, that is, such that the map of algebras $\s{O}(X)\to \prod_{k=1}^n \s{O}(Y_{i_k})$  is descendable in $\n{V}$. 
\end{lemma}
\begin{proof}
By \cref{LemmaProperties1etaleMaps} (1), given $g\colon Z\to W$ a $\kappa$-presentable map in $0\cat{Aff}_{\n{V}}$, the left and right adjoints $g_{1,\sharp}$ and $g_{1,*}$ of $g_1^*\colon \cat{Pr}_{W}\to \cat{Pr}_Z$ are naturally identified. Therefore, by \Cref{LemmaEquivalentConditionsCover} (3), $\{f_i\}$ is a $1$-\'etale cover if and only if the natural map
\[
\varinjlim_{\Delta_{I}^{\op}} \ob{D}(Y_{([n],i_{\bullet})})\to \ob{D}(X)
\]
is an equivalence in $\cat{Pr}^{\kappa}$, where the transition maps are given by forgetful functors. Passing  to right adjoints, this is equivalent for the natural functor
\begin{equation}\label{eqpjapsf32pajhfnolqwrkq}
\ob{D}(X)\to {\varprojlim_{\Delta_I}}^!\ob{D}(Y_{([n],i_{\bullet})})
\end{equation}
to be an equivalence, where transition maps are upper $!$-functors (which by definition are nothing but the right adjoints to the forgetful functors, see the six functor formalism constructed in the proof of (1) of \Cref{LemmaProperties1etaleMaps}). Suppose that \eqref{eqpjapsf32pajhfnolqwrkq} holds, then $\s{O}(X)=\varinjlim_{\Delta_I} \iHom_{\s{O}(X)}(\s{O}(Y_{([n],i_{\bullet})}),\s{O}(X))$, and since the unit of $\n{V}$ is compact, then so is the unit of $\ob{D}(X)=\Mod_{\s{O}(X)}(\n{V})$, and there exists a finite subset $I'\subset I$ and an index $m$ such that $\s{O}(X)$ is a retract of $\varinjlim_{\Delta_{I,\leq m}} \iHom_{\s{O}(X)}(\s{O}(Y_{([n],i_{\bullet})}),\s{O}(X))$. In particular, $\s{O}(X)$ is in the thick tensor ideal in $\ob{D}(X)$ generated by the algebra $\prod_{i\in I'} \s{O}(Y_i)$, proving that $\{Y_{i'}\to X\}$ is a descendable cover. Conversely, if there is a finite subset $I'\subset I$ such that $\{Y_{i}\to X\}_{i\in I'}$ is a descendable cover of $X$, then by \cite[Corollary 3.42]{MathewDescent} the natural functor of \Cref{LemmaEquivalentConditionsCover} (2) is an equivalence for $I'$ instead of $I$, and then so is for $I$ being a subcover. 
\end{proof}

\begin{example}\label{ExampleDifferentMaps}
In this paragraph we give different examples and counter-examples of $1$-\'etale, $1$-suave, $0$-suave, $0$-prim, etc. maps in $1\cat{Aff}_{\n{V}}$.  These are equivalent to the analogue definitions of \cite{GestaltenScholze}  when restricted to $1$-affine Gestalten over $\cat{Gest}(\n{V})$. For sake of concreteness, we will take $\n{V}=\ob{D}(R)$ for $R$ some analytic ring, and work with categories that arise from analytic stacks over $R$ via the quasi-coherent six functor formalism, see \Cref{ss:AnStk}. 

\begin{enumerate}

\item  By \cite[Proposition 9.5]{GestaltenScholze} one can construct Gestalten from six functor formalisms; this procedure consists in taking higher and higher categories of kernels in an inductive way. In particular, we can apply this to the quasi-coherent six functor formalism of analytic stacks, and given $X$ an analytic stack we let $\ob{Gest}(X)$  denote its associated Gestalt.   By \cite[Proposition 9.5]{GestaltenScholze}, any $!$-able map $Y\to X$ of analytic stacks produces a $1$-\'etale and $1$-proper map of Gestalten $\ob{Gest}(Y)\to \ob{Gest}(X)$. Moreover, \cite[Corollary 9.8]{GestaltenScholze} says that if a map $Y\to X$ of analytic stacks is such that $Y=\varinjlim_i Y_i$ is a (countable) colimit of $!$-able analytic stacks over $X$, then $\ob{Gest}(Y)\to \ob{Gest}(X)$ is still $1$-\'etale.\footnote{  The countability assumption is only used due to the choice of $\aleph_1$-presentable categories in the formalism of \cite{GestaltenScholze}, working in Gestalten with larger cardinals allows us to take colimits by larger index categories. } 

In practice, many of the morphisms of interest in analytic geometry are colimits  of $!$-able maps of analytic stacks, for example, this is the case for algebraic geometry as discussed in  \cite[Lectures IV and V]{GestaltenScholze}, and also in the setup of Gelfand stacks of \cite{anschutz2025analytic}. Indeed,  in both of the previous situations the categories of stacks are constructed from $0$-affine objects which are in particular $0$-proper (by \cite[Proposition 6.13]{GestaltenScholze}) and therefore $1$-\'etale (by \cite[Proposition 6.21]{GestaltenScholze}).  

In conclusion, in practice most of all the relevant maps that occur in analytic and algebraic geometry are $1$-\'etale as Gestalten, and the corresponding maps of categories in $1\cat{Aff}_{\n{V}}$ are  $1$-\'etale as long as the corresponding analytic stacks are $1$-affine over $\AnSpec R$, that is, as long as $\ob{Gest}(X)=\ob{Gest}(\ob{D}(X))$.

\item Any morphism $f\colon Y\to X$ of  affinoid analytic stacks with induced analytic ring structure gives rise to a $0$-proper map $\ob{D}(X)\to \ob{D}(Y)$. Furthermore, a map $Y\to X$ of affinoid analytic stacks is $0$-proper if and only if it is $0$-prim if and only if it has the induced structure. This follows from the fact that $0$-primness implies that the forgetful functor $\ob{D}(Y)\to \ob{D}(X)$  is colimit preserving  $\ob{D}(X)$-linear by projective formula, and by the monadicity theorem that $\ob{D}(Y)=\ob{Mod}_{f_*1}(\ob{D}(X))$, this is precisely the definition of having the induced analytic ring structure, and these maps are $0$-proper by \Cref{Remark1Primmaps} (3).

\item  Consider the Gestalten associated to analytic stacks as in part (1).  Remark 9.6 of \cite{GestaltenScholze} says that 
 a map $f\colon Y\to X$ of analytic stacks is suave, prim, cohomologically \'etale and cohomologically proper if and only if the map of Gestalten $[f]\colon \ob{Gest}(Y)\to \ob{Gest}(X)$ is   $0$-suave, $0$-prim, $f$ is truncated and $[f]$ is $0$-\'etale, and $f$ is truncated and $[f]$ is $0$-proper respectively.  Thus, the same holds for the associated objects in $1\cat{Aff}_{\n{V}}$ as long as the Gestalten associated to the analytic stacks are $1$-affine over $\AnSpec R$.

\item  Following (3), any quasi-compact open immersion of schemes $U\subset X$ is $0$-affine, that is, this can be checked locally on $X$ so we can assume $X=\Spec A$ is affine. Then, $U$ is the complement of a vanishing locus of a finitely generated ideal $I=(f_0,\ldots, f_n)$, and $\ob{D}(U)$ is the spectrum of the idempotent $A$-algebra 
\[
A(U)=\varprojlim_{J\subset [n]} A_{f_{J}}
\]
where $J$ runs over finite subsets, and $A_{f_{J}}= \bigotimes_{j\in J} A_{f_j}$.  We deduce that quasi-compact open immersions are $0$-proper.  On the other hand, by (2) any affine map of schemes  is also $0$-proper. From this one can deduce that any qcqs maps of schemes $Y\to X$ is $0$-proper when considered as Gestalten, and hence in $1\cat{Aff}_{\ob{D}(\bb{Z})}$ as long as $X$ is qcqs by the $1$-affine result of \cite[Theorem 1.0.5]{stefanich2023tannakaduality1affineness}.

\item   We highlight that, contrary to what the intuition says, open immersions of schemes are not in general open immersions (i.e $0$-\'etale immersions) as Gestalten. For example, $f\colon \Spec \bb{Z}[T^{\pm}]\to \Spec \bb{Z}[T]$ is not an open immersion as $f^*$ does not preserve limits, and therefore $f$ cannot be $0$-suave. Indeed, by \cite[Lemma 7.6.6]{SolidNotes} a qcqs open immersion $U\subset X$ of schemes is $0$-suave if and only if $U$ is a clopen Zariski subspace of $X$. On the other hand, if one uses the solid incarnation of schemes into analytic stacks, see \cite[Corollary 5.4.17]{SolidNotes}, open immersions of schemes $U\to X$ give rise to $0$-\'etale maps of analytic stacks.

\item Let $R$ be a classical ring and  $f\colon X\to \Spec R$ a smooth proper variety. Then Serre duality for proper smooth schemes implies that $f$ is a $0$-suave map of Gestalten, and $1$-affineness of $X$ that the associated map in $1\cat{Aff}_{\ob{D}(R)}$ is also $0$-suave. Now, if $X$ is only a smooth morphism not necessarily proper, $\Delta_f\colon X\to X\times_{R} X$ is a local complete intersection, which by \cite[Lemma 6.3.5]{SolidNotes} is a $0$-suave map.  However, the second diagonal $X\to X^{S^2/R}$ is not longer $0$-suave but only $0$-proper. Indeed, let us analyse the case $X=\bb{A}^1_{R}$, then the second diagonal is just $\Spec R[\epsilon_1]$ where $\epsilon_1$ is a free generator in homological degree $1$ satisfying $\epsilon_1^2=0$, that is, 
\[
R[\epsilon_1]=R\otimes_{R[T]} R =\ob{Sym}^{\bullet}_{R} R[1]
\]
where $T\mapsto 0$ in both components. The map of rings $R[\epsilon_1]\to R$ is not suave since $R$ is not a perfect $R[\epsilon_1]$-module.

\item Another interesting source of examples of stacks with different suave and primness properties are classifying stacks. In \Cref{s:Examples} we will study  classifying stacks of multiplicative groups, discrete $\bb{Z}$-lattices, and different incarnations of vector bundles in algebraic and analytic geometry.

\end{enumerate}

\end{example}

\subsubsection{Cartier duality in $\cat{Pr}_{\n{V}}$}\label{sss:CartierDualityPrV}

We finish our discussion of Cartier duality with a general theorem that will be helpful in identifying Cartier duals of quasi-affinoid analytic  group stacks from a pairing under some suave or prim conditions. Consider the small full subcategory $1\cat{Aff}^{\kappa}_{\n{V}}\subset 1\cat{Aff}_{\n{V}}$ spanned by those algebras in $\cat{Pr}_{\n{V}}$ which are $\kappa$-compact,  equivalently, we let $1\cat{Aff}^{\kappa}_{\n{V}}$ be the opposite of the category of commutative algebras in the full subcategory $\cat{Pr}_{\n{V},\kappa}$ of $\kappa$-compact objects in $\n{V}$.

 Let  $\ob{Shv}(1\cat{Aff}_{\n{V}}^{\kappa})$ be the category of sheaves with respect to the $1$-\'etale topology of \Cref{DefinitionEtaleTop}, we let $\ob{Shv}(1\cat{Aff}_{\n{V}}^{\kappa},\Sp_{\geq 0})\subset \ob{Shv}(1\cat{Aff}_{\n{V}}^{\kappa},\cat{CMon})$ be the categories of connective spectra and commutative monoids on $\ob{Shv}(1\cat{Aff}_{\n{V}}^{\kappa})$. 

\begin{lemma}\label{LemmaSheafA1GL1}
Let $\bb{A}^1_{\n{V}}$ be the presheaf on $1\cat{Aff}_{\n{V}}^{\kappa}$ given by $X\mapsto \ob{D}(X)_{\kappa}$, and let ${\bf{GL}}_{1,\n{V}}\subset \bb{A}^1_{\n{V}}$ be the subsheaf of invertible objects, see \Cref{DefinitionAffineGl}. Then $\bb{A}^1_{\n{V}}$ and  ${\bf{GL}}_{1,\n{V}}$ are sheaves for the $1$-\'etale topology.  In particular, if $G\in \cat{PShv}(1\cat{Aff}_{\n{V}}^{\kappa},\cat{CMon})$ is a presheaf on a commutative monoid, then its Cartier dual 
\[
\bb{D}(G)=\iHom(G, \bb{A}^1_{\n{V}})
\]
is a $1$-\'etale sheaf.
\end{lemma}
\begin{proof}
The case of $\bb{A}^1_{\n{V}}$ follows  from \Cref{LemmaSubcanonicalDescent} (2). The case of invertible objects follows from \Cref{LemmaSubcanonicalDescent} (3) and the fact that a dualizable object $M$ is invertible if and only if the unit map $1\to M^{\vee}\otimes M$ is an equivalence (which can be checked locally  by descent).  

Finally, the claim about the Cartier duals is formal from the fact that $\bb{A}^1_{\n{V}}$ is a sheaf for the $1$-\'etale topology. 
\end{proof}

\begin{proposition}\label{TheoCartierDualityPrV}
Let $G\in \ob{Shv}(1\cat{Aff}^{\kappa}_{\n{V}},\Sp_{\geq 0})$ be a connective spectra object in $\ob{Shv}(1\cat{Aff}^{\kappa}_{\n{V}})$, and consider its Cartier dual 
\[
\bb{D}(G):= \iHom(G, {\bf{GL}_{1,\n{V}}})=\iHom(G, \bb{A}^1_{\n{V}}),
\]
in $\ob{Shv}(1\cat{Aff}^{\kappa}_{\n{V}},\Sp_{\geq 0})$.   Let $e\colon *\to  G$ be the unit map, and let $\Omega G= *\times_G *$ be the loops of $G$. The following holds:

\begin{enumerate}

\item  Suppose that $G$ is represented by an object in $1\cat{Aff}^{\kappa}_{\n{V}}$, and therefore it is also the case for the loops $\Omega G=*\times_G *$. Suppose that the categories $\ob{D}(G),\ob{D}(\Omega G)\in \cat{Pr}_{\n{V}}$ are dualizable, and that $*\to G$ is a $1$-\'etale cover.  Then $\bb{D}(G)$ and $\bb{D}(\Omega G)$ are also represented  in $1\cat{Aff}^{\kappa}_{\n{V}}$ and we have a fiber sequence in $\ob{Shv}(1\cat{Aff}^{\kappa}_{\n{V}},\Sp_{\geq 0})$
\[
\bb{D}(G)\to *\to \bb{D}(\Omega G).
\]

\item Keep the hypothesis of (1).  Suppose that there exists an object $H\in \ob{Shv}(1\cat{Aff}^{\kappa}_{\n{V}},\Sp_{\geq 0})$  represented in $1\cat{Aff}^{\kappa}_{\n{V}}$  such that $\ob{D}(H)\in \cat{Pr}_{\n{V}}$ is dualizable.  Suppose that the unit map $e\colon *\to H$  is a $1$-\'etale cover and that we have a fiber sequence 
\[
\bb{D}(G)\to * \to H.
\] Then one has a natural equivalence of sheaves of connective spectra 
\[
\bb{D}(H)=\Omega G,
\]
and hence, after taking duals again, a natural equivalence $\bb{D}(\Omega G)=H$. In particular, Cartier duality sends the fiber/cofiber sequence $\Omega G\to * \to G$ to the fiber cofiber sequence $\Omega H \to * \to H$. 
\end{enumerate}

\end{proposition}
\begin{proof}
\begin{enumerate}

\item  Note that $\cat{Pr}_{\n{V},\kappa}$ contains all the dualizable objects since $\n{V}$ is $\kappa$-presentable.   Since the categories $\ob{D}(G)$ and $\ob{D}(\Omega G)\in \cat{Pr}_{\n{V}}$ are dualizable, the Cartier duals $\bb{D}(G)$ and $\bb{D}(\Omega G)$ are represented in $1\cat{Aff}_{\n{V}}^{\kappa}$ by \Cref{ThmCartierDualitySymmetricmonoidal} (2). 

Now, by assumption the map $*\to G$  is a $1$-\'etale cover. This produces a fiber/cofiber sequence of connective spectrum objects 
\[
\Omega G\to *\to G
\] 
in $\ob{Shv}(1\cat{Aff}^{\kappa}_{\n{V}},\Sp_{\geq 0})$, since $\bb{D}$ is a right adjoint, one obtains a fiber sequence  of connective spectra
\[
\bb{D}(G)\to *\to \bb{D}(\Omega G)
\]
proving part (1).

\item Suppose that there is $H\in \ob{Shv}(1\cat{Aff}^{\kappa}_{\n{V}},\Sp_{\geq 0})$ as in (2). Applying (1) to the fiber/cofiber sequence $\bb{D}(G)\to *\to H$, one has a fiber sequence of connective spectrum objects
\[
\bb{D}(H)\to *\to \bb{D}(\bb{D}(G)).
\]
Since Cartier duality is an anti-involution for objects corepresented by dualizable categories thanks to  \Cref{ThmCartierDualitySymmetricmonoidal} (2), we have that $\bb{D}(\bb{D}(G))=G$. We deduce that $\bb{D}(H)=\Omega G$, and by applying the anti-equivalence again that $\bb{D}(\Omega G)=H$. This proves (2). 
\end{enumerate}
\end{proof}

\begin{remark}\label{RemarkComparisonGestalten}
The theory of Scholze and Stefanich in Gestalten produces a perfect Cartier duality in an appropriate sense, the reader might ask about the relation with \Cref{TheoCartierDualityPrV}. We claim that the proposition is nothing but a very simple case of the full power of their formalism.  To explain this fact, we need to borrow  terminology from \cite{GestaltenScholze} that we shall use freely in this paragraph. Recall that in the formalism of Gestalten, one fixes once and for all a non-countable regular cardinal $\kappa$ and works with $\kappa$-presentable $n$-categories; we will ignore the choice of the cardinal. 

Let $S$ be a Gestalten, let $n\in \bb{N}$  and let    $\cat{Gest}_{/S}^{n-\ob{et}}$ and $\cat{Gest}_{/S}^{n-\ob{prop}}$ denote the categories of $n$-\'etale and $n$-proper Gestalten over $S$ respectively. Cartier duality for Gestalten (\cite[Theorem 10.8]{GestaltenScholze}) states that the internal Hom functor $\underline{\ob{Hom}}_{S}(-,\ob{GL}_1)$ induces an  anti-self equivalence of stable categories
\[
\ob{Sp}(\cat{Gest}_{/S}^{n-\ob{et}}) \cong \ob{Sp}(\cat{Gest}_{/S}^{n-\ob{prop}})^{\op}.
\]
In the situation of \Cref{TheoCartierDualityPrV} we are interested in the base Gestalt $S=\ob{Gest}(\n{V})$ defined by the symmetric monoidal category $\n{V}$, in that case there is a fully faithful embedding 
\[
1\cat{Aff}^{\kappa,\op}_{\n{V}} \hookrightarrow \cat{Gest}^{1-\ob{prop}}_{/S}
\]
with essential image given by $1$-affine maps $X\to S$, see \cite[Definition 6.5 and Proposition 6.13 (i)]{GestaltenScholze}.

Now, there are some properties of morphisms of Gestalten that are translated under Cartier duality. We are particularly interested in those of \cite[Proposition 12.4]{GestaltenScholze} that translate connectivity with affineness. More precisely, \textit{loc. cit.} states in particular that if $G\in \cat{Sp}_{\geq 0}(\cat{Gest}_{/S})$ is an $n$-suave connective spectra over $S$ such that $BG$ is $n+1$-affine, then the Cartier dual $G^*=((G^*)_0,(G^*)_1,(G^*)_2,\ldots)\in \ob{Sp}(\cat{Gest}_{/S})$ is $(-n)$-connective and determined by  the Gestalt of the commutative algebra in $\s{O}(S)_{n+1}$ given by the $!$-sheaves $\s{O}(G/S)^!_{n}$ endowed with the convolution structure arising from the group structure of $G$.  In other words, $G^*[n]$ is a connective spectrum object in $\cat{Gest}_{/S}$, equal to $(G^*)_n$,  which is $n$-affine and given by the relative  spectrum of $\s{O}(G/S)^!_{n}$ over $S$. 

In the framework of \Cref{TheoCartierDualityPrV}, we denote $S$ by the point $*$, and we assume that the unit $S\to G$ is a $1$-\'etale cover. This implies that $G=B\Omega G$, that is, the group $G$ is a $1$-connective spectrum in $\cat{Sp}_{\geq 0}(\cat{Gest}_{/S})$ so that $\Omega G$ is still connective. Moreover, since $S$ is $1$-\'etale over itself, the fact that $n$-\'etale  Gestalten is a topos (cf \cite[Theorem 7.1]{GestaltenScholze}) implies that $G$ and also $\Omega G$ are $1$-\'etale Gestalten over $S$. We then apply \cite[Proposition 12.4]{GestaltenScholze} to the Gestalten $G$, that is, $G$ is $1$-\'etale and being $1$-affine one has that $BG$ is $2$-affine by \cite[Proposition 8.1]{GestaltenScholze} (where we use that $1$-\'etale implies $2$-prim, and that the diagonal of $BG\to S$ is $1$-affine). Then, the Cartier dual $G^*$ is $(-1)$-connective, and the shift $G^*[1]$ is a connective spectrum object in $\cat{Gest}_{/S}$ (equal to $(G^*)_1$) which is Cartier dual to $\Omega G$, and given by the $1$-affine Gestalt over $S$ defined by the commutative algebra  $\s{O}(G/S)^!_1$ which in our framework is nothing but the dual of the underlying  category of $G$ in $\cat{Pr}_{\n{V}}^{\kappa}$ endowed with the convolution tensor  product with respect to the group action of $G$, namely, this follows from dualizability of  $\s{O}(G/S)_1$.

To finish the discussion of part (1) of \Cref{TheoCartierDualityPrV}, it is left to relate $\bb{D}(G)$ with $G^*$. We claim that $\bb{D}(G)=(G^*)_1= G^*[1]$. Indeed, the sheaf $\GL_{1,\n{V}}$ represents invertible objects in the underlying  symmetric monoidal $1$-category of objects in $\cat{Aff}_{\n{V}}^{\kappa}$. In terms of \cite[Lecture XI]{GestaltenScholze}, $\GL_{1,\n{V}}$ is nothing but the restriction of the functor $(\GL_1)_1$ to $1$-affine objects in $\cat{Gest}_{/S}$.  Therefore, for a $1$-affine commutative group object $G$ over $S$, one has that 
\[
(G^*)_1=\ob{Hom}_S(G, (\GL_1)_1)= \ob{Hom}_{\ob{Shv}(1\cat{Aff}^{\kappa}_{\n{V}})}(G, \GL_{1,\n{V}}) = \bb{D}(G). 
\]

 On the other hand, the object $\bb{D}(\Omega G)$ is in principle only the  $1$-affinization of the Gestalten $G^*[1][1]=G^*[2]$. Nonetheless, the fiber/cofiber sequence $\Omega G\to *\to G$ produces a fiber/cofiber sequence of Cartier duals 
\[
G^*\to * \to G^*[1]
\]
which after shifting by $1$, and taking $1$-affinizations get the (a priori only)  fiber sequence 
\begin{equation}\label{eqFiberCofiberCartierDualityGComparisonSS}
\bb{D}(G)\to * \to \bb{D}(\Omega G).
\end{equation}

Finally, the hypothesis of part (2)  of \Cref{TheoCartierDualityPrV} guarantee that the cofiber  of the map  $G^*[1]\to *$ is represented in $1$-affine Gestalten over $S$. Hence, we have that $G^*[2]=\bb{D}(\Omega G)$, and the fiber  sequence  \eqref{eqFiberCofiberCartierDualityGComparisonSS} of connective spectra in  $\cat{Gest}_{/S}$ is also a cofiber sequence. 
\end{remark}

Let us specialize \Cref{TheoCartierDualityPrV} to the case of quasi-affinoid commutative monoids in analytic stacks.  Consider the small subcategory $\n{C}\subset \cat{AnStk}$ of analytic stacks as in \Cref{ss:AnStk}, and  let $S=\AnSpec A\in \n{C}$ be an affinoid analytic stack. In the following we identify the objects $X\in \cat{Pr}_{\ob{D},S}$ with $\ob{D}(X)$ via the equivalence $\cat{Pr}_{\ob{D},S}=\cat{Pr}_{\ob{D}(S)}$ of \Cref{PropKunnethKernel}. In particular, if $X\in \n{C}^E_{/S}$,  the commutative algebra $[X]^*$  of \Cref{ConstructionCartierDualityKernel} is identified with the symmetric monoidal category $\ob{D}(X)$. Following the notation of \Cref{ss:CartierDualitySixFunctors},  we let $\Spec^S  [X]^*\in \cat{Aff}_{\ob{D}(S)}$ be the object corepresented by the algebra $[X]^*$. Note that $\Spec^S [S]^*$ is nothing but the final object in $\cat{Aff}_{\ob{D}(S)}$ that we are  denoting  by $*$. 

\begin{theorem}\label{CoroCartierDualityAnStkQAff}

Keep the previous notation. The following holds:

\begin{enumerate}

\item  Let $f\colon G\to \AnSpec A$ be a $!$-able quasi-affinoid grouplike commutative monoid admitting a  factorization $G\to \overline{G} \to \AnSpec(A)$ where: 
\begin{itemize}

\item[i.] $\overline{G}$ is an affinoid analytic stack with induced analytic ring structure from $A$,

\item[ii.] the stack $\overline{G}$ carries a natural right action of $G$, and 

\item[iii.] the map $G\to \overline{G}$ is a $G$-equivariant open immersion. 

\end{itemize} 

Suppose that $e\colon S\to BG$ satisfies universal $\ob{D}^!$-descent, so that $BG$ is $!$-able over $S$.   Then the induced map $e\colon * \to  \Spec^S [BG]^*$ in $\cat{Aff}_{\ob{D}(S)}$ is a $1$-\'etale $!$-cover with \v{C}ech nerve given by $\{ \Spec^S [G^{\bullet}]^*\}_{\Delta^{\op}}$ where $G^{\bullet}$ is the \v{C}ech nerve of $e\colon S\to BG$ in $\n{C}$.  

\item    Suppose that $f\colon G\to S$ is a suave quasi-affinoid grouplike commutative monoid satisfying the hypothesis of (1). Then the Cartier dual $\bb{D}(\Spec^S ([BG]^*))=\Spec^S([BG]_!)$ is represented by an object in $0\cat{Aff}_{\ob{D}(S)}^{\kappa}$, i.e. it is $0$-affine and $\kappa$-presentable. The underlying commutative algebra of $[BG]_!$ is given by the endormorphisms of the unit $e_! 1\in \ob{D}(BG)$, that is by $\Hom_{BG}(e_! 1, e_! 1)\in \ob{D}(A)$. By suaveness of $e$ and proper base change, the underlying module of this algebra is equivalent to $f_!f^! 1\in \ob{D}(A)$.

\item Keep the assumptions of (2).  Let $H$ be an affinoid  grouplike commutative monoid over $S$ with   induced analytic ring structure. Let $\psi\colon H\times BG\to {\bf{GL}_{1,\n{C}}}$ be a pairing of commutative monoids, where ${\bf{GL}_{1,\n{C}}}$ is the sheaf on $\n{C}$ representing invertible elements in the underlying category of quasi-coherent sheaves. Then the underlying map 
\[
\Spec^S ([H]^*)\to \Spec^S( [BG]_!)
\]
induced by the Fourier-Mukai transform \eqref{eqFourierMukai}  arises from a unique map of commutative algebras in $\n{V}$
\begin{equation}\label{eqpkfopefmoaerka344}
f_!f^! 1\to \s{O}(H).
\end{equation}
Furthermore, the map \eqref{eqpkfopefmoaerka344}  as $A$-modules corresponds to the element in
\begin{equation}\label{eqpapsfnoaspamwoirq}
\iHom_A(f_!f^! 1, \s{O}(H)) =\Gamma(H\times G)
\end{equation}
given by the preimage of the universal invertible element in $\Omega {\bf{GL}_{1,\n{C}}}$ via the pairing  $H\times G\to \Omega {\bf{GL}_{1,\n{C}}}$.

\item Keep the assumptions in (2),  let $H$ be an affinoid grouplike commutative monoid over $S$ with induced analytic ring structure. Suppose we are given with a pairing $\psi\colon H\otimes BG\to {\bf{GL}_{1,\n{C}}}$ such that the Cartier duality map 
\[
\Spec^S([H]^*)\xrightarrow{\sim} \Spec^S([BG]_!)
\] 
is an equivalence. Suppose that $h\colon S\to BH$ is descendable. Then the pairing $BH\otimes G=H\otimes BG\to \bf{GL}_{1,\n{C}}$ induces an isomorphism 
\[
\Spec^S([G]^*)\xrightarrow{\sim} \Spec^S([BH]_!).
\]
In other words, $\psi$ induces a $1$-categorical Cartier duality as in \Cref{DefinitionCartierDuals} between $G$ and $BH$, and $BG$ and $H$ respectively.
\end{enumerate}

\end{theorem}

\begin{remark}\label{RemarkIdentificationsTheorems}

\begin{enumerate}

\item In \Cref{CoroCartierDualityAnStkQAff}  the identification  \eqref{eqpapsfnoaspamwoirq} holds by proper base change and projection formula along the cartesian square
\[
\begin{tikzcd}
H\times G \ar[r,"\pr_G"] \ar[d,"\pr_H"']&  G \ar[d,"f"] \\ 
H \ar[r,"g"] & S
\end{tikzcd}
\]
together with the fact that $f\colon G\to S$ is suave and $g\colon H\to *$ is prim, which produces an equivalence of functors $f_!f^!=f_{\sharp}f^*$, with $f_{\sharp}$ the left adjoint of $f^*$. Indeed, we have that 
\[
\begin{gathered}
\Hom_{S}(f_!f^! 1_G, g_* 1_H)  \cong \Hom_{S}(f_{\sharp} 1_G, g_* 1_H) =  \\ \Hom_{G}(1_G, f^* g_* 1_H)= f_* f^* g_* 1_H=f_*\pr_{G,*}\pr_H^* 1_H =  \Gamma(H\times G). 
\end{gathered}
\]

\item  Let $\bf{GL}_{1,\n{C}}$ be as in \Cref{CoroCartierDualityAnStkQAff} (3). Then $\Omega \bf{GL}_{1,\n{C}}$ is the sheaf on  $\n{C}$ sending an analytic ring $A$ to the commutative monoid of units  in $A^{\triangleright}(*)$, equivalently, to the commutative monoid of automorphisms of $1\in \ob{D}(A)$.

\end{enumerate}

\end{remark}

\begin{proof}[Proof of \Cref{CoroCartierDualityAnStkQAff}]
\begin{enumerate}

\item By \Cref{PropositionKunnethQuasiAffine}  (applied to $S=BG$, $I=\{0\}$ with $S_0=\AnSpec(A)$ and $\overline{S}_0=\overline{G}/G$),  the map $*:=\AnSpec(A)\to BG$ is K\"unneth. In particular, the \v{C}ech nerve of $\Spec^{S} [S]^*\to \Spec^S [BG]^*$ is given by $\Spec^S [G^{\bullet}]^*$, equivalently, we have that $\ob{D}(*)\otimes_{\ob{D}(BG)}\ob{D}(*)=\ob{D}(G)$.  Furthermore, the diagonal map $BG\to BG\times BG$ can be also written as $G\backslash G/ G\to B(G\times G)$, and so it is represented in quasi-affinoid analytic stacks. It follows by an inductive argument that the $n$-th diagonal of $BG$ is given by $BG\to G\backslash G^{S^{n-1}}/G$ (and $S^n$ is the $n$-th sphere). In particular, $G^{S^{n-1}}\to G\backslash G^{S^{n-1}}/G$ is a $!$-cover representable in quasi-affinoid analytic stacks. We claim that $BG\to G\backslash G^{S^{n-1}}/G$ is K\"unneth, by \Cref{LemmaCategoricalKunnethAfterOpenIm} it suffices to show that $BG\to G\backslash \overline{G}^{S^{n-1}} / G$ is K\"unneth, in which case it follows from   the conclusion of  \Cref{PropositionKunnethQuasiAffine} (applied to $S=G\backslash \overline{G}^{S^{n-1}} / G$, $I=\{0\} $ with $S_0= \overline{G}^{S^{n-1}}$ and $\overline{S}_0= G\backslash \overline{G}\times \overline{G}^{S^{n-1}}\times \overline{G}/G$).  We deduce   by \Cref{LemmaProperties1etaleMaps} (2) that  the map $\Spec^S [S]^*\to \Spec^S [G]^*$ is a $1$-\'etale map. Since $S\to BG$ is a $!$-cover, we have a natural equivalence 
\[
\ob{D}(BG)={\varinjlim_{\Delta^{\op}}}_! \ob{D}(G^{\bullet})
\]
where the transition maps  are given by lower $!$-maps. It follows from \Cref{LemmaEquivalentConditionsCover} (3) that $\Spec^S [S]^*\to \Spec^S [G]^*$ is a $1$-\'etale cover in $1\cat{Aff}_{\ob{D}(S)}$.

\item Suppose that $f\colon G\to S$ is suave and let $e\colon S\to BG$. The Cartier dual $\bb{D}([BG]^*)=[BG]_!$ has unit $e_{!}\colon \ob{D}(A)\to \ob{D}(BG)$. Thus, to see that $[BG]_!$ is $0$-affine it suffices to show that $\ob{D}(BG)$ is isomorphic to the module category of endormophisms of $e_! 1$ which are given by $\Hom_{BG}(e_!1,e_!1)=e^!e_!1\in \ob{D}(A)$.  Since $e\colon S\to BG$ is suave as $G$ is so, the functor $e^!$ is $\ob{D}(A)$-linear and by the monadicity theorem we have an equivalence of categories 
\[
\ob{D}(BG)\cong  \Mod_{e^!e_! 1}(\ob{D}(A)).
\]
Since the six functor formalism is $\kappa$-presentable, and $e^!$ is the right adjoint of $e_!$ in the kernel category, one  has that $e^!e_! 1\in \ob{D}(A)_{\kappa}$ is $\kappa$-compact.  The final statement relating $e^!e_! 1$ with $f_!f^! 1\in \ob{D}(A)$  follows from suaveness of $f$, and proper base change on the cartesian square
\[
\begin{tikzcd}
G \ar[r] \ar[d] & S \ar[d] \\
S \ar[r] & BG.
\end{tikzcd}
\]

\item Let $f\colon G\to S$ be as in (2), and let $H$ be an affinoid grouplike commutative monoid over $S$ with the induced analytic ring structure. Let $\psi\colon H\otimes BG\to {\bf{GL}}_{1,\n{C}}$ be a pairing of commutative monoids. By Cartier duality and the $0$-affineness of (2) we have an induced map 
\begin{equation}\label{eqpa39hnkawjeaer}
\Spec^S ([H]^*)\to \Spec^S ([BG]_!)
\end{equation}
in $0\cat{Aff}_{\n{V}}$. The underlying algebra of $[H]^*$ is given by $\s{O}(H)$, and the underlying module of the algebra of $[BG]_!$ is given by $e^!e_! 1= f_!f^! 1$. The map \eqref{eqpa39hnkawjeaer} gives rise to a morphism of $A$-modules
\[
f_!f^! 1\to \s{O}(H)
\] 
that we want to determine.

Since $*\to \Spec^S [BG]^*$ is a $1$-\'etale cover in $1\cat{Aff}_{\ob{D}(S)}$, we have that 
\[
\Spec^S [BG]^* = B (\Spec^S [G]^*)
\]
as connective spectral sheaves on $1\cat{Aff}^{\lambda}_{\ob{D}(S)}$.  Hence, we have that 
\[
\Spec^S [BG]_!=\bb{D}(\Spec^S [BG]^*) = \iHom(\Spec^S [BG]^*, {\bf{GL}}_{1,\ob{D}(S)})= \iHom(\Spec^S [G]^*, \Omega^1 {\bf{GL}_{1,\ob{D}(S)}}). 
\]

Now, the pairing $\psi$ gives rise to a map 
\[
\Spec^S [H]^*\to \iHom(\Spec^S [G]^*, \Omega^1 {\bf{GL}_{1,\ob{D}(S)}})
\]
which by construction is induced by the  adjoint to the loops of the map $\psi$, that is the map $H\otimes G\to \Omega {\bf{GL}_{1,\n{C}}}$.  Since $G$ is suave, part (2) implies that $\iHom(\Spec^S [G]^*, \Omega^1 {\bf{GL}_{1,\ob{D}(S)}})$ is corepresented by the algebra $f_!f^! 1\in \ob{D}(A)$. Let $\s{O}$ be the sheaf on $1\cat{Aff}_{\ob{D}(A)}$ sending $Y$ to $\s{O}(Y)$. Then we have a natural inclusion  $\Omega^1 {\bf{GL}_{1,\ob{D}(S)}}\to \s{O}$ with essential image given by the units of the sheaf $\s{O}$.  To determine the underlying map of modules $f_!f^!1\to \s{O}(H)$ it suffices to describe the composition 
\begin{equation}\label{eqmaosi8h3jasd} 
\Spec^S [H]^*\to  \iHom(\Spec^S [G]^*, \Omega^1 {\bf{GL}_{1,\ob{D}(A)}})\to \underline{\Map}(\Spec^S [G]^*, \s{O})
\end{equation}
seen as sheaves on $0$-affine objects $0\cat{Aff}_{\ob{D}(A)}$. Indeed, if $X$ is a $0$-affine object with algebra $\s{O}(X)\in \ob{D}(A)$, then 
\[
\underline{\Map}(\Spec^S [G]^*, \s{O})(X)=  \Gamma(G, f^*\s{O}(X))=\tau_{\geq 0}\Hom_{A}(f_!f^! 1, \s{O}(X))=\Map_{\cat{CAlg}(\ob{D}(A))}(\Sym_{A} f_!f^!1, \s{O}(X)),
\]
where in the second equivalence we have used the identification of \Cref{RemarkIdentificationsTheorems} (2). In other words,  $\underline{\Map}(\Spec^S [G]^*, \s{O})$ is corepresented by the algebra $\Sym_{A} f_!f^! 1$.  Thus, the composite \eqref{eqmaosi8h3jasd} corresponds to the map of $A$-modules  $f_!f^! 1\to \s{O}(H)$  associated to the global section $T\in \Gamma(H\times G)$ that comes from the universal unit of $\Omega \GL_{1,\n{C}}$ via the pairing $H\times G\to \Omega \GL_{1,\n{C}}$,  proving what we wanted.

\item Finally  we prove (4).  Since $G$ is suave, the map $ S\to BG$ is a $!$-cover of analytic stacks. Similarly, since $S\to BH$ is prim and descendable, it is a $!$-cover of analytic stacks. Part (1) implies that the associated maps in $\cat{Aff}_{\ob{D}(1)}$ are $1$-\'etale covers.  Thus, the $1$-categorical Cartier duality between $H$ and $BG$ and  \Cref{TheoCartierDualityPrV} (2) yields the $1$-categorical Cartier duality between $BH$ and $G$ as wanted. 
\end{enumerate}
\end{proof}

\section{Examples}\label{s:Examples}

In this section we produce different examples of Cartier duality in algebraic/analytic stacks\footnote{Some of the Cartier duality results of this paper are already mentioned in \cite{camargo2024analytic}. The proofs in \textit{loc. cit.} have some missing details that we correct here.}. The strategy is always the same: one starts with a concrete easy-to-prove Cartier duality between two concrete quasi-affinoid  abelian group stacks where \Cref{CoroCartierDualityAnStkQAff} can be directly applied. Then, using descent techniques in the category of kernels, one deduces a Cartier duality for more stacky objects.

We let $\n{C}\subset \cat{AnStk}$ denote the small full subcategory of analytic stacks of \Cref{ss:AnStk}; it consists of a sufficiently large full subcategory of analytic stacks stable under finite limits and countable colimits, and containing all the  analytic spectra of analytic rings that we are interested in.  We recall the following definition from  the previous section. 
 
\begin{definition}\label{DefinitionPicFunctor}
We let ${\bf{GL}_{1,\n{C}}}\colon \n{C}\to \ob{CMon}$ be the functor sending an analytic stack $X$ to its commutative monoid of invertible objects ${\bf{GL}_{1,\n{C}}}(X):= \ob{D}(X)^{\times}$. 
\end{definition}

\begin{remark}\label{RemarkPic}
Since invertible objects satisfy $*$-descent, ${\bf{GL}_{1,\n{C}}}$ is a sheaf for the Grothendieck topology of analytic stacks. Let $S\in \n{C}$ and consider the functor 
\[
F\colon= \Spec^S([-]^*)\colon \n{C}^{E}_{/S}\to \cat{Aff}^{\kappa}_{\ob{D},S}
\]
of \Cref{ss:CartierDualitySixFunctors}. By definition we have that ${\bf{GL}_{1,\n{C}}}|_{\n{C}^E_{/S}}=F^* {\bf{GL}}_{1,\ob{D},S}$.  
\end{remark}

By the assumptions on our category $\n{C}$ of analytic stacks,  we have that $\bb{G}_m\in \n{C}$, and  since $\n{C}$ is stable under finite limits and countable colimits, we have   $B\bb{G}_m:=\ob{AnSpec}(\Z^{\ob{cond}})/\bb{G}_m \in \n{C}$. In the examples of this paper, Cartier duality will arise from a pairing with values in $B\bb{G}_m$, this is justified thanks to the following lemma. 

\begin{lemma}\label{LemmaGmModuli}
The analytic stack $B\bb{G}_m\in \cat{AnStk}$ represents the functor sending an analytic ring $A$ to the full subanima of objects  ${\bf{GL}_{1,\n{C}}}(A)\subset \ob{D}(A)$ consisting on those invertible sheaves $\n{L}\in \ob{D}(A)$ for which there exists a $!$-cover of analytic rings $A\to B$,  such that the pullback $\n{L}_B:=B\otimes_{A} \n{L}$ is a  line bundle isomorphic to $B$. In particular, $B\bb{G}_m$ defines a commutative monoid in $\n{C}$  and there is a natural map of commutative monoids 
\[
B\bb{G}_m\to {\bf{GL}_{1,\n{C}}}. 
\] 
\end{lemma}
\begin{proof}
Consider the stack $\ob{Pic}$ of line bundles on analytic rings, that is, the stack given by the $!$-sheafification of the functor sending an analytic ring $A$ to the commutative monoid of trivial line bundles on $\ob{D}(A)$. The trivial line bundle gives rise to a map $e\colon \ob{AnSpec}(\Z^{\ob{cond}})\to \ob{Pic}$ which is surjective by definition. The  group of isomorphisms  of the map $e$ is precisely the group sending an analytic ring $A$ to the units in $\ob{End}_A(A^{\triangleright})=A^{\triangleright}(*)$, that is, the group $\bb{G}_m$. This gives rise to the presentation $B\bb{G}_m=\ob{Pic}$. By construction, $\ob{Pic}\subset {\bf{GL}_{1,\n{C}}}$ as commutative monoid, giving rise to the desired map $B\bb{G}_m\to {\bf{GL}_{1,\n{C}}}$ as wanted. 
\end{proof}

\begin{remark}\label{RemarkAnimatedGroupStructure}
Since $\bb{G}_m$ is naturally a $\Z$-module, $B\bb{G}_m$ also has an additional structure of $\Z$-module refining its commutative monoid structure. Therefore, if $G$ and $H$ are $\Z$-module objects in $\n{C}$,  in order to produce a bilinear pairing $G\times H\to {\bf{GL}_{1,\n{C}}}$ of commutative group objects, it suffices to produce a $\Z$-bilinear pairing $H\times G\to B\bb{G}_m$ and compose with the natural map of commutative group objects $B\bb{G}_m\to {\bf{GL}_{1,\n{C}}}$. 
\end{remark}

\subsection{Cartier duality for tori}\label{ss:CartierTori}

We start with the most basic Cartier duality between $\bb{Z}$ and $\bb{G}_m$. Let us first discuss the suave side of the duality, that is, $\bb{Z}_{\Betti}$.   For any discrete set $X$, $X_{\ob{Betti}}$ is cohomologically \'etale (\cite[Definition 4.6.1]{HeyerMannSix}) over $*=\AnSpec(\Z^{\ob{cond}})$ being a  disjoint union of points. This makes $B\bb{Z}_{\ob{Betti}}$ the quotient of $*$ by an \'etale group in $\n{C}$, and so it is a $!$-able stack that is cohomologically \'etale (\cite[Lemma 4.6.3 (ii)]{HeyerMannSix}).

\begin{lemma}\label{LemmaHopfAlgebraZBetti}
Let $g\colon \Z_{\Betti}\to \AnSpec \Z^{\cond}$ be the structural map. Then there is an equivalence $g_{!} 1=g_{\sharp} 1 = \bigoplus_{n\in \N} \Z$ of condensed abelian groups. 
\end{lemma} 
\begin{proof}
This follows by writing $\Z_{\Betti}$ as a  union of  points. 
\end{proof}  

Next, we discuss the prim side of the Cartier duality, that is $\bb{G}_m$. We need the following lemma:

\begin{lemma}\label{LemmaBGmShierk}
The map $e\colon \AnSpec(\Z^{\cond})\to B\bb{G}_m$ is prim and descendable. In particular $B\bb{G}_m$ is $!$-able  over $\Z^{\ob{cond}}$. 
\end{lemma}
\begin{proof}
The map $e$ is prim being representable in affinoid analytic stacks with the induced structure. Since $\Z[T^{\pm 1} ]$ is a flat condensed abelian group,  \Cref{LemmaRepTheoryFlat} implies that $\ob{D}(B\bb{G}_m)^{\heartsuit}$ is the abelian category of $\Z[T^{\pm 1}]$-comodules in condensed abelian groups. Then, $e_* 1$ lies in $\ob{D}(B\bb{G}_m)^{\heartsuit}$ and is isomorphic to the regular representation $\Z[T^{\pm 1}]$. The descendability follows from the fact that the trivial representation of $\bb{G}_m$ is a direct summand of the regular representation. 
\end{proof}

We let $\Psi\colon \Z_{\Betti}\times \bb{G}_m\to \bb{G}_m$ be the natural multiplication map arising form the $\Z$-module structure of $\bb{G}_m$. We have the following Cartier duality:

\begin{proposition}\label{PropCartierZGm}
The pairing  $\Psi\colon \Z_{\Betti}\times \bb{G}_m\to \bb{G}_m$ gives rise to $1$-categorical Cartier dualities 
\[
[B\bb{G}_m]_{!}\cong [\Z_{\Betti}]^* \mbox{ and } [\bb{G}_{m}]_!\cong [B\bb{Z}_{\Betti}]^*
\]
in $\ob{K}_{\ob{D}, \AnSpec \Z^{\cond}}$ as in \Cref{DefinitionCartierDuals}. 
\end{proposition}
\begin{proof}
Thanks to \Cref{CoroCartierDualityAnStkQAff} (with $G=\bb{Z}_{\Betti}$, $\overline{G}=\Spec(\s{O}(\bb{Z}_{\Betti}))$ endowed with the natural translation action by $\bb{Z}$,  and $H=\bb{G}_m$), since $f\colon\bb{Z}_{\Betti}\to \AnSpec \Z^{\cond}$ is suave and $\AnSpec \Z^{\cond}\to B\bb{G}_m$ is descendable, it suffices to show that the induced global sections of $\Psi$  gives rise to an isomorphism
\[
f_{\sharp} 1 \to \Z[T^{\pm 1}].
\]
We have that 
\[
\Gamma(\Z_{\Betti}\times \bb{G}_m)= \prod_{\Z} \Z[T^{\pm 1}]
\]
and the global section given by $\Psi$ is the tuple $(T^{n})_{n\in N}$. This produces the map 
\[
f_{\sharp} 1 = \bigoplus_{n\in \Z}  \Z\to  \Z[T^{\pm 1}]
\]
sending the $n$-th term basis the dirct sum  to $T^n$, proving that it is an isomorphism as wanted. 
\end{proof}

It is now easy to improve \Cref{PropCartierZGm} to a  stacky statement, this will imply in particular a Cartier duality for arbitrary tori.  In the following, we let $\ob{GL}_n$ be the general linear group of degree $n$ (note that $\ob{GL}_1=\bb{G}_m$ which  is different from what we are denoting $\GL_{1,\n{C}}$). 

\begin{theorem}\label{TheoCartoerDualityTori}
Let $\cat{Latt}$ be the algebraic stack of finite free $\Z_{\Betti}$-lattices, that is, the algebraic stack whose values in an analytic ring $A$ consists on $\Z_{\Betti}$-modules $M$ over $\AnSpec(A)$ that, locally in the $!$-topology, are isomorphic to $\Z^n_{\Betti}$ for some $n\in \N$. The following hold: 
\begin{enumerate}

\item Consider the map $f\colon \bigsqcup_{n\in \N} B\ob{GL}_n(\Z)_{\ob{Betti}} \to \ob{Latt}$ where the lattices in the left terms are induced by the standard represetations.  Then $f$ is an equivalence of analytic stacks. 

\item Let $V$ be the universal $\Z_{\Betti}$-lattice over $\ob{Latt}$ and let $V^{*}$ be its $\Z_{\Betti}$-linear dual. We let $\bb{T}_{V^{*}}:= \bb{G}_m\otimes_{\Z_{\ob{Betti}}} V^{\bullet}$. Then the $\Z$-linear   pairing $\Psi\colon V\times \bb{T}_{V^*}\to \bb{G}_m$ gives rise to a $1$-categorical Cartier duality in the kernel category $\ob{K}_{\ob{D}, \ob{Latt}}$ 
\begin{equation}\label{eqjoqnwfpqwkepqe}
[B\bb{T}_{V^*}]_!\cong [V]^* \mbox{ and } [\bb{T}_{V^*}]_!\cong [BV]^*. 
\end{equation}

\end{enumerate}
\end{theorem}
\begin{proof}
Part (1) follows essentially by definition of $\cat{Latt}$. Indeed, given $A$ an analytic ring, a lattice $M\to \AnSpec A$ is, locally in the $!$-topology, isomorphic as $\Z_{\Betti}$-module to $\Z^n_{\Betti}$ for some $n\in \N$. This implies that the map $\bigsqcup_{n\in \N} *\to \cat{Latt}$ where the $n$-th map corresponds to the constant lattice $\Z^n_{\Betti}$ is an epimorphism of analytic stacks. Now, the automorphisms of $*\xrightarrow{\Z^n_{\Betti}} \cat{Latt}$ is nothing but $\ob{Aut}_{\Z_{\Betti}}(\Z^n_{\Betti})=\GL_n(\Z)_{\Betti}$, proving that $f$ in (1) is an equivalence. 

For (2), to show that the pairings produce a Cartier duality, by \Cref{RemarkDescentBaseChange} it suffices to prove this locally in the $!$-topology of $\cat{Latt}$. Thus, by (1) we can assume without loss of generality that $V\cong \Z^n$ as a $\Z$-lattice over $S=\AnSpec(\Z^{\cond})$.  Then the pairing $\Psi$ is a  direct products of the component-wise pairings,  this reduces  the problem   to  case of $V=\Z$ and hence to \Cref{PropCartierZGm}. 
\end{proof}

\subsection{Cartier duality for vector bundles}\label{ss:CartierDualityVectorBundles}

In this section we prove  Cartier duality for different incarnations of vector bundles. The strategy is the same as in \Cref{TheoCartoerDualityTori} and the only difference consists in the analogue of \Cref{PropCartierZGm}. Thus, to avoid repetition in the argument, we will state a general theorem involving all the cases of interest in the paper, and only prove the key Cartier duality in the basic cases following \Cref{CoroCartierDualityAnStkQAff}.

\subsubsection{Algebraic vector bundles}\label{ExAlgebraicVectorBundles}
The first and most fundamental example for us is the case of algebraic vector bundles. This discussion can take place in the six functor formalism of quasi-coherent sheaves on algebraic stacks. However, in order to keep a similar framework for all the cases of the paper, we will state it  in the six functor formalism of analytic stacks. 

 Let us first discuss the case of the trivial vector bundle. Consider the algebraic  affine  line  $\bb{G}_{a}=\AnSpec \bb{Z}[T]^{\cond}$ as a ring stack in analytic stacks. Let  $\Z[T]^{\ob{DP}}=\bigoplus_n \Z \frac{T^n}{n!}$ be the divided power envelope of $\Z[T]$ at $T=0$, and let $\bb{G}_{a}^{\sharp}=\AnSpec(\Z[T]^{\ob{DP},\ob{cond}})$. 
 
\begin{lemma}\label{LemmaModuleGasharp}
The analytic stack $\bb{G}_{a}^{\sharp}$ has a natural structure of $\bb{G}_a$-module. 
\end{lemma} 
\begin{proof}
This follows formally from the theory of animated PD pairs of \cite{MaoCrystalline}. Let us give a more  elementary proof. Since the analytic stacks $\bb{G}_{a}^{\sharp}$ and $\bb{G}_a$ are represented by static analytic rings that are flat over $\Z^{\ob{cond}}$, it suffices to endow $\bb{G}_a^{\sharp}$ with a   $\bb{G}_a$-module when restricted to static analytic rings. Indeed, if $\n{D}_0\subset \cat{AnStk}^{\ob{aff}}$ is the full subcategory of condensed rings which are static and flat over $\Z^{\cond}$, the left Kan extension 
\[
F_!\colon \ob{PShv}(\n{D}_0)\to \cat{AnStk}
\]
of the inclusion $F\colon \n{D}_0\to \cat{AnStk}$ preserves colimits and finite products, and in particular it is symmetric monoidal for the Cartesian symmetric monoidal structure. Thus, after passing to spectral objects, the functor $F_!$ is symmetric monoidal and it preserves algebras and modules over algebras.

Now, as both $\bb{G}_a$ and $\bb{G}_a^{\sharp}$ are $0$-truncated in $\n{D}_0$, to endow $\bb{G}_a^{\sharp}$ with a $\bb{G}_a$-module  structure, it suffices to do it at abelian level. That is, we need to construct the addition map $\ob{Ad}\colon \bb{G}_a^{\sharp}\times \bb{G}_a^{\sharp}\to \bb{G}_a^{\sharp}$ and the multiplication map $m\colon \bb{G}_a\times \bb{G}_a^{\sharp}\to \bb{G}_a^{\sharp}$, and prove the standard axioms for a module structure in abelian groups. The maps $\ob{Ad}$ and $m$ correspond to the following maps of algebras 
\[
\ob{Ad}\colon  \bb{Z}[T]^{\ob{DP}}\to \bb{Z}[Y,X]^{\ob{DP}}, \;\;\; T\mapsto Y+X
\]
and 
\[
m\colon \bb{Z}[T]^{\ob{DP}} \to \Z[Y]\otimes_{\Z} \Z[X]^{\ob{DP}}= \Z[X]^{\ob{DP}}[Y], \;\;\; T\mapsto YX.
\]
The verification that they endow $\bb{G}_a^{\sharp}$ with a $\bb{G}_a$-module structure are routine verifications. That is, in terms of polynomial laws, the coassociativity of the coalgebra structure boils down to $f(X+(Y+Z))=f((Y+X)+Z)$, the  cocommutativity to $f(X+Y)=f(Y+X)$, the counit to  $f(X+0)=f(X)$ and the antipode to $f(X-X)=f(0)$. Similarly, the $\bb{G}_a$-module  structure  boils down to the identities $f(X(Y+Z))=f(XY+XZ)$ and $f(X(YZ))=f((XY)Z)$.  
\end{proof}
 
Next, we construct the pairing that  witnesses the Cartier duality for algebraic vector bundles. We let $\widehat{\bb{G}}_a=\varinjlim_{n} \AnSpec \Z[T]^{\cond}/T^n$ be the formal completion at $0$ of $\bb{G}_a$, it is the $\bb{G}_a$-ideal corepresenting the nilpotent elements of the underlying discrete ring  $A^{\triangleright}(*)$ of an analytic ring $A$. 

\begin{construction}\label{ConstructionExponential}
We define the exponential map 
\[
\ob{exp}(xy)\colon \widehat{\bb{G}}_a\times \bb{G}_a^{\sharp}\to \bb{G}_m 
\]
to be the map of analytic stacks induced by the power series map
\[
\bb{Z}[T^{\pm 1}]\to \Z[Y]^{\ob{DP}}[[X]], \;\;\; T\mapsto \exp(YX)=\sum_{n\in N} \frac{Y^nX^n}{n!}
\]
where the right-hand-side ring is endowed with the $X$-adic topology, and the left term with the trivial topology.  
\end{construction}

\begin{lemma}\label{LemmaBilinearExponential}
The exponential map \Cref{ConstructionExponential}  is a $\Z$-bilinear pairing of analytic stacks. 
\end{lemma}
\begin{proof}
By applying the same argument as in \Cref{LemmaModuleGasharp}, it suffices to show that the exponential is $\Z$-bilinear as a pairing when restricted to static analytic rings. In this case, the lemma reduces to proving that the finitely many diagrams witnessing the $\Z$-bilinearity are commutative, which  then reduces to the standard identity of the exponential power series, that is,  $\exp((X_1+X_2)Y)= \exp(X_1Y) \exp(X_2Y)$ in the ring $\Z[Y]^{\ob{DP}}[[X_1,X_2]]$  and $\exp(X(Y_1+Y_2))=\exp(XY_1)\exp(XY_2)$ in $\Z[Y_1,Y_2]^{\ob{DP}}[[X]]$.   
\end{proof}

\begin{remark}\label{RemarkStackLinear}
Notice that the exponential map $\ob{exp}(XY)\colon  \widehat{\bb{G}}_a\times \bb{G}_a^{\sharp}\to \bb{G}_m$ is compatible with the action of $\bb{G}_a$ on both terms. Then it actually factors as the composite 
\[
\widehat{\bb{G}}_a\times \bb{G}_a^{\sharp}\to \widehat{\bb{G}}_a\otimes_{\bb{G}_a} \bb{G}_a^{\sharp} \to \bb{G}_m
\]
where the second map is a $\Z$-linear map that we can safely denote $\exp$. 
\end{remark}

The analytic stack $\widehat{\bb{G}}_a$ is suave over $\AnSpec \Z^{\cond}$ by \cite[Lemma 6.4.15]{SolidNotes}, it is also quasi-affinoid being an open immersion (in the sense of analytic stacks) of the algebraic affine line $\bb{G}_a$. We need to establish the $!$-ability of the stack $B\bb{G}_a^{\sharp}$.

\begin{lemma}\label{LemmaShierkBGasharp}
The map $s\colon \AnSpec \Z^{\Cond}\to B\bb{G}_a^{\sharp}$ is descendable, in particular $B\bb{G}_{a}^{\sharp}$ is $!$-able over $\AnSpec \Z^{\cond}$. 
\end{lemma}
\begin{proof}
The algebra $\s{O}(\bb{G}_a^{\sharp})=\Z[T]^{\ob{DP}}$ is flat over $\Z$, thus by  \Cref{LemmaRepTheoryFlat}  we know that $\ob{D}(B\bb{G}_a^{\sharp})$ has a natural $t$-structure whose heart is the category of $\Z[T]^{\ob{DP}}$-comodules. The comodule $s_*1\in \ob{D}( B\bb{G}_a^{\sharp})$ is the regular comodule $\Z[T]^{\ob{DP}}$ induced by comultiplication. By the Poincar\'e lemma we have a short exact sequence 
\[
0\to \Z\to \Z[T]^{\ob{DP}}\xrightarrow{\partial_{T}} \Z[T]^{\ob{DP}}\to 0
\]
where $\partial_T$ is the derivation with respect to the variable $T$.  It is a classical computation that this is an exact  sequence of $\Z[T]^{\ob{DP}}$-comodules, namely, this is equivalent to the identity $(\partial_T f)(X+Y)=\partial_{Y}(f(X+Y))$. This proves that the unit of $\ob{D}(B\bb{G}_a^{\sharp})$ is in the thick tensor ideal of $s_* 1$ and therefore that $s$ is descendable. 
\end{proof}

We deduce the Cartier duality for the divided-power affine line:

\begin{proposition}\label{PropBasicAlgVectCD}
The pairing  $\Psi\colon \widehat{\bb{G}}_{a}\times \bb{G}_a^{\sharp}\to \bb{G}_m$  of \Cref{ConstructionExponential} gives rise to $1$-categorical Cartier dualities in $\ob{K}_{\ob{D}, \AnSpec \Z^{\cond}}$
\[
[B\bb{G}_a^{\sharp}]_!\cong [\widehat{\bb{G}}_a]^* \mbox{ and } [\bb{G}_a^{\sharp}]_!\cong [B\widehat{\bb{G}}_a]^*
\]
as in \Cref{DefinitionCartierDuals}. 
\end{proposition}
\begin{proof}
Let us denote $S=\AnSpec \Z^{\cond}$. We apply \Cref{CoroCartierDualityAnStkQAff}  (with $G=\widehat{\bb{G}}_a$, $\overline{G}=\bb{G}_a$ and $H=\bb{G}_a^{\sharp}$) knowing that $\widehat{\bb{G}}_a\to S$ is suave and  $S\to B\bb{G}_a^{\sharp}$ is a $!$-cover. Let $g\colon \widehat{\bb{G}}_a\to S$ be the structural map, and let $g_!g^!1=g_{\natural} 1\in \ob{D}(\Z^{\cond})$ be its homology.  A classical computation shows that $g_{\natural} 1$ is the \textit{continuous dual} of $g_* 1=\Z[[X]]$ for the $X$-adic topology, that is $g_{\sharp} 1 \cong \bigoplus_{n\in \N} \Z X^{n,\vee}$. It suffices to see that the morphism
\begin{equation}\label{eqomaspfonaonpawrf}
\bigoplus_{n\in \N} \Z X^{n,\vee}\to \Z[T]^{\ob{DP}}
\end{equation}
arising from the pairing $\Psi$  is an isomorphism. The pairing $\Psi$ is given by the exponential power series
\[
\exp(XT)\in \Z[T]^{\ob{DP}}[[X]]=\Hom_{\Z}(\bigoplus_{n\in \N} \Z X^{n,\vee}, \Z[T]^{\ob{DP}}),
\]
and it is a standard computation that this map sends $X^{n,\vee}$ to $\frac{T^n}{n!}$, proving that \eqref{eqomaspfonaonpawrf} is an isomorphism as wanted.
\end{proof}

Note that $\bb{G}_a^{\sharp}\times_{\AnSpec \Z^{\cond}} \AnSpec \Q^{\cond}= \bb{G}_{a,\Q}$ is the affine line. One has the following specialization of \Cref{PropBasicAlgVectCD} by taking base change to characteristic zero. 
 
 \begin{corollary}\label{CorBasicAlgChar0}
Let $\bb{G}_{a,\Q}$ be the base change of $\bb{G}_{a}$ to $\Q$. Then the pairing  $\exp(xy)\colon \bb{G}_{a,\Q}\times \widehat{\bb{G}}_{a,\Q}\to \bb{G}_{m,\Q}$ of \Cref{ConstructionExponential} factors as the composite of the multiplication map $\bb{G}_{a,\Q}\times \widehat{\bb{G}}_{a,\Q}\to \widehat{\bb{G}}_a$ and the exponential map $\exp\colon \widehat{\bb{G}}_{a,\Q}\to \bb{G}_{m,\Q}$. Furthermore, the Cartier duality of \Cref{PropBasicAlgVectCD} induces a $1$-categorical Cartier duality of analytic stacks over $\Q^{\cond}$ 
\[
[B\bb{G}_{a,\Q}]_!\cong [\widehat{\bb{G}}_{a,\Q}]^* \mbox{ and } [\bb{G}_{a,\Q}]_!\cong [B\widehat{\bb{G}}_{a,\Q}]^*.
\]
 \end{corollary}

\subsubsection{Solid vector bundles}
\label{ExSolidVectorBundle}

Next, we prove a Cartier duality for a solid incarnation of vector bundles. This incarnation of Cartier duality is related to work in progress with Aoki and Zavyalov \cite{AokiRodZavSoliddR} where we develop the theory of (ultra)solid de Rham stacks.   We shall work over the base $S=\AnSpec \Z_{\sol}$ of solid integers. Let $\bb{G}_{a,\sol}=\AnSpec \Z[T]_{\sol}$ be the solid affine line. As it is shown in \Cref{ExAlgebraicVectorBundles}, the Cartier dual of $\bb{G}_{a,\sol}$ should be some incarnation of the divided-power envelope of $\bb{G}_a$ at zero. Let $\Z[[T]]^{\ob{DP}}:=\prod_{n\in \N} \Z \frac{T^n}{n!}$ be the  completion of $\Z[T]^{\ob{DP}}$ with respect to its divided-power filtration, seen as a solid ring endowed with the product topology. We let $\bb{G}_{a,\sol}^{\sharp}:=\AnSpec ((\Z[[T]]^{\ob{DP}},\Z_{\sol}))$ be the analytic spectrum of $\Z[[T]]^{\ob{DP}}$ endowed with the induced solid analytic ring structure.

We first construct the pairing.

\begin{lemma}\label{LemmaModuleGasharpSolid}
The analytic stack $\bb{G}_{a,\sol}^{\sharp}$ has a natural structure of $\bb{G}_{a,\sol}$-module. 
\end{lemma} 
\begin{proof}
This follows formally from the theory of complete solid PD pairs of \cite{AokiRodZavSoliddR}. As the reference is not yet available we give an elementary proof. Since finite products of objects $\bb{G}_{a,\sol}$ and $\widehat{\bb{G}}_{a,\sol}^{\sharp}$ are static, by the same left-Kan-extension argument of \Cref{LemmaModuleGasharp} (applied to the full subcategory of analytic stacks generated by finite products of $\bb{G}_{a,\sol}$ and $\widehat{\bb{G}}_{a,\sol}^{\sharp}$) it suffices to produce the structure of a $\bb{G}_{a,\sol}$-module on $\widehat{\bb{G}}^{\sharp}_{a,\sol}$ as presheaves on sets. For that, it suffices to construct an addition map $\ob{Ad}\colon\widehat{\bb{G}}_{a,\sol}^{\sharp}\times \widehat{\bb{G}}_{a,\sol}^{\sharp}\to \widehat{\bb{G}}_{a,\sol}^{\sharp}$ and a multiplication map $m\colon \bb{G}_{a,\sol}\times \widehat{\bb{G}}_{a,\sol}^{\sharp}\to \widehat{\bb{G}}_{a,\sol}^{\sharp}$ satisfying the obvious module axioms. Let us directly write the maps, the verification of the module axioms are straightforward computations and proven by checking the same identities as in the proof of \Cref{LemmaModuleGasharp}. The addition map $\ob{Ad}$ corresponds to the morphism of solid algebras 
\[
\bb{Z}[[Y]]^{\ob{DP}} \to \bb{Z}[[Y_1]]^{\ob{DP}}\otimes_{\Z_{\sol}} \bb{Z}[[Y_2]]^{\ob{DP}}=\bb{Z}[[Y_1,Y_2]]^{\ob{DP}},\;\;\;  Y\mapsto Y_1+Y_2,
\]
note that this map is well defined since the solid tensor product satisfies $\prod_{I} \Z \otimes_{\Z_{\sol}} \prod_{J} \Z = \prod_{I\times J }\Z$ for countable sets $I$ and $J$. The multiplication map  corresponds to the morphism of algebras
\[
\bb{Z}[[Y]]^{\ob{DP}}\to \bb{Z}[X]_{\sol} \otimes_{\Z_{\sol}} \bb{Z}[[Y]]^{\ob{DP}} = \Z[X][[Y]]^{\ob{DP}},\;\;\; Y\mapsto XY,
\]
this map is well defined since we have the solid tensor product $\Z[X]_{\sol}\otimes_{\Z_{\sol}} \prod_{I }\Z_{\sol} = \prod_{I} \Z[X]$ for a countable set $I$. 
\end{proof}

Having constructed the $\bb{G}_{a,\sol}$-module structure on $\widehat{\bb{G}}_{a,\sol}^{\sharp}$, we can construct the exponential map that gives rise the Cartier duality.

\begin{construction}\label{ConstructionPairingGaSol}
Consider the exponential map $\exp\colon \widehat{\bb{G}}_{a,\sol}^{\sharp}\to \bb{G}_{m}$ induced by the morphism of $\bb{Z}_{\sol}$-algebras  
\[
\bb{Z}[T^{\pm 1}]\to \bb{Z}[[Y]]^{\ob{DP}},\;\;\; T\mapsto \exp(Y)=\sum_{n} \frac{Y^n}{n!}
\]
where $\bb{Z}[T^{\pm 1}]$ is considered with the induced analytic ring structure. By the usual additive  law of the exponential as power series $\exp(Y_1+Y_2)=\exp(Y_1)\exp(Y_2)$ in the ring $\bb{Z}[[Y_1,Y_2]]^{\ob{DP}}$, the map $\exp$ above is a morphism of $\Z$-modules in analytic stacks. We define the $\Z$-linear paring 
\[
\Psi\colon \bb{G}_{a,\sol} \times \widehat{\bb{G}}_{a,\sol}^{\sharp}\to \widehat{\bb{G}}^{\sharp}_{a,\sol} \to \bb{G}_m
\] 
to be the map $(X,Y)\mapsto \exp(YX)$.  
\end{construction}

The analytic stack $\bb{G}_{a,\sol}$ is quasi-affinoid being an open localization of $\bb{G}_{a}$, it is also suave over $S$ by \cite[Proposition 7.1.11]{SolidNotes}. Next we verify the $!$-ability of $\widehat{\bb{G}}^{\sharp}_{a,\sol}$. 

\begin{lemma}\label{LemmaDescendableGasolidSharp}
The map $e\colon \AnSpec \Z_{\sol}\to B\widehat{\bb{G}}_{a,\sol}^{\sharp}$ is descendable. In particular, $ B\widehat{\bb{G}}_{a,\sol}^{\sharp}$ is $!$-able.
\end{lemma}
\begin{proof}
This follows by the  argument of \Cref{LemmaShierkBGasharp}, noticing  that $\prod_{\N} \Z$ is flat in light solid abelian groups by \cite[Corollary 3.4.3]{SolidNotes},  and that the Poincar\'e lemma still holds with $\Z[[T]]^{\ob{DP}}$. 
\end{proof}

\begin{proposition}\label{PropBasicSolidVectCD}
The pairing $\Psi\colon \bb{G}_{a,\sol}\times \widehat{\bb{G}}^{\sharp}_{a,\sol} $ of \Cref{ConstructionPairingGaSol} gives rise to $1$-categorical Cartier dualities in $\ob{K}_{\ob{D},\AnSpec \Z_{\sol}}$
\[
[B\widehat{\bb{G}}_{a,\sol}^{\sharp}]_!\cong [\bb{G}_{a,\sol}]^* \mbox{ and } [\widehat{\bb{G}}_{a,\sol}^{\sharp}]_!\cong [B\bb{G}_{a,\sol}]^*.
\]
\end{proposition}
\begin{proof}
Let $g\colon \bb{G}_{a,\sol}\to \AnSpec \Z_{\sol}$ be the structural map. The proposition follows from  \Cref{CoroCartierDualityAnStkQAff} (with $G=\bb{G}_{a,\sol}$, $\overline{G}=\bb{G}_a$ and $H=\widehat{\bb{G}}_{a,\sol}^{\sharp}$) after we have shown that the map $g_{\sharp}1\to \s{O}(\widehat{\bb{G}}_{a,\sol}^{\sharp})= \Z[[T]]^{\ob{DP}}$ induced by $\Psi$ is an isomorphism. But $g_{\sharp} 1$ is the solid dual of $\Z[X]$, namely, $\prod_{\N} \Z X^{n,\vee}$, and the map $\prod_{\N} \Z X^{n,\vee}\to \Z[[T]]^{\ob{DP}}$ is induced by the exponential map $\exp(TX)\in \Z[X][[T]]^{\ob{DP}}=\Gamma(\bb{G}_{a,\sol}\times \widehat{\bb{G}}_{a,\sol}^{\sharp})$ and therefore an isomorphism. 
\end{proof}

\begin{remark}\label{RemarkCartierDualityOtherVectorBundles}
The reader might wonder why  there is  an asymmetry in the statements of \Cref{PropBasicAlgVectCD,PropBasicSolidVectCD}.  Indeed, in principle one might wonder what  the Cartier duals of the stacks $\bb{G}_{a,\bb{Z}}=\Spec \bb{Z}[T]$ and $\widehat{\bb{G}}_{a,\sol}=\AnSpec (\bb{Z}[[T]],\bb{Z})_{\sol}$ are.  A technical problem that prevent us to use  \Cref{CoroCartierDualityAnStkQAff} is that the maps from the point to the classifying stacks of $\bb{G}_a$ is not descendable. Indeed, this fails in characteristic $p$ as $\bb{G}_a$-group cohomology fails to have finite cohomological dimension. Following conversations with Scholze and Stefanich, the description of the Cartier dual of $\bb{G}_{a,\bb{Z}}$   in Gestalten is more subtle: one can use \cite[Proposition 12.4]{GestaltenScholze} to formally deduce that the Cartier dual $\bb{G}_a^*$ of $\bb{G}_a$ is $(-1)$-connective as a Gestalt,  \Cref{PropBasicAlgVectCD} tells us that when base changed to $\bb{Q}$ this spectrum object in Gestalten is actually $0$-connective and equal to $\widehat{\bb{G}}_{a,\bb{Q}}$. However, in characteristic $p$, the lack of $!$-ability of  $B\bb{G}_{a,\bb{F}_p}$ should be reflected in the fact that $\bb{G}_a^{*}$ \textbf{is not} $0$-connective, and therefore it does not arise from a classical object.  A similar problem should occur with $\widehat{\bb{G}}_{a,\sol}$. 
\end{remark}

\subsubsection{Disc bundles}\label{ExDiscVectorBundle}

Next, we prove Cartier duality for unit discs, this requires to work over a base ring $R$ with a pseudo-uniformizer $\pi$. We will choose as our base the ring $R=\Z((\pi))$ with the induced solid analytic ring structure.  We will see classical sheafy Tate-Huber adic spaces as analytic stacks as in  \cite{Andreychev}, namely,  \cite[Theorem 4.1]{Andreychev} and \cite[Proposition 2.3.2]{camargo2024analytic} imply that rational covers of the adic spectrum of sheafy analytic Huber rings give rise to open covers of analytic stacks, and we can glue affinoid Huber analytic adic spaces into analytic stacks via the topology of the underlying adic space.

 The ring $R$ is a Tate Banach ring, and the base change $\bb{G}_{a,\sol,R}:=\bb{G}_{a,\sol}\times_{\AnSpec (\Z_{\sol})} \AnSpec (R)$ is affinoid and corepresented by the solid $R$-algebra given by the Huber pair
\[
R\langle T \rangle_{\sol}:=(\Z[\pi,T]^{\wedge}_{\pi} [\frac{1}{\pi}], \Z[\pi,T]^{\wedge}_{\pi} )_{\sol},
\]
where for a ring $A$, an element $f\in A$, and an $A$-module $M$, we write $M^{\wedge}_f$ for the (derived) $f$-adic completion. In particular, this base change is nothing but the closed affinoid unit disc over $R$ seen as an adic space, that we also denote by $\bb{D}_{R}$. 

\begin{definition}\label{DefinitionDiscDPversions}
The analytic affine line $\bb{G}_{a,R}^{\an}$ over $R$ is the analytic stack given as the union of open discs
\[
\bb{G}_{R}^{\an}=\bigcup_{n\in \N} \AnSpec  R\langle \pi^nT \rangle_{\sol}.
\]
Equivalently, it is the analytic ring stack over $\AnSpec R$ given by 
\[
\bb{G}_{R}^{\an}= \bb{G}_{a,\sol,R}[\frac{1}{\pi}]. 
\]
\end{definition}

The analytic stack $\bb{G}_R^{\an}$ admits a  \textit{norm map} $|-|_{\pi}\colon \bb{G}_{R}^{\an}\to [0,\infty)$ defined by the collection of idempotent algebras of  \cite[Definition 2.2.7]{anschutz2025analytic} (the reference only deals with the case of $\Q_p$ but the exact same line of arguments work over the base $R$). We normalize the  norm map  such that $|\pi|=1/2$.

Let $r\in (0,\infty)$, the \textit{closed overconvergent  disc of radius $r$} denoted by $\bb{D}^{\leq r}_{R}$ is defined as the preimage of $[0,r]$. Similarly, the \textit{open  disc of radius $r$} denoted by $\bb{D}^{<r}_{R}$ is defined as the preimage of $[0,r)$.  By \cite[Lemma 2.2.11]{anschutz2025analytic}, the closed overconvergent unit disc $\bb{D}^{\leq 1}_{R}$ has a natural structure of a ring stack, and the discs $\bb{D}^{\leq r}_{R}$ and $\bb{D}^{<r}_{R}$ have natural structures of $\bb{D}^{\leq 1}_R$-modules.  In order to have a more concrete description of the algebras defining these spaces we need to introduce the following sequential spaces:

\begin{definition}\label{DefBanachSpacesSequences}
Consider the Banach norm $|-|_{\pi}$ on $R=\Z((\pi))$ so that $R^{\leq 1}=\Z[[\pi]]$, $\pi$ is multiplicative for the norm, and $|\pi|=1/2$. Let $r\in (0,\infty)$ be a positive real number, we define the $R$-Banach space of sequences of $r$-exponential decay  $\ell_{R}^{\N}(r)$ to be the subspace of $\prod_{\N} R$ of sequences $(a_{n})_{n\in N}$ such that $|a_{n}|_{\pi}r^n\to 0$ as $n\to \infty$. We endow $\ell_{R}^{\N}(r)$ with the Banach norm given by $|(a_n)|=\sup_{n\in N} (|a_n|_{\pi} r^n)$. 
\end{definition}

\begin{remark}\label{RemarkBanachSolid}
\begin{enumerate}
\item A Banach $R$-module $M$ is solid, namely, $M$ admits a $\pi$-complete $\Z[[\pi]]$-submodule $M^0$ such that $M=M^{0}[\frac{1}{\pi}]$, and $M^0=\varprojlim_k M^0/\pi^k$ with $M^0/\pi^k$ a discrete $\Z[[\pi]]$-module. The claim follows from the fact that discrete modules are solid, and that solid modules are stable under limits and colimits. 

\item By \cite[Proposition 2.12.10]{MannSix} the solid tensor product of connective $\pi$-complete $\Z[[\pi]]$-modules is $\pi$-complete. It follows that the solid tensor product over $R$ of $\ell^{\N}_R(r)$ and $\ell^{\N}_R(r')$ is the classical $\pi$-complete tensor product $\ell^{\N\times \N}_R(r,r')$ of sequences $(a_{n,m})_{n\in n,m}$ such that $|a_{n,m}|_{\pi} r^{n}r^{'m}\to 0$ as $(n,m)\to \infty$. 

\end{enumerate}

\end{remark}

\begin{lemma}\label{LemmaFlat}
The Banach space $\ell^{\N}_{R}(r)$ is a $R$-flat light solid module. 
\end{lemma}
\begin{proof}
We have a different presentation as solid $R$-module
\[
\ell^{\N}_R(r)=\varinjlim_{f\in \s{S}} \prod_{\N} \Z[[\pi]] \pi^{f(n)}
\]
where $\s{S}$ is the poset of functions $f\colon \N\to \Z$ such that $f(n)+n \log_{1/2}(r)\to \infty$. Since $\prod_{\N} \Z[[\pi]] \pi^{f(n)}\cong \prod_{\N} \Z \otimes_{Z_{\sol}} \Z[[\pi]]$, it is a flat $\Z[[\pi]]$-module by \cite[Corollary 3.4.3]{SolidNotes}, and hence so is $\ell^{\N}_R(r)$ as $R$-module. 
\end{proof}

By construction, the algebra $R\langle T\rangle_{\leq r}=\s{O}(\bb{D}^{\leq r}_{R})$ is the subalgebra of $R[[T]]=\prod_{n\in \N} RT^n$ given by the colimit of Banach spaces
\[
R\langle T\rangle_{\leq r}=\varinjlim_{r'>r} \ell_R^{\N}(r')=\{\sum_{n\in \N} a_n T^n: \;\; |a_n|r^{'n}\to 0 \mbox{  for some } r'>r \}
\] 
with ordered basis $\{T^n\}_{n\in \N}$.

 We can define divided power variants of these algebras

\begin{definition}\label{DefinitionDividedPowerVariant}
Let $r\in (0,\infty)$, we define the subspace $R\langle T \rangle_{\leq r}^{\ob{DP}}$ of $R[[T]]^{\ob{DP}}=\prod_{\N} R\frac{T^n}{n!}$  to be the colimit of Banach spaces
\[
R\langle T \rangle_{\leq r}^{\ob{DP}}= \varinjlim_{r'>r} \ell^{\N}_{R}(r') =\{\sum_{n\in \N} a_n \frac{T^n}{n!}: \;\; |a_n|r^{'n}\to 0 \mbox{  for some } r'>r \}
\]
with ordered basis $(\frac{T^n}{n!})_{n\in \N}$. Given $(r_1,\ldots, r_n)$ we denote 
\[
R\langle T_1,\ldots, T_n \rangle_{\leq (r_1,\ldots, r_n)}^{\ob{DP}}:= R\langle T_1\rangle_{\leq r_1}^{\ob{DP}}\otimes_R \cdots \otimes_R R\langle T_n\rangle_{\leq r_n}^{\ob{DP}}.
\]
If $r_1=r_2=\cdots=r_n$ we simply write $R\langle T_1,\ldots, T_n \rangle_{\leq r}^{\ob{DP}}=R\langle T_1,\ldots, T_n \rangle_{\leq (r,\ldots, r)}^{\ob{DP}}$. 
\end{definition}

We have the following key lemma:

\begin{lemma}\label{LemmaPresentabilityDisc}
The following holds:
\begin{enumerate}

\item $R\langle T \rangle_{\leq r}^{\ob{DP}}$ is a subalgebra of $R[[T]]^{\ob{DP}}$. We define the affinoid analytic stack $\bb{D}^{\sharp,\leq r}_{R}:=\AnSpec R\langle T \rangle_{\leq r}^{\ob{DP}}$.

\item The map $R[[T]]^{\ob{DP}}\to R[[X,Y]]^{\ob{DP}}$ sending $T\mapsto X+Y$ restricts to a map $R\langle T\rangle_{\leq r}^{\ob{DP}} \to R\langle X,Y \rangle_{\leq r}^{\ob{DP}}$ making $R\langle T\rangle_{\leq r}^{\ob{DP}} $ a cocommutative Hopf algebra over $R$, in particular  $\bb{D}^{\sharp,\leq r}_{R}$ has a natural  $\Z$-module structure. 

\item The map $m\colon R[[T]]^{\ob{DP}}\to R[[X]][[T]]^{\ob{DP}}$ sending $T\mapsto TX$ restricts to a map of $R$-subalgebras 
\[
m\colon R\langle T \rangle_{\leq rs}^{\ob{DP}} \to R\langle T \rangle_{\leq r}^{\ob{DP}}  \otimes_R  R\langle  X\rangle_{\leq s}.
\]
In particular, taking $s=1$ the analytic stack  $\bb{D}^{\sharp, \leq r}_{R}$ admits a multiplication map $m\colon \bb{D}_R^{\leq 1}\times \bb{D}_{R}^{\sharp,\leq r}\to \bb{D}_R^{\sharp,\leq r}$ which endows $\bb{D}_R^{\sharp,\leq r}$ with a $\bb{D}_R^{\leq 1}$-module structure.

\item The short exact sequence given by the Poincar\'e lemma
\[
0\to R\to R[[T]]^{\ob{DP}}\xrightarrow{\partial_T} R[[T]]^{\ob{DP}}\to 0
\]
restricts to a short exact sequence
\begin{equation}\label{eqlomaosmoabnaawe}
0\to R \to R\langle T \rangle_{\leq r}^{\ob{DP}}\xrightarrow{\partial_T}  R\langle T \rangle_{\leq r}^{\ob{DP}}\to 0.
\end{equation}
Furthermore, this short exact sequence is as $R\langle T \rangle_{\leq r}^{\ob{DP}}$-comodules. 

\end{enumerate}

\end{lemma}
\begin{proof}
\begin{enumerate}

\item It suffices to show that $R\langle  T \rangle^{\ob{DP}}_{\leq r}$ is stable under multiplication, this follows from a straightforward computation of power series and the ultrametric inequality: let $f(T)=\sum_{n} a_n\frac{T^n}{n!}$ and $g(T)=\sum_{n} b_n\frac{T^n}{n!}$ be elements in $R\langle T \rangle^{\ob{DP}}_{\leq r}$, then 
\[
f(T)g(T)=\sum_{n\in \N} \big( \sum_{k=0}^n  \binom{n}{k} a_{k} b_{n-k} \big) \frac{T^n}{n!}=\sum_{n\in \N} c_n \frac{T^n}{n!} 
\]
and $|c_n|_{\pi}=| \sum_{k=0}^n  \binom{n}{k} a_{k} b_{n-k} \big) |\leq  \sup_{k=0,\ldots, n}(|a_k|_{\pi} |b_{n-k}|_{\pi})$. Hence, if $r'>r$ is such that $|a_n|_{\pi} r^{'n}\to 0$ and $|b_n|_{\pi} r^{'n}\to 0$  as $n\to \infty$, the same holds for $c_n$. 

\item The factorization of $R\langle T \rangle_{\leq r}^{\ob{DP}}\to R\langle X,Y \rangle_{\leq r}^{\ob{DP}}$  under the map $T\mapsto X+Y$  follows from a straightforward computation of power series: let $f(T)=\sum_{n} a_n \frac{T^n}{n!}$, then 
\[
f(X+T)=\sum_{n} a_n\frac{(X+Y)^n}{n!}=\sum_{n} \sum_{k=0}^n a_n \frac{X^k}{k!} \frac{Y^{n-k}}{k!}=\sum_{k,l\in \N} a_{k+l} \frac{X^k}{k!}\frac{Y^l}{l!}. 
\]
Thus, if there is $r'>r$ with $|a_n|r^{',n}\to 0$ as $n\to \infty$, then $|a_{k+l}| r^{' k+l}\to 0$ as $k+l\to \infty$. 

The Hopf algebra structure on $R\langle T \rangle^{\ob{DP}}_{\leq r}$ is then immediate from standard computations, namely, in terms of power series this amounts to proving that $f(X+(Y+Z))= f((X+Y)+Z)$ for coassociativity, that $f(X+Y)=f(Y+X)$ for cocommutativity, that $f(X+0)= f(X)$ for the counit, and $f(X-X)=f(0)$ for the antipode map.

\item The factorization of $R\langle T \rangle_{\leq rs}^{\ob{DP}}\to R\langle T \rangle_{\leq r}^{\ob{DP}}\otimes_{R} R\langle  X \rangle_{\leq s}$ under the map $T\mapsto TX$ follows from the following trivial identity for a power series ring $f(T)=\sum_{n} a_n \frac{T^n}{n!}$
\[
f(XT)=\sum_{n\in \bb{N}} a_n \frac{(XT)^n}{n!}=\sum_{n} a_n \frac{T^n}{n!} X^n
\]
so that if $|a_n|(r's')^n=|a_n|{r'}^n {s'}^n\to 0$ as $n\to \infty$ for some $r'>r$ and $s'>s$, then $\sum_{n\in N} a_n \frac{(XT)^n}{m!}\in  R\langle T \rangle_{\leq r}^{\ob{DP}}  \otimes_R  R\langle  X\rangle_{\leq s}$. In particular, if $s=1$, the routine verifications involving the maps $T\mapsto X+Y$ and $T\mapsto TX$ imply that $\bb{D}_{R}^{\sharp,\leq r}$ is endowed with a $\bb{D}^{\leq 1}_R$-module structure.  Indeed, in terms of power series this amounts to proving that $f(X(Y_1+Y_2))= f(XY_1+XY_2)$ for the linearity of the coaction, and $f(X_1(X_2Y))= f((X_1X_2) Y)$ for the associativity of the action.

\item By definition, the map $\partial_T$ sends $\frac{T^n}{n!}$ to $\frac{T^{n-1}}{(n-1)!}$. Thus, at the level of sequential spaces, it acts as a shift functor killing the $0$-th entry. The exactness of \eqref{eqlomaosmoabnaawe} follows from the definition.  The fact that  \eqref{eqlomaosmoabnaawe} is a short exact sequence of $R\langle T \rangle_{\leq r}^{\ob{DP}}$-comodules is a standard power series verification, more precisely,  it is equivalent to the identity 
\[
(\partial_{T}f)(X+Y)=\partial_Y (f(X+Y)).
\]
\end{enumerate}
\end{proof}

The analytic stack $\bb{G}_{R}^{\an}$   is an open substack of the algebraic affine line $\bb{G}_{a,R}=\AnSpec( R[T]_{\Z_{\sol}/})$, hence it is quasi-affinoid. It is also suave over $R$  being a  union of affinoid open  discs $\pi^{-n}\bb{G}_{a,\sol, R}$ along open immersions. It follows that the open discs $\bb{D}^{< r}_{R}$ are quasi-affinoid and suave over $R$. It is left to establish the descendability of $B\bb{D}^{\ob{DP},\leq r}_{R}$. 

\begin{lemma}\label{LemmaModuleStructureDisc}
The map $e\colon \AnSpec R\to B(\bb{D}^{\ob{DP},\leq r}_{R})$ is descendable, in particular $B\bb{D}^{\ob{DP},\leq r}_{R}$ is $!$-able over $R$.
\end{lemma}
\begin{proof}
By \Cref{LemmaFlat} we know that $\s{O}(\bb{D}^{\ob{DP},\leq r}_{R})$ is a flat solid $R$-module. Hence, by \Cref{LemmaRepTheoryFlat} the category $B(\bb{D}^{\ob{DP},\leq r}_{R})$ has a $t$-structure whose heart are $\s{O}(R)$-comodules on solid $R$-modules.  The object $e_* 1$ corresponds to the comodule  $\s{O}(R)$ induced by comultiplication. The short exact sequence of \Cref{LemmaPresentabilityDisc} (4) implies that $e_* 1$ is descendable in $\ob{D}(B(\bb{D}^{\ob{DP},\leq r}_{R}))$ proving what we wanted. 
\end{proof}

\begin{construction}\label{ConstructionPairingDiscCase}
Let $r\in (0,\infty)$ and $r'>r$, we define the exponential map 
\begin{equation}\label{eqpmapsnoawnoawafdaf}
\exp(YX)\colon \bb{D}^{\sharp, \leq r}_R\times \bb{D}^{\leq 1/{r'}}_R\to \bb{G}_m
\end{equation}
to be the map induced by the power series 
\[
\exp(YX)=\sum_{n\in \N} \frac{Y^n}{n!} X^n
\]
where $Y$ and $X$ are the coordinates of $\bb{D}^{\sharp, \leq r}_R $ and $ \bb{D}^{\leq 1/{r'}}_R$    respectively. Note that by the condition $r'>r$ one has that $r(1/r')< r/r=1$ and so the exponential $\exp(YX)$ is a well defined function. Furthermore, the  additive property of the exponential makes the pairing \eqref{eqpmapsnoawnoawafdaf} $\Z$-bilinear.  Taking colimits on $r'$ as $r'>r$ we obtain a $\Z$-bilinear pairing 
\[
\Psi\colon \bb{D}^{\sharp, \leq r}_R\times \bb{D}^{< 1/{r}}_R\to \bb{G}_m
\]
between $\bb{D}^{\sharp, \leq r}_R$ and the open disc $ \bb{D}^{< 1/{r}}_R$.  Following \Cref{RemarkStackLinear}, the pairing $\Psi$ is compatible with the $\bb{D}^{\leq 1}_{R}$-action in both terms, and   it factors as the composite 
\[
\Psi\colon \bb{D}^{\sharp, \leq r}_R\times \bb{D}^{< 1/{r}}_R\to     \bb{D}^{\sharp, \leq r}_R\otimes_{\bb{D}^{\leq 1}_{R}} \bb{D}^{< 1/{r}}_R \to \bb{G}_m
\]
where the map $\bb{D}^{\sharp, \leq r}_R\otimes_{\bb{D}^{\leq 1}_{R}} \bb{D}^{< 1/{r}}_R\to \bb{G}_m$ is a $\Z$-linear map that we can safely denote by $\exp$.  Writing $\bb{D}^{<1/r}_{R}$ as a filtered colimit of the invertible  (for the $!$-topology) $\bb{D}_{R}^{\leq 1}$-modules $\bb{D}^{\leq r'}_{R}$ (with $r'<1/r$), one can show that $\bb{D}^{\sharp, \leq r}_R\otimes_{\bb{D}^{\leq 1}_{R}} \bb{D}^{< 1/r}_R=\bb{D}^{\sharp, \leq 1}_R\otimes_{\bb{D}^{\leq 1}_{R}} \bb{D}^{< 1}_R$ is independent of $r$, though we will not need this fact. 
\end{construction}

 \begin{proposition}\label{PropBasicDiscVectCD}
 Let $r\in (0,\infty)$. The pairing $\exp(YX)\colon\bb{D}^{\sharp, \leq r}_{R}\times  \bb{D}^{<1/r}_{R} \to \bb{G}_m$ of \Cref{ConstructionPairingDiscCase}  gives rise to $1$-categorical Cartier dualities in the category of kernels $\ob{K}_{\ob{D},\AnSpec R}$
 \[
  [B\bb{D}^{\sharp, \leq r}_{R}]_!\cong [\bb{D}^{<1/r}_R]^*   \mbox{ and }  [\bb{D}^{\sharp,\leq r}_R]_!\cong [B\bb{D}^{<1/r}_{R}]^* .
 \]
 \end{proposition}
\begin{proof}
Let $g\colon \bb{D}^{<1/r}_{R}\to \AnSpec R$. By \Cref{CoroCartierDualityAnStkQAff} (applied to $G=\bb{D}^{<1/r}_{R}$, $\overline{G}=\bb{G}_{a,R}$ the algebraic affine line over $R$ and $H=\bb{D}_R^{\sharp,\leq r}$) this reduces to proving that the pairing of \Cref{ConstructionPairingDiscCase} given by the exponential $\exp(YX)\in \Gamma(\bb{D}_R^{\sharp,\leq r}\times \bb{D}^{<1/r}_{R})$ induces an equivalence of $R$-modules
\[
f_{\natural} 1 \to R\langle Y \rangle_{\leq r}^{\ob{DP}}.
\]
The $R$-module $f_{\natural} 1$ is the \textit{naive} continuous $R$-dual of the Fr\'echet space $\s{O}(\bb{D}^{< 1/r})$, since 
\[
\s{O}(\bb{D}^{< 1/r})= \{\sum_{n\in \N} a_n X^n : \;\; |a_n|r^{',-n}\to 0 \mbox{ for all } r'>r \}.
\]
one has that $f_{\natural} 1$ is the colimit of Banach spaces given by 
\[
f_{\natural} 1=\{\sum_{n\in \N} a_n X^{n,\vee}: |a_n|r^{'n}\to 0 \;\; \mbox{ for some  } r'>r\}
\]
where $X^{n,\vee}$ is the dual of $X^n$ with respect to the basis $\{X^n\}_{n\in \N}$. Now, the exponential map $\exp(YX)$ induces the map $f_{\natural}1\to R\langle  Y \rangle^{\ob{DP}}_{\leq r}$  sending $X^{n,\vee}$ to $\frac{Y^n}{n!}$. It follows from the power series estimates that this map is an isomorphism, proving what we wanted. 
\end{proof}

%
%\begin{remark}\label{RemarkCartierDualityStackNorms}
%The Cartier duality \Cref{PropBasicDiscVectCD} also holds over the  non-archimedean locus of the gaseous stack of norms. We do not give a proof of this fact in this paper, but only a sketch of why it is true.  The open unit disc $\mathring{\bb{D}}$ is an analytic stack that descends to the gaseous stack of norms $\n{N}_{\ob{gas}}$. If one further restricts to its non-archimedean locus $\n{N}_{\ob{gas}}^{na}$, then $\mathring{\bb{D}}$ has a natural structure of $\Z$-abelian group (it even has a structure of $\bb{D}^{\leq 1}$-module). Similarly, the $\leq 1$-overconvergent divided power disc $\bb{D}^{\ob{DP},\leq 1}$  admits a descent to $\n{N}_{\ob{gas}}^{na}$, and it is also the case for the pairing \Cref{ConstructionPairingDiscCase}. The analogue computation of \Cref{PropBasicDiscVectCD} then proves the Cartier duality over $\n{N}_{\ob{gas}}^{na}$.   
%\end{remark}

\subsubsection{Analytic vector bundles}\label{ExAnalyticVectorBundle}

We continue with the Cartier duality of analytic vector bundles, we keep working over the solid ring $R=\Z((\pi))$ with induced structure from $\Z_{\sol}$. In one hand we have the analytic ring stack $\bb{G}_{a,R}^{\an}$ from \Cref{DefinitionDiscDPversions}, its Cartier dual will be the space of germs of divided power series at $0$.

\begin{definition}\label{DefinitionGadaggersharp}
We define the affinoid analytic stacks  $\bb{G}_{a,R}^{\dagger}=\varprojlim_{r\to 0} \bb{D}_R^{\leq r}$ and $\bb{G}_{a,R}^{\sharp, \dagger}=\varprojlim_{r\to 0} \bb{D}^{\sharp, \leq r}_{R}$.
\end{definition}

\begin{remark}\label{RemarkGadagger}
The analytic  stack $\bb{G}_{a,R}^{\dagger}$ is nothing but the preimage of $0$ of the norm map $|-|_{\pi}\colon \bb{G}_{a,R}^{\an}\to [0,\infty)$. In particular, $\bb{G}_{a,R}^{\dagger}$ is an ideal of $\bb{G}_{a,R}^{\an}$. 
\end{remark}

\begin{lemma}\label{LemmaBasicPropertiesGadaggerSharp}
The analytic stack $\bb{G}_{a,R}^{\sharp, \dagger}$ has a natural $\bb{G}_{a,R}^{\an}$-module structure. 
\end{lemma}
\begin{proof}
By \Cref{LemmaPresentabilityDisc} (3) we know that $\bb{D}^{\sharp, \leq r}_{R}$ has a natural $\bb{D}^{\leq 1}_{R}$-module structure. It is clear from the construction that the module structure is compatible with the divided power discs of different radius $r$, and therefore $\bb{G}_{a,R}^{\sharp, \dagger}$ also has a natural $\bb{D}^{\leq 1}_{R}$-module structure. To see that it has a $\bb{G}_{a,R}^{\an}$-module structure, it suffices to show that multiplication by $\pi$ on  $\bb{G}^{\sharp,\dagger}_{R}$ is invertible, namely, we have as ring stacks $\bb{G}_{a,R}^{\an}=\bb{D}^{\leq 1}_{R}[\frac{1}{\pi}]$. This follows from the fact that multiplication by $\pi$ induces a map 
\[
[\pi]\cdot \bb{D}^{\sharp, \leq r}_R\to  \bb{D}^{\sharp, \leq r+1/2}_R
\]
thanks to \Cref{LemmaPresentabilityDisc}.  By looking at power series this map sends 
\[
\sum_{n\in \N} a_{n} \frac{T^n}{n!} \mapsto \sum_{n\in \N} a_n \pi^n \frac{T^n}{n!} 
\]
which has an obvious inverse sending the power series $\sum_{n\in N} b_{n}\frac{T^n}{n!}$ to $\sum_{n\in N} \pi^{-n} b_{n}\frac{T^n}{n!}$ (and where the convergence conditions match since $|\pi|=1/2$). This proves that multiplication by $\pi$ is an isomorphism on $\bb{G}_{a,R}^{\sharp,\dagger}$, proving what we wanted. 
\end{proof}

\begin{lemma}\label{LemaDescetGadagger}
The map $e\colon \AnSpec R\to B\bb{G}_{a,R}^{\sharp,\dagger}$ is prim and  descendable. In particular, $B\bb{G}_{a,R}^{\sharp,\dagger}$ is a $!$-able analytic stack.
\end{lemma}
\begin{proof}
This follows from the same argument of \Cref{LemmaModuleStructureDisc}. 
\end{proof}

\begin{construction}\label{ConstructionPairingAnVB}
We define the exponential map $\exp\colon \bb{G}_{a,R}^{\sharp, \dagger}\to \bb{G}_m$ to be induced by the exponential power series $\exp(T)$. The additive property of the exponential implies that $\exp$ is a morphism of $\Z$-modules. We define the pariring 
\[
\Psi\colon \bb{G}_{a,R}^{\sharp,\dagger}\times \bb{G}_{a,R}^{\an}\to \bb{G}_m
\]
to be the composite of the multiplication map $\bb{G}_{a,R}^{\sharp,\dagger}\times \bb{G}_{a,R}^{\an}\to \bb{G}_{a,R}^{\sharp,\dagger}$ and the exponential $\exp$.
\end{construction}

 \begin{proposition}\label{PropBasicAnVectCD}
 The pairing $\Psi\colon \bb{G}_{a,R}^{\sharp,\dagger}\times \bb{G}_{a,R}^{\an}\to \bb{G}_m$ give rise to $1$-categorical Cartier dualities in the kernel category $\ob{K}_{\ob{D},\AnSpec R}$
 \[
  [B\bb{G}_{a,R}^{\sharp, \dagger}]_!\cong [\bb{G}_{a,R}^{\an}]^* \mbox{ and } [\bb{G}_{a,R}^{\sharp, \dagger}]_!\cong  [B\bb{G}_{a,R}^{\an}]^*.
 \]
 \end{proposition}
 \begin{proof}
 This follows by the same argument of \Cref{PropBasicDiscVectCD}, where we apply \Cref{CoroCartierDualityAnStkQAff} with $G=\bb{G}_{a,R}^{\an}$, $\overline{G}=\bb{G}_{a,R}$ the algebraic affine line over $R$ and $H= \bb{G}_{a,R}^{\sharp,\dagger}$. 
 \end{proof}

\begin{remark}\label{RemarkBaseChangeQp}
Let us specialize the Cartier dualities for open discs and analytic vector bundles to $\Q_p$ with norm $|p|=p^{-1}$. In this case, we have the following estimate of the $p$-adic valuation $v_p(n!)=\frac{n-s_p(n)}{p-1}$ where $s_p(n)$ is the sum of the digits of $n$ in the base-$p$ expansion (and in particular of logaritmic growth).  Hence, we have isomorphism of analytic stacks over $\Q_{p,\sol}:=(\Q_p,\Z)_{\sol}$
\[
\bb{D}^{\sharp,\leq r}_{\Q_p}= \bb{D}^{\leq rp^{-1/(p-1)}}_{\Q_p} \mbox{ and } \bb{G}_{a,\Q_p}^{\sharp,\dagger}=\bb{G}_{a,\Q_p}^{\dagger}. 
\]
In particular, \Cref{PropBasicDiscVectCD} and \Cref{PropBasicAnVectCD} specialize to Cartier dualities between $\bb{D}^{<1}_{\Q_p}$ and $\bb{D}^{\leq p^{-1/(p-1)}}_{\Q_p}$, and between $\bb{G}_{a,\Q_p}^{\an}$ and  $\bb{G}_{a,\Q_p}^{\dagger}$ as in \cite{camargo2024analytic}. 
\end{remark}

\subsubsection{Locally analytic $\bb{Z}_p$-lattices}\label{ExZpLAvectorbundle}

We finish with the Cartier duality of locally analytic $\bb{Z}_p$-lattices. For this discussion, we shall work over the analytic ring $\Q_{p,\sol}$ of $p$-adic rational numbers with the induced analytic ring structure from $\Z_{\sol}$. 

\begin{definition}\label{DefZpla}
We let $\bb{Z}_p^{la}$ be the affinoid analytic stack over $\AnSpec \Q_{p,\sol}$  given by the spectrum of the locally analytic functions $C^{la}(\Z_p,\bb{Q}_p)$ of $\Z_p$ with values in $\bb{Q}_p$.
\end{definition}

For the theory of solid locally analytic representations  we refer to \cite{RRLocallyAnalytic,RJRCSolidLocAn2}. The structure of abelian $p$-adic Lie group of $\bb{Z}_p$ endows $\bb{Z}_p^{la}$ with a $\Z$-module structure as analytic stack. Furthermore, $\Z_p^{la}$ is naturally endowed with the structure of a ring stack over $\Q_{p,\sol}$. 

\begin{lemma}\label{LemmaZpladescent}
The map $e\colon \AnSpec \Q_{p,\sol}\to B\Z_p^{la}$ is descendable, in particular $B\Z_p^{la}$ is a $!$-able analytic stack. 
\end{lemma}
\begin{proof}
This follows from the identification of $\ob{D}(B\bb{Z}_p^{la})$ with the category of solid locally analytic representations of $\Z_p^{la}$ \cite[Theorem 4.3.3]{RJRCSolidLocAn2}, and  (the dual of) Lazard's resolution \cite[Proposition 2.2.1 (3)]{RJRCSolidLocAn2}. Let us give a self contained proof.  By definition $\Z_p^{la}$ is the pullback 
\begin{equation}\label{eqlmaoniafasd}
\begin{tikzcd}
\bb{Z}_p^{la} \ar[r]\ar[d] & \bb{Z}_{p,\Betti} \ar[d] \\ 
\bb{G}_{a,\Q_p}^{\an}  \ar[r]& \bb{G}_{a,\Q_p}^{\an,\dR},
\end{tikzcd}
\end{equation}
that is, the locally analytic functions of $\Z_p$ are the germs of functions on $\bb{G}_a^{\an}$ on the closed subset $\bb{Z}_p\subset  \bb{G}_a(\Q_p)$. Here $\bb{G}_{a,\Q_p}^{\an,\dR}$ is the analytic de Rham stack as in \cite{camargo2024analytic} or \cite{anschutz2025analytic}. Since the lower horizontal map of \eqref{eqlmaoniafasd} is an epimorphism  with fiber $\bb{G}_{a,\Q_p}^{\dagger}$ (by definition of the de Rham stack), the upper  horizontal map is an epimorphism with fiber $\bb{G}_{a,\Q_p}^{\dagger}$. We deduce a pullback square
\[
\begin{tikzcd}
B\bb{G}_{a,\Q_p}^{\dagger} \ar[r] \ar[d] & \AnSpec \Q_p \ar[d] \\ 
B\bb{Z}_p^{la} \ar[r] & B\bb{Z}_{\Betti}
\end{tikzcd}
\]
The right vertical map is descendable by \cite[Proposition 5.2.5]{HeyerMannSix}, and the map $\AnSpec \Q_{p,\sol}\to B\bb{G}_a^{\dagger}$ is descendable by \Cref{LemaDescetGadagger} and \Cref{RemarkBaseChangeQp}. One deduces that $\AnSpec \Q_{p,\sol}\to B\bb{Z}_p^{la}$ is descendable as wanted. 
\end{proof}

Let $\s{O}(1+\bb{D}^{<1}_{\Q_p})$ be the space of functions on the open unit disc centered at $1$ seen as a subgroup of $\bb{G}_{m,\Q_p}^{\an}$. It is well known that the naive Cartier dual of $C^{la}(\Z_p^{la}, \Q_p)$, that is, its $\Q_p$-linear dual, is isomorphic to $\s{O}(1+\bb{D}^{<1}_{\Q_p})$ via the Amice transform  \cite{MR188199}, with multiplication and convolution being exchanged. This can be promoted to a $\bb{Z}_p^{la}$-module structure on $1+\bb{D}^{<1}_{\Q_p}$ as follows:

\begin{lemma}\label{LemmaModuleStructureOpenGRoup}
The multiplicative group $1+\bb{D}^{<1}_{\Q_p}$ has a natural structure of $\Z_p^{la}$-module with multiplication map
\[
\Z_p^{la}\times (1+\bb{D}^{<1}_{\Q_p})\to 1+\bb{D}^{<1}_{\Q_p}
\]
given by sending $(a,1+X)$ to the Amice transform $(1+X)^a$.  
\end{lemma}
\begin{proof}
This can be proven via a direct but tedious computation with power series. We give a soft argument via  Banach-Colmez spaces and the analytic de Rham stack of \cite{anschutz2025analytic}. Let $\n{BC}(\s{O})$ and $\n{BC}(\s{O}(1))$ denote the Banach-Colmez spaces of  $\s{O}$ and $\s{O}(1)$ seen as qfd arc-stacks over $\bb{F}_p$, see \cite[Definition 4.1.11]{anschutz2025analytic}. Consider their base change to $\n{M}_{\ob{arc}}(\Q_p)$, one has that 
\[
\n{BC}(\s{O})_{\n{M}_{\ob{arc}}(\Q_p)}= \underline{\Q}_p\times \n{M}_{\ob{arc}}(\Q_p) 
\] 
with $\underline{S}$ the arc-stack associated to a locally profinite set $S$, and that
\begin{equation}\label{eqBCSpacesIdentification}
\n{BC}(\s{O}(1))_{\n{M}_{\ob{arc}}(\Q_p)} =\varprojlim_{x\mapsto x^p} (1+\bb{D}^{<1,\diamond}_{\Q_p})
\end{equation}
is the perfection of the multiplicative open unit disc  around $1$. The space $\n{BC}(\s{O}(1))_{\n{M}_{\ob{arc}}(\Q_p)}$ has a natural module structure over $\underline{\Q}_{p,\n{M}_{\ob{arc}}(\Q_p)}=\n{BC}(\s{O})_{\n{M}_{\ob{arc}}(\Q_p)}$.  This endows $(1+\bb{D}^{<1,\diamond}_{\Q_p})$ with a natural $\underline{\Z}_{p,\n{M}_{\ob{arc}}(\Q_p)}$-module structure that we will show agrees with the one stated in the lemma.

 By \cite[Proposition II.22]{FarguesScholze} the identification \eqref{eqBCSpacesIdentification} sends a perfectoid ring $R$ over $\bb{F}_p$ to the group homomorphism  
\begin{equation}\label{eqAmmiceBCSpaces}
1+R^{<1}\xrightarrow{\sim}  \n{BC}(\s{O}(1))(R), \;\; 1+r \mapsto \sum_{i\in \bb{Z}} p^i[(1+r)^{p^{-i}}]
\end{equation}
where we see $ \n{BC}(\s{O}(1))(R) = \Gamma(\n{Y}_{R,(0,\infty)},\s{O})^{\varphi=p}$.   From this formula, we know that  the action of $\bb{Z}$ on $1+R^{<1}$ is by $n\cdot (1+r)\mapsto (1+r)^{n}$. By density, the map  \eqref{eqAmmiceBCSpaces} extends to a map of $\underline{\bb{Z}}_p(R)$-modules where the right term has the natural multiplication action, and the left term the action $a\cdot  (1+r)=(1+r)^a$.  Taking $1+r =((1+X)^{1/p^n})\in \varprojlim_{x\mapsto x^p} (1+\bb{D}^{<1}_{\Q_p})^{\diamond}$ as the coordinate, we deduce that the multiplication map $\underline{\bb{Z}}_p\times (1+\bb{D}_{\bb{Q}_p}^{<1})^{\diamond}\to (1+\bb{D}_{\bb{Q}_p}^{<1})^{\diamond}$ is given precisely by the Amice transform $(a,1+X)\mapsto (1+X)^{a}$.

 On the other hand, the  logarithm map $\log\colon 1+\bb{D}^{<1,\diamond}_{\Q_p}\to \bb{G}_{a,\Q_p}^{\an,\lozenge}$ is compatible with respect to the module structure of the morphism of rings $\underline{\Z}_{p,\n{M}_{\ob{arc}}(\Q_p)}\subset \bb{G}_{a,\Q_p}^{\an,\lozenge}$.  Passing to de Rham stacks, we see that the pullback (recall that the $\log$  map is \'etale)
\[
\begin{tikzcd}
1+\bb{D}^{<1}_{\Q_p} \ar[r] \ar[d]& 1+\bb{D}^{<1,\dR}_{\Q_p}   \ar[d] \\
\bb{G}_{a,\Q_p}^{\an} \ar[r] &  \bb{G}_{a,\Q_p}^{\an,\dR}
\end{tikzcd}
\]
is a module over the pullback of rings 
\[
\begin{tikzcd}
\Z_p^{la} \ar[r] \ar[d]&  \Z_{p,\Betti}   \ar[d] \\
\bb{G}_{a,\Q_p}^{\an} \ar[r] &  \bb{G}_{a,\Q_p}^{\an,\dR}
\end{tikzcd}
\]
proving what we wanted. 
\end{proof}

 \begin{proposition}\label{PropBasicZpVectCD}
 The pairing $\Psi\colon \bb{Z}_p^{la}\times (1+\bb{D}^{<1}_{\Q_p})\to \bb{G}_m$ given by the composite of the  Amice transform of \Cref{LemmaModuleStructureOpenGRoup} and the inclusion  $1+\bb{D}^{<1}_{\Q_p} \subset \bb{G}_m$ gives rise to $1$-categorical Cartier dualities in the kernel category $\ob{K}_{\ob{D},\AnSpec \Q_{p,\sol}}$
 \[
 [B\Z_p^{la}]_!\cong [1+\bb{D}_{\Q_p}^{< 1}]^* \mbox{ and } [\Z_p^{la}]_!\cong [B(1+\bb{D}_{\Q_p}^{< 1})]^* .
 \]
 \end{proposition}
 \begin{proof}
 Let $g\colon 1+\bb{D}_{\Q_p}^{< 1}\to \AnSpec \Q_{p,\sol}$. We apply \Cref{CoroCartierDualityAnStkQAff} (with $G=1+\bb{D}_{\Q_p}^{< 1}$, $\overline{G}=\bb{G}_{m,\bb{Q}_p}$ the algebraic multiplicative group over $\bb{Q}_p$,  and $H=\Z_p^{la}$), where the only thing to check is that the Amice transform  of \Cref{LemmaModuleStructureOpenGRoup} induces an equivalence  $g_{\natural} 1\xrightarrow{\sim} C^{la}(\Z_p,\Q_p)$, or dually via \cite[Theorem 3.40]{RRLocallyAnalytic}, that it induces an equivalence 
 \[
 \iHom_{\Q_p}(C^{la}(\Z_p,\Q_p), \Q_p) \cong \s{O}(1+\bb{D}^{<1}_{\Q_p})
 \]
 which is classical, see \cite[Theorem II.2.2]{MR2642404}. 
 \end{proof}

\subsubsection{Cartier duality for vector bundles over stacks} \label{ss:CartierVBStacks}

Next, we state our stacky Cartier duality theorem for all the previous incarnations of vector bundles. 

\begin{definition}\label{DefVect}
We let $\mathbf{Vect}$ be the algebraic stack (for the Zariski topology) sending a discrete animated ring $A$ to the anima of vector bundles on $A$. We will see $\mathbf{Vect}$ as an analytic stack via the functor \cite[Corollary 6.4.5]{SolidNotes}.
\end{definition}

\begin{lemma}\label{LemmaVectStructure}
Let $\GL_{n}$  be the general linear algebraic group over $\Z$ of rank $n$, and let $\ob{St}_n$ be the standard representation of $\GL_{n}$ seen as a vector bundle over $B \GL_{n}$. Then the natural map 
\[
\bigsqcup_{n} B\GL_n \to \mathbf{Vect}
\]
is an equivalence of algebraic stacks. 
\end{lemma}
\begin{proof}

The rank of a vector bundle on a scheme $X$ is a locally constant function for the Zariski topology. This produces a natural disjoint union decomposition 
\[
\mathbf{Vect}=\bigsqcup_{n} \mathbf{Vect}^{\ob{rk}=n}.
\]
For each $n$, let $\psi\colon \Spec \bb{Z}\to \mathbf{Vect}^{\ob{rk}=n}$ denote the trivial rank $n$-vector bundle $\bb{A}^n_{\bb{Z}}$ over $\bb{Z}$. We claim that $\psi$ is surjective, and that the group of automorphisms is precisely $\GL_n$, this would yield the isomorphism $B\GL_n\xrightarrow{\sim}  \mathbf{Vect}^{\ob{rk}=n}$, where the universal rank $n$-vector bundle over $B\GL_n$ is precisely the quotient $\GL_n\backslash \bb{A}^n_{\bb{Z}}$ with respect to the standard left action of $\GL_n$ on $\bb{A}^n_{\bb{Z}}$.\footnote{We note that the choice of the left standard representation is just a convention.} 

 The surjectivity follows from the fact that any vector bundle $V$ over an animated ring $A$ is isomorphic to a trivial vector bundle $A^n$ (for $n$ the locally constant rank of $V$) locally in the Zariski topology. The injectivity follows from the fact that the automorphisms of the rank $n$  trivial vector  bundle is precisely $\GL_n$ by definition. 
\end{proof}

%\begin{theorem}\label{TheoremCartierDualityAlgebraic}
%Let $V/\mathbf{Vect}$ be the universal vector bundle and let $V^{*}$ be its $\bb{G}_{a}$-linear dual. Set $V^{\sharp}:=V\otimes_{\bb{G}_a} \bb{G}_{a}^{\sharp}$ and $\widehat{V}^*:=V^{*}\otimes_{\bb{G}_a} \widehat{\bb{G}}_a$. Consider the  bilinear pairing 
%\[
%\Psi\colon V^{\sharp}\otimes_{\bb{G}_a} \widehat{V}^{*}\to \bb{G}_{a}^{\sharp}\otimes_{\bb{G}_a} \widehat{\bb{G}}_a \to \bb{G}_m
%\]
%over $\mathbf{Vect}$, where the first map arises from the natural pairing of $\bb{G}_a$-modules, and the last map is the linear map from \Cref{RemarkStackLinear}.  Then $\Psi$ gives rise to  $1$-categorical  Cartier dualities in the category of kernels $\ob{K}_{\ob{D}, \mathbf{Vect}}$  
%\[
%[V^{\sharp}]^*\cong [B\widehat{V}^*]_!\mbox{ and } [\widehat{V}^*]^*\cong [BV^{\sharp}]_!
%\]
%in the sense of  \Cref{DefinitionCartierDuals}. In particular, taking base change to $\Q^{\cond}$, we have Cartier dualities in the kernel category 
%\[
%[V]^*\cong [B\widehat{V}^*]_!\mbox{ and } [\widehat{V}^*]^*\cong [BV]_!
%\]
%\end{theorem}
%\begin{proof}
%By applying the same descent technique of \Cref{TheoCartoerDualityTori}, we formally reduce to the case of a trivial vector bundle $V$. In this case the pairing is a product of pairings for each component, so we can assume that $V=\bb{G}_a$, in which case te theorem follows from \Cref{ExAlgebraicVectorBundles}. The last observation about base change to $\Q^{\cond}$ follows from \Cref{CorBasicAlgChar0}. 
%\end{proof}

For stating the  main stacky Cartier duality theorem of this paper it is convenient to briefly  discuss what the \textit{transmutation} along a ring stack is.

\begin{construction}[{\cite[Remark 2.3.8]{BhattGauges}}]\label{ConstructionTransmutation}
Let $S\in \cat{AnStk}$ be a base analytic stack. A ring stack  $\n{R}/S$  is a functor 
\[
\n{R}\in \ob{Fun}_{\Sigma}(\cat{Poly}^{\ob{ft},\op}, \cat{AnStk}_{/S})
\]
preserving finite   products, where $\cat{Poly}^{\ob{fy},\op}$ is the category of polynomial algebras over $\bb{Z}$ in finitely many variables.  Equivalently, $\n{R}$ is a sheaf $\n{R}\colon \cat{AnStk}_{/S}\to \cat{Ring}$ where $\cat{Ring}$ is the category of animated rings.  Indeed, given a presentable category $\n{C}$, the category of   $\n{C}$-valued sheaves on $\cat{AnStk}_{/S}$  is precisely  the category $\ob{Fun}^R(\cat{AnStk}_{/S}^{\op}, \n{C})$ of limit preserving functors  from $\cat{AnStk}_{/S}^{\op}\to \n{C}$. Now, since $\cat{Ring}$ is the animation of $\cat{Poly}^{\ob{ft}}$, we have a natural equivalence of categories 
\[
\cat{Ring}=\ob{Fun}_{\Sigma}(\cat{Poly}^{\ob{ft},\op} , \cat{Ani}).
\]
Thus, an easy adjunction shows that 
\[
\begin{aligned}
\ob{Fun}^R(\cat{AnStk}_{/S}^{\op}, \cat{Ring} ) & = \ob{Fun}^R(\cat{AnStk}_{/S}^{\op}, \ob{Fun}_{\Sigma}(\cat{Poly}^{\ob{ft},\op} , \cat{Ani})) \\ 
&  = \ob{Fun}_{\Sigma}(\cat{Poly}^{\ob{ft},\op} ,   \ob{Fun}^R(\cat{AnStk}_{/S}^{\op}, \cat{Ani})) \\ 
& =  \ob{Fun}_{\Sigma}(\cat{Poly}^{\ob{ft},\op} ,  \cat{AnStk}_{/S})
\end{aligned}
\]
where in the last equivalence we used that $\cat{Ani}$-valued sheaves on a topos $\n{X}$ is $\n{X}$ itself.

Let  $\n{R}/S$ be a ring stack over  $S$,  let $\cat{Aff}=\cat{Ring}^{\op}$ denote the category of  affine schemes and $\n{P}(\cat{Aff})=\ob{Fun}(\cat{Ring}, \cat{Ani})$ its category of presheaves. We can produce a natural   \textit{transmutation functor}
\begin{equation}\label{eqTransmutationFunctor}
(-)^{\n{R}}\colon \n{P}(\cat{Aff})\to \cat{AnStk}_{/S}  
\end{equation}
sending a prestack $F$ to the sheafification for the topology of analytic stacks to the presheaf sending a map of analytic stacks  $X\to S$ to $F(\n{R}(X))$. Since limits and colimits in presheaves are computed pointwise, and sheafification commutes with finite limits and colimits, the functor $(-)^{\n{R}}$ also commutes with finite limits and colimits.

 Given a Grothendieck topology $\n{T}$ on $\cat{Aff}$, we say that \textit{$\n{R}$ satisfies $\n{T}$-descent} if the functor \eqref{eqTransmutationFunctor}, which is a pullback map of topoi,  factors through the category of $\n{T}$-sheaves 
 \begin{equation}\label{eqooo3jlnkenkankasse4rrrtoposiiiii}
 \n{P}_{\n{T}}(\cat{Aff})\to \cat{AnStk}_{/S}. 
 \end{equation}
 If $\n{R}$ satisfies $\n{T}$-descent, the  functor \eqref{eqooo3jlnkenkankasse4rrrtoposiiiii} is also left exact and colimit preserving by definition. 

\end{construction}

\begin{example}\label{ExampleRingStackZariskiDescent}
The following are examples of ring stacks satisfying (at least!) Zariski descent:

\begin{enumerate}
\item  The algebraic affine line $\bb{G}_{a}\to \AnSpec \Z^{\cond}$.

\item  The solid affine line $\bb{G}_{a,\sol}\to \AnSpec \Z_{\sol}$.

 \item The  overconvergent closed disc  $\bb{D}^{\leq 1}_{R}\to \AnSpec R$ of norm $\leq 1$, with $R=\Z((\pi))$  endowed with the induced solid structure.
 
 \item The analytic affine line $\bb{G}_{a,R}^{\an}\to \AnSpec R$ with $R$ as in (3).
 
 \item The locally analytic $p$-adic integers $\bb{Z}_p^{la}\to \AnSpec \Q_{p,\sol}$.

\end{enumerate}

\end{example}

\begin{theorem}\label{TheoCartierDualityVectorBundles}
Let $\n{R}/S$ be one of the ring stacks of \Cref{ExampleRingStackZariskiDescent}. Let $V^{\n{R}}$ be the transmutation of the universal vector bundle over $\mathbf{Vect}^{\n{R}}$ and let $V^{\vee,\n{R}}$ be the transmutation of its dual.

\begin{enumerate}

\item Let $\n{R}/S=\bb{G}_{a}/\AnSpec \Z^{\cond}$ and let us denote $V\otimes_{\bb{G}_a} \bb{G}_a^{\sharp}=V^{\sharp}$ and $V\otimes_{\bb{G}_a} \widehat{\bb{G}}_a=\widehat{V}$. The pairing 
\[
V^{\sharp}\otimes_{\bb{G}_{a}} \widehat{V}^{\vee} \xrightarrow{\langle -,- \rangle} \bb{G}_{a}^{\sharp}\otimes_{\bb{G}_a} \widehat{\bb{G}}_a\xrightarrow{\exp} \bb{G}_m
\]
of \Cref{ConstructionExponential} and  \Cref{RemarkStackLinear} gives rise to $1$-categorical Cartier dualities in the kernel category $\ob{K}_{\ob{D}, \mathbf{Vect}}$
\[
[BV^{\sharp}]_!\cong [\widehat{V}^{\vee}]^* \mbox{ and } [V^{\sharp}]_!\cong [B\widehat{V}^{\vee}]^*
\]
in the sense of  \Cref{DefinitionCartierDuals}.

\item Let $\n{R}/S=\bb{D}^{\leq 1}_{R}/\AnSpec R$ be as in \Cref{ExampleRingStackZariskiDescent} (3). Given an algebraic stack $X$ over $\Z$ we let $X^{\leq 1}$ denote the transmutation along $\bb{D}_R^{\leq 1}$.  For $r\in (0,\infty)$ let us denote $V^{\sharp,\leq r}:=V^{\leq 1}\otimes_{\bb{D}^{\leq 1}_R} \bb{D}^{\sharp, \leq r}_{R}$ and $V^{<r}:=V^{\leq 1}\otimes_{\bb{D}^{\leq 1}_R} \bb{D}^{<r}_{R}$ as $\bb{D}^{\leq 1}_R$-modules over $\mathbf{Vect}^{\leq 1}$. The pairing 
\[
V^{\sharp,\leq r}\otimes_{\bb{D}_{a,R}^{\leq 1}} V^{\vee,<1/r} \xrightarrow{\langle -,- \rangle} \bb{D}_{a,R}^{\sharp,\leq r}\otimes_{\bb{D}_{R}^{\leq 1/r}} \bb{D}_a^{<1/r} \xrightarrow{\exp} \bb{G}_m
\]
of \Cref{ConstructionPairingDiscCase} gives rise to $1$-categorical Cartier dualities in the kernel category $\ob{K}_{\ob{D},\cat{Vect}^{\leq 1}}$
\[
[BV^{\sharp,\leq r}]_!\cong [V^{\vee,<1/r}]^* \mbox{ and } [V^{\sharp,\leq r}]_!\cong [BV^{\vee,<1/r}]^*.
\]

\item Let $\n{R}/S$ be any of the rings \Cref{ExampleRingStackZariskiDescent}  (2), (4) and (5).  Following \Cref{ExSolidVectorBundle,ExAnalyticVectorBundle,ExZpLAvectorbundle},  the ring stack $\n{R}$ admits a Cartier dual $\bb{D}(\n{R})/S$ having the structure of an $\n{R}$-module and endowed with an exponential map $\exp\colon \bb{D}(\n{R})\to \bb{G}_m$, such that the Cartier duality pairing between $\n{R}$ and $\bb{D}(\n{R})$ is the composite of the multiplication and the exponential
\[
\n{R}\times \bb{D}(\n{R})\to \bb{D}(\n{R})\to \bb{G}_m.
\]
Explicitly, we have  $\bb{D}(\bb{G}_{a,\sol})=\widehat{\bb{G}}_{a,\sol}^{\sharp}$, $\bb{D}(\bb{G}_{a,R}^{\an})=\bb{G}_{a}^{\sharp,\dagger}$ and $\bb{D}(\bb{Z}_p^{la})=1+\bb{D}_{\Q_p}^{<1}$. Then the pairing 
\[
V^{\n{R}}\otimes_{\n{R}} (V^{\vee,\n{R}}\otimes_{\n{R}}  \bb{D}(R))\xrightarrow{\langle -,- \rangle} \bb{D}(R)\xrightarrow{\exp} \bb{G}_m
\]
gives rise to $1$-categorical Cartier dualities in the category of kernels $\ob{K}_{\ob{D},\mathbf{Vect}^{\n{R}}}$
\[
[B(V^{\vee,\n{R}}\otimes_{\n{R}} \bb{D}(R))]_!\cong [V^{\n{R}}]^* \mbox{ and } [V^{\vee,\n{R}}\otimes_{\n{R}} \bb{D}(R)]_!\cong [BV^{\n{R}}]^*.
\] 

\end{enumerate}

\end{theorem}
\begin{proof}
Following the same  argument of \Cref{TheoCartoerDualityTori} of  $\ob{D}^*$-descent on the basis $\mathbf{Vect}^{\n{R}}$,  it suffices to show the simplest case of Cartier duality when $V$ is a trivial vector bundle of rank $1$ in any of the cases of   \Cref{ExSolidVectorBundle,ExDiscVectorBundle,ExAnalyticVectorBundle,ExZpLAvectorbundle}, this reduces to  \Cref{PropBasicAlgVectCD,PropBasicSolidVectCD,PropBasicDiscVectCD,PropBasicAnVectCD,PropBasicZpVectCD} respectively. 
\end{proof}

\begin{remark}\label{RemarkCartierDualityObservationsProblems}
We continue the discussion of \Cref{RemarkCartierDualityOtherVectorBundles}. The reason why \Cref{TheoCartierDualityVectorBundles} is divided in three different points is the fact that the Cartier dual of $\bb{G}_a$  as a Gestalten  is not classical unless one works  in characteristic zero, in which case $\bb{G}_{a,\bb{Q}}=\bb{G}_{a,\bb{Q}}^{\sharp}$ and the Cartier dual is computed in  (1) of \textit{loc. cit.}  I would expect that the similar problem occurs for the stack $\bb{D}^{\leq 1}_{R}$ over an arbitrary characteristic, and that only after a suitable localization of $R$  in characteristic zero the Cartier dual of   $\bb{D}^{\leq 1}_{R}$ can be computed as an analytic stack; this is the case for the specialization to $\bb{Q}_p$, in which case we have an equivalence $\bb{D}^{\sharp,\leq 1}_{\bb{Q}_p}= \bb{D}^{\leq p^{1/(p-1)} }_{\bb{Q}_p}$ by \Cref{RemarkBaseChangeQp}. 
\end{remark}

\subsection{Cartier duality for gerbes}\label{ss:CDGerbesVB}

In this section we discuss a general Cartier duality for gerbes on analytic stacks. We start with some general discussion on commutative  gerbes. Let $\n{X}$ be an $\infty$-topos and let $\ob{Shv}(\n{X}, \ob{Sp}_{\geq 0})$ be the category of sheaves of connective spectra on $\n{X}$ or, equivalently, grouplike commutative monoids on $\n{X}$.  We let $*_{\n{X}}$ denote the final object of $\n{X}$.  We see $\ob{Shv}(\n{X}, \ob{Sp}_{\geq 0})$ as a symmetric monoidal category via the smashing tensor product.   Let $R\in \ob{CAlg}(\ob{Shv}(\n{X}, \ob{Sp}_{\geq 0}))$ and let $\ob{Mod}_{R,\geq 0}(\n{X})$ be the category of $R$-modules on $\ob{Shv}(\n{X}, \ob{Sp}_{\geq 0})$.

 Let $M,N,\omega \in \ob{Mod}_{R,\geq 0}(\n{X})$  and suppose we are given with an $R$-linear pairing  $\Psi\colon M\otimes_R N\to \omega$.  Consider the slice topos $\n{Y}:= \n{X}_{/B^2M}$, and for $X\in \n{X}$ let $X_{\n{Y}}=X\times B^2M$ be the pullback to $\n{Y}$. The object $R_{\n{Y}}$ has a natural structure of commutative ring on the topos $\n{Y}$, the objects  $M_{\n{Y}}, N_{\n{Y}},\omega_{\n{Y}}$ are naturally $R_{\n{Y}}$-modules, and we have an $R_{\n{Y}}$-linear pairing $\Psi \colon M_{\n{Y}}\otimes_{R_{\n{Y}}} N_{\n{Y}}\to \omega_{\n{Y}}$.  We have a natural section $*_{\n{Y}}\to B^2 M_{\n{Y}}$ given by the diagonal map of $B^2M$, this maps extends to a morphism  $\phi\colon R_{\n{Y}}\to B^2M_{\n{Y}}$ of $R_{\n{Y}}$-modules. Tensoring $\phi$ with  $N_{\n{Y}}$ and  composing with $\Psi$ we get a morphism $\psi\colon N_{\n{Y}}\to B^2 \omega_{\n{Y}}$ of $R_{Y}$-modules fitting in a commutative diagram  
 \begin{equation}\label{eqpmaosnoapsdmasf}
 \begin{tikzcd}
 B^2M_{\n{Y}}\otimes_{R_{\n{Y}}}  N_{\n{Y}}  \ar[r,"\Psi"] & B^2\omega_{\n{Y}} \\ 
 N_{\n{Y}}\ar[u,"{\phi\otimes \id_N}"] \ar[ru,"\psi"'] & 
 \end{tikzcd}
 \end{equation}

\begin{construction}\label{DefinitionNotationGerbes}
Keep the previous notation. We construct the $R_{\n{Y}}$-modules $\s{N}_{\psi}$ and $\s{M}$ as  the pullbacks 
\[
\begin{tikzcd}
\s{N}_{\psi} \ar[r] \ar[d]  & *_{\n{Y}} \ar[d] & &  \s{M} \ar[r]\ar[d] &  *_{\n{Y}}  \ar[d ]\\ 
 N_{\n{Y}} \ar[r,"\psi"] & B^2\omega_{\n{Y}} & & R_{\n{Y}} \ar[r,"\phi"] & B^2M_{\n{Y}}
\end{tikzcd}
\]
(from the construction it is clear that $\s{N}_{\Psi}$ does depend on the pairing $\Psi$ while $\s{M}$ only depends on the $R$-module $M$).  
\end{construction} 

\begin{lemma}\label{LemmaPairingKey}
Keep the previous notation. There is a natural pairing of $R_{\n{Y}}$-modules $\s{N}_{\psi}\otimes_{R_{\n{Y}}} \s{M}\to B\omega_{\n{Y}}$ fitting in a  commutative diagram   of $R_{\n{Y}}$-modules
\[
\begin{tikzcd}
B\omega_{\n{Y}}\otimes_{R_{\n{Y}}} \s{M} \ar[d] \ar[r] & B\omega_{Y}\otimes_{R_{\n{Y}}} R_{\n{Y}} \ar[d,"\id"] \\
\s{N}_{\psi}\otimes_{R_{\n{Y}}} \s{M} \ar[r] & B\omega_{\n{Y}} \\
\s{N}_{\psi} \otimes_{R_{\n{Y}}} BM_{\n{Y}} \ar[u]  \ar[r]  &  N\otimes_{R_{\n{Y}}} BM_{\n{Y}}  \ar[u,"\Psi"']
\end{tikzcd}
\]
where the left vertical maps and  upper and lower  horizontal maps arise from the fiber squares of  \Cref{DefinitionNotationGerbes}.
\end{lemma}
\begin{proof}
The map $N_{\n{Y}}\to B^2\omega_{\n{Y}}$ is an epimorphism as already $*_{\n{Y}}\to B^2\omega_{\n{Y}}$ is so. Then,  we have  a  fiber/cofiber sequence $N_{\n{Y}}\xrightarrow{\psi} B^2\omega_{\n{Y}}\to B\s{N}_{\psi}$ of $R_{\n{Y}}$-modules. Taking  $\iHom_{R_{\n{Y}}}(-,B^2\omega_{\n{Y}})$  in connective $R_{\n{Y}}$-modules we obtain a fiber sequence 
\[
\iHom_{\n{R}_{\n{Y}}}(B\s{N}_{\psi}, B^2\omega_{\n{Y}})  \to \iHom_{\n{R}_{\n{Y}}}(B^2\omega_{\n{Y}}, B^2\omega_{\n{Y}})  \to \iHom_{\n{R}_{\n{Y}}}(N_{\n{Y}}, B^2\omega_{\n{Y}}). 
\]
By \eqref{eqpmaosnoapsdmasf} we have a commutative diagram  (depicted by solid arrows) that extends to a morphism of fiber sequences (depicted by dashed arrows)
\[
\begin{tikzcd}
\s{M} \ar[d,dashed] \ar[r,dashed] & \iHom_{\n{R}_{\n{Y}}}(B\s{N}_{\psi}, B^2\omega_{\n{Y}})  \ar[d,dashed] \\
R_{\n{Y}} \ar[r] \ar[d,"\phi"] & \iHom_{\n{R}_{\n{Y}}}(B^2\omega_{\n{Y}}, B^2\omega_{\n{Y}}) \ar[d] \\ 
B^2 M_{\n{Y}} \ar[r] & \iHom_{\n{R}_{\n{Y}}}(N_{\n{Y}}, B^2\omega_{\n{Y}})
\end{tikzcd}
\]
where the lower horizontal map is induced by the pairing $\Psi$, and the middle map is induced by the identity of $B^2\omega_{\n{Y}}$.  This produces a  pairing $\s{M}\otimes_{R_{\n{Y}}} B\s{N}_{\psi}\to B^2\omega_{\n{Y}}$ whose loops is the desired pairing of the statement. 
\end{proof}

\begin{proposition}\label{PropDualityForGerbes}
Keep the convention of analytic stacks of \Cref{ss:AnStk}.  Let $S\in \n{C}\subset \cat{AnStk}$ be an analytic stack, let $\n{R}/S$ be a $!$-able ring stack over $S$, let $M,N/S$ be $!$-able $\n{R}$-modules on $\cat{AnStk}_{/S}$.  Suppose that there is a $!$-able $\n{R}$-module $\bb{D}(\n{R})$ and a morphism of $\Z$-linear modules $\exp\colon \bb{D}(\n{R})\to \bb{G}_m$ such that the induced pairing  $\exp(XY)\colon \n{R}\times \bb{D}(\n{R})\to \bb{D}(\n{R})\to \bb{G}_m$  gives rise to  a $1$-categorical Cartier duality in the kernel category $\ob{K}_{\ob{D},S}$
\[
[\n{R}]^*\cong [B(\bb{D}(\n{R}))]_!.
\] 
Let $\Psi \colon M\otimes_{\n{R}}N\to \bb{D}(\n{R})$ be an $\n{R}$-linear map   whose composite with $\exp$ gives rise to a  $1$-categorical Cartier duality 
\[
[BM]^*\cong [N]_!.
\]

Let $X=B^2M$ be the analytic stack over $S$ corepresenting gerbes banded by $M$ and  let   $\s{N}_{\psi}$  and $\s{M}$ be the extensions 
\begin{equation}\label{eqpo9jo3nkqlweqwdm}
B \bb{D}(\n{R})_X \to  \s{N}_{\psi}\to N_{X} \mbox{ and } BM_X\to \s{M} \to \n{R}_{X}
\end{equation}
of \Cref{DefinitionNotationGerbes} as $\n{R}_X$-linear stacks over $X$.   Then the pairing  $\s{N}_{\psi}\otimes_{\n{R}_{X}} \s{M}\to  B \bb{D}(\n{R})_X$ of \Cref{LemmaPairingKey} gives rise to a $1$-categorical Cartier duality in the kernel category $\ob{K}_{\ob{D},X}$
\begin{equation}\label{eqpmoasbfoa3ihnwd}
[\s{N}_{\psi}]^* \cong [\s{M}]_!
\end{equation}
as in  \Cref{DefinitionCartierDuals}.
\end{proposition}
\begin{proof}
By \Cref{RemarkDescentBaseChange} we can prove the $1$-categorical Cartier duality locally on the base $X$. Since $X=B^2M$, the map  $S\to X$ is an epimorphism, and it suffices to show that \eqref{eqpmoasbfoa3ihnwd} is an equivalence after pulling back to $S$. In that case, we can assume that the extensions \eqref{eqpo9jo3nkqlweqwdm} are split, and by \Cref{LemmaPairingKey} that the pairing $\s{N}_{\psi}\otimes_{\n{R}_{X}} \s{M}\to  B \bb{D}(\n{R})_X$ is the direct sum of the Cartier duality pairings of $\n{R}$ and $\bb{D}(\n{R})$, and of $M$ and $N$ respectively. The proposition follows since these last pairings give rise to $1$-categorical Cartier dualities by assumption. 
\end{proof}

\begin{example}\label{ExampleCartierDualityGerbes}
We can apply \Cref{PropDualityForGerbes} to the stack $S=\mathbf{Vect}^{\n{R}}$ given  as the transmutation   of the stack of vector bundles, where $\n{R}$ is a ring stack as  in  \Cref{ExampleRingStackZariskiDescent}. We can then take $N$ and $M$ to be the modifications of the universal vector bundle over $S$ appearing in the Cartier duality of  \Cref{TheoCartierDualityVectorBundles}. 
\end{example}

\subsection{Application: Cartier duality for the Hodge-Tate stack and the Simpson gerbe}\label{ss:CartierDualityHTSImp}

In  joint work in progress with Ansch\"utz, Le Bras and Scholze on the analytic prismatization, we introduce the Hodge-Tate stack for a smooth rigid variety $X$ over a complete non-archimedean algebraically closed field $C$ over $\Q_p$. In the dual side, Bhatt and Zhang have constructed the Simpson gerbe $\s{S}_X\to T^{*,\an}_{X}(-1)$, that is, a $B\bb{G}_m$-torsor over the analytic cotangent bundle of $X$ (tensored with the inverse of the Tate twist). In this section we explain how to obtain a Cartier duality between both constructions as a consequence of the Cartier duality for gerbes of \Cref{ss:CDGerbesVB}.

First, let us give a quick ad-hoc definition of the Hodge-Tate stack that is good enough for smooth rigid spaces. Let $\nu\colon X_{v}\to X_{\et}$ be the projection from the $v$ to the \'etale site of $X$.  By \cite[Proposition 2.23]{ScholzePerfectoidSurvey} one has a natural equivalence $R^1\nu_{*} \widehat{\s{O}}_X = \Omega^1_X(-1)$, where  $(-1)$ refers to the inverse of the Tate twist. In particular, we have a natural section $\eta_{\ob{HT}}\in H^1_v(X, T_X(1)\otimes_{\s{O}_X} \widehat{\s{O}}_X)$. On the other hand, by \cite[Remark 4.6.6 and Proposition 4.6.7]{anschutz2025analytic} we have a natural equivalence
\[
R\Gamma_{v}(X, \widehat{\s{O}}_X)=R\Gamma(X, \bb{G}_a^{\an,\dR})
\]
where the right term is the cohomology of the sheaf $\bb{G}_a^{\an,\dR}$ for the $!$-topology (see \Cref{RemarkCaseHonestRigidSpaces}). From the fiber sequence $\bb{G}_{a}^{\dagger}\to \bb{G}_a^{\an}\to \bb{G}_a^{\an,\dR}$ we see that the class $\eta_{\ob{HT}}$ produces a class (that we denote in the same way) in $H^2(X, T^{\an}_X(1)\otimes_{\bb{G}_a^{\an}}\bb{G}_a^{\dagger})$.

\begin{definition}\label{DefHodgeTateStack}
Denote $ T^{\dagger}_X(1):=T^{\an}_X(1)\otimes_{\bb{G}_a^{\an}}\bb{G}_a^{\dagger}$. The \textit{analytic Hodge-Tate stack of $X$} is the pullback in Gelfand stacks (cf.  \cite{anschutz2025analytic})
\[
\begin{tikzcd}
X^{\HT} \ar[r] \ar[d] & X\ar[d,"e"] \\ 
X \ar[r,"\eta_{\HT}"] & B^2 T^{\dagger}_{X}(1)
\end{tikzcd}
\]
where $e$ is the natural projection map, and $\eta_{\HT}$ is the map induced by the class in $H^2(X, T^{\dagger}_X(1))$  with same name.
\end{definition}

\begin{remark}\label{RemarkCaseHonestRigidSpaces}
The formalism of \cite{anschutz2025analytic} only uses induced analytic ring structures from $\bb{Q}_{p,\sol}$ so  a priori it can be directly applied only to partially proper smooth rigid spaces or to smooth dagger spaces. However, we can go around that problem as follows: let $X$ be a smooth rigid space and let $\overline{X}$ be Huber's compactification of $X$. The space $\overline{X}$ is a partially proper adic space and therefore defines a (qfd) Gelfand stack, indeed, for this it suffices to consider the Gelfand stack $\overline{X}$ obtained by descent on affinoids, and sending $\Spa(A,A^+)\mapsto \ob{GSpec} A$.  Considering   all objects as analytic stacks, we have a natural map  $X\to  \overline{X}$. 

 Then, the  $\bb{G}_a^{\dR}$-valued cohomology for $\overline{X}$ in the $!$-topology is exactly the same as $\widehat{\s{O}}$-cohomology of the $v$-stack $\overline{X}^{\diamond}$. Since $X$ and $\overline{X}$ have the same $\widehat{\s{O}}$-cohomology, namely $\widehat{\s{O}}$ is an overconvergent sheaf, we can construct a Hodge--Tate stack $\overline{X}^{\HT}$ for $\overline{X}$ in the exact same way as in \Cref{DefHodgeTateStack} and with the same properties, i.e., the map $\overline{X}^{\HT}\to \overline{X}$ is still a gerbe banded by $T_{X}^{\dagger}(1)$. Taking pullbacks along $X\to \overline{X}$, we obtain the Hodge--Tate stack for the original rigid space $X$ without additional assumptions. 
\end{remark}

Following the same definition of the Hodge-Tate stack, we give an ad-hoc definition of the Simpson gerbe that will fit in our  setting of Cartier duality. Tensoring the map $\eta_{\HT}\colon X\to B^2 T^{\dagger}_X(1)$ with the analytic cotangent bundle $T^{*,\an}_{X}(-1)$ produces a map $T^{*,\an}_{X}(-1)\to B^2 T^{\dagger}_X(1)\times_X T^{*,\an}_X(-1)$, we compose with the natural pairing 
\[
B^2T^{\dagger}_X(1)\times_X  T^{*,\an}_X(-1)\xrightarrow{\langle -,-\rangle} B^2\bb{G}_{a,X}^{\dagger} \xrightarrow{\exp} B^2\bb{G}_{m,X}
\]
where the classifying stacks are taken relative to $X$, and $\bb{G}_{m,X}$ is an algebraic multiplicative group relative to $X$. We denote $\ob{Simp}_X\colon T^{*,\an}_X(-1)\to B^2 \bb{G}_{m,X}$ the resulting map.

\begin{definition}\label{DefinitionSimpsonGerbe}
The Simpson gerbe $\s{S}_X$ of $X$ is the pullback square 
\[
\begin{tikzcd}
\s{S}_X  \ar[r] \ar[d] & X \ar[d,"e"] \\
T^*_{X}(-1) \ar[r,"\ob{Simp}_X"] & B^2\bb{G}_{m,X}
\end{tikzcd}
\]
where $e\colon X\to B^2\bb{G}_{m,X}$ is the natural quotient map. 
\end{definition}

\begin{remark}
Bhatt and Zhang's definition of the Simpson's gerbe is different from  \Cref{DefinitionSimpsonGerbe}. In the following paragraph we discuss their definition for partially proper smooth rigid spaces, see \cite[Remarks 12.4.7 and 12.4.8]{BhattpAdicHodge2026}: let $X_!$ denote the site given by affinoid   nilperfectoid spaces over $X$ endowed with the $!$-topology, and $X_{\ob{arc}}$ the site of affinoid  perfectoid spaces over $X$ endowed with the arc-topology, see \cite{anschutz2025analytic}. We have geometric morphisms of topoi $\widetilde{X}_{\ob{arc}}\xrightarrow{\eta} \widetilde{X}_{!}\xrightarrow{\rho} \widetilde{X}_{\et}$ where the map $\rho$ arises from the fact that any \'etale cover is a $!$-cover, and the map $\eta$  is given by the de Rham stack \cite[Remark 4.0.1]{anschutz2025analytic}. By \cite[Proposition 4.5.7]{anschutz2025analytic} the functor $\eta$ is fully faithful, and the sheaf $\widehat{\s{O}}_X^{\times}$ on $\widetilde{X}_{\ob{arc}}$ is nothing but $\eta^* \bb{G}_m^{\an,\dR}$. Thus, we have an equivalence of \'etale sheaves of $\Z$-modules
\[
R(\rho\circ \eta)_* \widehat{\s{O}}^{\times}_X = R\rho_* \bb{G}_m^{\an,\dR}. 
\]
On the other hand, we have a fiber sequence of sheaves on $X_!$ 
\begin{equation}\label{eqShortExactGmdR}
\bb{G}_{a}^{\dagger}\xrightarrow{\exp} \bb{G}_m^{\an}\to \bb{G}_{m}^{\an,\dR}
\end{equation}
which produces a fiber sequence on \'etale sheaves
\[
R\rho_* \bb{G}_{a}^{\dagger} \to R\rho_* \bb{G}_{m}^{\an} \to R\rho_* \bb{G}_{m}^{\an,\dR}.
\]
Similarly, the fiber sequence $\bb{G}_a^{\dagger}\to \bb{G}_a^{\an}\to \bb{G}_{a}^{\dR}$ produces a fiber sequence of \'etale sheaves
\[
R\rho_*\bb{G}_a^{\dagger} \to R\rho_*\bb{G}_a^{\an}\to R\rho_*\bb{G}_a^{\an,\dR}.
\]
Since the structural sheaf of an analytic ring satisfies $!$-descent, we have that  $ R\rho_*\bb{G}_a^{\an}= \bb{G}_{a,\et}$. Since $X$ is a smooth rigid space, we also have that $R^0\rho_* \bb{G}_a^{\dagger}=0$ and $R^0\rho_* \bb{G}_a^{\an,\dR}=\bb{G}_{a,\et}$. Thus, we have an equivalence of sheaves on $X_{\et}$
\[
\tau_{\leq -1} R\rho_* \bb{G}_a^{\an,\dR}\xrightarrow{\sim}  R \rho_* \bb{G}_{a}^{\dagger}[1]. 
\]
In particular, $R^{i+1}\rho_* \bb{G}_a^{\dagger}= R^{i}\rho_* \bb{G}_a^{\an,\dR}=\Omega^i_X(-i)$ for $i\geq 1$.  Hence, the fiber sequence \eqref{eqShortExactGmdR} produces a connecting morphism
\[
\Omega_X^1(-1)=R^{2}\rho_* \bb{G}_a^{\dagger}\to R^2  \rho_* \bb{G}_m^{\an}
\]
in $X_{\et}$, which gives rise to a map of analytic stacks $T^{*,\an}_X(-1)\to B^2\bb{G}_{m,X}^{\an}$, i.e. a $\bb{G}_m^{\an}$-gerbe over $T^{*,\an}_X(-1)$.  A bookkeeping of the construction  shows that this gerbe is precisely that of \Cref{DefinitionSimpsonGerbe} after taking the pushout along the natural morphism $\mathbb{G}_m^{\an}\to \mathbb{G}_m$ from the analytic to the algebraic multiplicative group. 
\end{remark}

\begin{remark}\label{RemarkAnGmGerbe}
By construction the Simpson gerbe of Bhatt and Zhang is a $\mathbb{G}_{m}^{\an}$-gerbe, that is, a gerbe  banded by the \textit{analytic multiplicative group}. The gerbe of  \Cref{DefinitionSimpsonGerbe} is banded by the \textit{algebraic multiplicative group $\mathbb{G}_m$}. Both gerbes are refined by an $\mathbb{G}_m^{\dagger}\cong \mathbb{G}_a^{\dagger}$-gerbe over $T^*_X(-1)$ by applying  \Cref{DefinitionNotationGerbes} to the Hodge-Tate stack where $T_X^{\dagger}(1)$ is considered as $\mathbb{G}_a^{\an}$-module. 
\end{remark}

As  special case of \Cref{PropDualityForGerbes}, we deduce the Cartier duality between the Hodge-Tate stack and the Simpson gerbe. For that, let us write $S=B^2T^{\dagger}_{X}(1)$, and let $\Delta_S\colon S\to (B^2T^{\dagger}_{X}(1))_S$ be the natural section induced by the diagonal map. Let $\bb{Z}_{\ob{Betti},S}\to  (B^2T^{\dagger}_{X}(1))_S$ be the natural $\Z$-linear  extension of $S$ and let  $\s{M}$ be the fiber of the map  $\bb{Z}_{\ob{Betti},S}\to  (B^2T^{\dagger}_{X}(1))_S$. The $\Z$-module   stack is an extension 
\[
BT^{\dagger}_X(1)_S\to \s{M}\to   \bb{Z}_{\ob{Betti},S}.
\] 
We define the \textit{extended Hodge-Tate stack} as the pullback 
\[
\begin{tikzcd}
X^{\ob{HT}, \ob{ext}} \ar[r] \ar[d] &  \s{M} \ar[d] \\
X \ar[r,"\eta_{\HT}"] & S.
\end{tikzcd}
\] 
By construction, $X^{\ob{HT},\ob{ext}}$ is an extension of $\Z$-modules in analytic stacks over $X$
\[
BT^{\dagger}_X(1) \to  X^{\ob{HT},\ob{ext}}\to \bb{Z}_{\Betti,X}
\]
whose fiber at $1$ is precisely the Hodge-Tate stack  $X^{\ob{HT}}$. We have the following theorem.  

\begin{theorem}\label{CartierDualityHodgeTateSimpson}
There is a natural pairing of $\Z$-modules 
\[
X^{\HT,\ob{ext}}\otimes_{\Z} \s{S}_X\to B\bb{G}_m
\]
where $\s{S}_X$ is the Simpson gerbe of \Cref{DefinitionSimpsonGerbe} producing a $1$-categorical Cartier duality  in the kernel category $\ob{K}_{\ob{D},X}$
\begin{equation}\label{eqmpoasi3krnoqwdq}
[X^{\HT,\ob{ext}}]^* \cong [ \s{S}_X]_!
\end{equation}
as in \Cref{DefinitionCartierDuals}. In particular, passing to quasi-coherent sheaves, there is a natural decomposition  of $\ob{D}(T^{*,\an}_X(-1))$-linear categories 
\begin{equation}\label{eqk9jnoq1i3hnikqsda}
\ob{D}(\s{S}_X) = \prod_{n\in \Z} \ob{D}(\s{S}_X)^{\ob{wt}=n}
\end{equation}
such that:
\begin{enumerate}

\item For $n\in \Z$, let $X^{\HT, (n)}$ be the fiber at $n\in \Z_{\Betti,X}$ of $X^{\HT,\ob{ext}}$.  There is a natural equivalence of $\ob{D}(T_X^{*}(-1))$-linear categories $\ob{D}(X^{\HT,n})=\ob{D}(\s{S}_X)^{\ob{wt}=n}$, where $\ob{D}(T_X^{*,\an}(-1))\cong \ob{D}(BT_X^{\dagger}(1))$ acts on $\ob{D}(X^{\HT,n})$ via the  $!$-convolution arising from the $BT_X^{\dagger}(1)$-action on $X^{\HT,n}$.  

\item The $\ob{D}(T^*_X(-1))$-modules $\ob{D}(\s{S}_X)^{\ob{wt}=n}$ are invertible. Moreover,  $\ob{D}(\s{S}_X)^{\ob{wt}=0}\subset \ob{D}(\s{S}_X)$ is equivalent to the fully faithful inclusion $\ob{D}(T_X^{*,\an}(-1))\subset \ob{D}(\s{S}_X)$ obtained by pullback along the map $\s{S}_X\to T^{*,\an}_X(-1)$. 

\item  Given $n,m\in \Z$ there are natural equivalences  of $\ob{D}(T^{*,\an}_X(-1))$-linear categories 
\[
\ob{D}(\s{S}_X)^{\ob{wt}=n}\otimes_{\ob{D}(T^*_X(-1))} \ob{D}(\s{S}_X)^{\ob{wt}=m}=\ob{D}(\s{S}_X)^{\ob{wt}=n+m}.
\]

\end{enumerate}
\end{theorem}
\begin{proof}
The variety $X$ admits an open cover $\{X_i\}_{i}$ by affinoid analytic stacks. By \Cref{PropositionKunnethQuasiAffine} (with $S=X$, $S_i=X_i$ and $\overline{S}_i=X$) any $!$-able map $Y\to X$ satisfies K\"unneth, thus by  \Cref{PropKunnethKernel} the presentable category of kernels $\cat{Pr}_{\ob{D},X}$ over $X$ is equivalent to $\cat{Pr}_{\ob{D}(X)}$ and $\ob{K}_{\ob{D},X}\hookrightarrow \cat{Pr}_{\ob{D}(X)}$ is a $2$-fully faithful functor.

The first statement about Cartier duality is a  formal consequence of \Cref{TheoCartierDualityVectorBundles} and \Cref{PropDualityForGerbes} after taking pullback from the  the universal duality of \Cref{ExampleCartierDualityGerbes}. For (1), the equivalence \eqref{eqmpoasi3krnoqwdq} in $\ob{K}_{\ob{D},X}$ produces a natural equivalence of $\ob{D}(X)$-linear categories 
\[
\ob{D}(X^{\HT,\ob{ext}})\cong \ob{D}(\s{S}_X). 
\]
Then, looking at the fibers of the map $X^{\HT,\ob{ext}}\to \Z_{\Betti}$, we have the $\ob{D}(X)$-linear decomposition 
\[
\ob{D}(X^{\HT,\ob{ext}})=\prod_{n\in \Z} \ob{D}(X^{\HT,(n)}).
\]
This produces a decomposition of   $\ob{D}(\s{S}_X)$ as in \Cref{eqk9jnoq1i3hnikqsda} such that 
\[
\ob{D}(X^{\HT,(n)})\cong \ob{D}(\s{S}_X)^{\ob{wt}=n}
\]
by construction.  We want to see that the decomposition $\ob{D}(\s{S}_X) = \prod_{n\in \Z} \ob{D}(\s{S}_X)^{\ob{wt}=n}\cong \prod_{n\in \Z} \ob{D}(X^{\HT,(n)})$ is as $\ob{D}(T^{*,\an}_{X}(-1))$-linear categories. For that, notice that Cartier duality gives rise to a commutative  diagram of Hopf algebras in the category of kernels of $X$
\[
\begin{tikzcd}
{[T^*_X(-1)]^*} \ar[r] \ar[d] &{ [\s{S}_X]^*} \ar[r] \ar[d] & {[B\bb{G}_{m,X}]^*} \ar[d] \\ 
{[BT_X^{\dagger}(1)]_! } \ar[r] &  {[X^{\HT,\ob{ext}}]_!} \ar[r] & {[\bb{Z}_{\ob{Betti},X}]_! }
\end{tikzcd}
\]
where the vertical arrows are isomorphisms. In particular, the fibers of the map $X^{\HT,\ob{ext}}\to \bb{Z}_{\Betti,X}$ carry an action of $BT^{\dagger}_X(1)$ which translates in an action of $\ob{D}(BT^{\dagger}_X(1))$ by convolution, or equivalently, on a module structure over $\ob{D}(T^{*,\an}_X(-1))$.  

Part (2) follows from (3) since $\ob{D}(\s{S}_X)^{\ob{wt}=0}$ is the Cartier dual of $\ob{D}(BT_{X}^{\dagger}(1))$ which is nothing but $\ob{D}(T_X^{*,\an}(-1))$. Finally, for (3),  the morphism of groups $X^{\ob{HT},\ob{ext}}\to \bb{Z}_{\Betti,X}$ induces multiplication maps  on fibers
\[
X^{\HT,(n)} \times_X X^{\HT, (m)}  \to X^{\HT,(n+m)}
\]
that when passing to the category of kernels on $X$ give rise to a $[BT^{\dagger}_X(1)]_!$-bilinear map $X^{\HT,(n)} \otimes  X^{\HT, (m)}\to X^{\HT,(n+m)}$. Passing to module categories and Cartier duals this produces the natural $\ob{D}(T^{*,\an}_X(-1))$-linear map 
\[
\ob{D}(\s{S}_X)^{\ob{wt}=n}\otimes_{\ob{D}(T^{*,\an}_X(-1))} \ob{D}(\s{S}_X)^{\ob{wt}=m}\to \ob{D}(\s{S}_X)^{\ob{wt}=n+m}.
\] 
To see that this map is an equivalence, we can argue locally in the $!$-topology on $T^{*,\an}_X(-1)$, and by pullying back along the Simpson gerbe $\s{S}_X\to T^{*,\an}_X(-1)$, assume that the $\bb{G}_m$-gerbe  $\s{S}_X\times_{T^{*,\an}_X(-1)} \s{S}_X = B\bb{G}_m\times \s{S}_X$ is split. In this situation, the weight decomposition becomes the base change of the natural weight decomposition of $\ob{D}(B\bb{G}_m)$ which is clearly multiplicative.  
\end{proof}

\begin{remark}\label{RemarkTwistsHTStack}
Let  $X^{\HT,\ob{ext}}=\bigsqcup_{n\in \Z} X^{\HT,(n)}$ be the extended Hodge-Tate stack. We have that $X^{\HT,(0)}\cong BT^{\dagger}_X(1)$ while for $n\neq 0$  the stack $X^{\HT,(n)}$ is isomorphic to the Hodge-Tate stack as $BT^{\dagger}_X(1)$-torsors, namely, the stack $X^{\HT,(n)}$ corresponds to the class $n\cdot \eta_{\HT}\in H^{2}(X, T^{\dagger}_X(1))$, and $n\in \Q_p^{\times}$.  
\end{remark}

\bibliographystyle{alpha}
\bibliography{BiblioGerbes}

\end{document}